\documentclass[12pt]{amsart}
\usepackage{amssymb,amsfonts,amsmath,amsopn,amstext,amscd,color,latexsym,,mathrsfs,stackrel,url}
\usepackage[mathscr]{euscript}
\usepackage{mathabx}
\usepackage[utf8]{inputenc}
\usepackage{hyperref}

\usepackage[all]{xy}

\usepackage[active]{srcltx}

\theoremstyle{remark}

\newtheorem{para}{\bf}[subsection]

\newtheorem{aux}[para]{\it Auxiliary result}

\theoremstyle{definition}

\newtheorem{dfn}[para]{Definition}

\theoremstyle{plain}

\newtheorem{thm}[para]{Theorem}

\newtheorem{lemma}[para]{Lemma}
\newtheorem{cor}[para]{Corollary}
\newtheorem{prop}[para]{Proposition}

\newenvironment{numequation}{\addtocounter{para}{1}
\begin{equation}}{\end{equation}}

\newcommand{\vpi}{{\varpi}}

\newcommand{\Ga}{\Gamma}

\newcommand{\bbG}{{\mathbb G}}

\newcommand{\bbN}{{\mathbb N}}

\newcommand{\bbQ}{{\mathbb Q}}

\newcommand{\fro}{{\mathfrak o}}

\newcommand{\frS}{{\mathfrak S}}

\newcommand{\frU}{{\mathfrak U}}
\newcommand{\frV}{{\mathfrak V}}
\newcommand{\frW}{{\mathfrak W}}
\newcommand{\frX}{{\mathfrak X}}

\newcommand{\frZ}{{\mathfrak Z}}

\newcommand{\cA}{{\mathcal A}}
\newcommand{\cB}{{\mathcal B}}
\newcommand{\cC}{{\mathcal C}}

\newcommand{\cE}{{\mathcal E}}
\newcommand{\cF}{{\mathcal F}}
\newcommand{\cG}{{\mathcal G}}
\newcommand{\cH}{{\mathcal H}}
\newcommand{\cI}{{\mathcal I}}

\newcommand{\cM}{{\mathcal M}}
\newcommand{\cN}{{\mathcal N}}
\newcommand{\cO}{{\mathcal O}}
\newcommand{\cP}{{\mathcal P}}

\newcommand{\cV}{{\mathcal V}}

\newcommand{\sB}{{\mathscr B}}

\newcommand{\sD}{{\mathscr D}}

\newcommand{\sI}{{\mathscr I}}

\newcommand{\sK}{{\mathscr K}}

\newcommand{\sM}{{\mathscr M}}
\newcommand{\sN}{{\mathscr N}}

\newcommand{\sT}{{\mathscr T}}

\newcommand{\tpi}{\tilde{\pi}}
\newcommand{\tpr}{\tilde{\pr}}
\newcommand{\tcF}{\tilde{{\mathcal F}}}

\newcommand{\wD}{\widehat{D}}

\newcommand{\Q}{{\mathbb Q}}
\newcommand{\Z}{{\mathbb Z}}
\newcommand{\Ne}{{\mathbb N}}

\newcommand{\Qp}{{\mathbb Q_p}}

\newcommand{\Zp}{{\mathbb Z_p}}

\newcommand{\ocI}{\overline{\cI}}

\newcommand{\what}{\widehat}
\newcommand{\ot}{\otimes}
\newcommand{\der}{\partial}
\newcommand{\lan}{\langle}
\newcommand{\ran}{\rangle}

\newcommand{\ovP}{\overline{P}}

\newcommand{\Pf}{{\it Proof. }}

\newcommand{\Spec}{{\rm Spec}}

\newcommand{\Spf}{{\rm Spf}}
\newcommand{\Spm}{{\rm Spf}}

\newcommand{\lra}{\longrightarrow}
\newcommand{\tra}{\twoheadrightarrow}
\newcommand{\sta}{\stackrel}

\newcommand{\hra}{\hookrightarrow}
\newcommand{\im}{{\rm im}}

\newcommand{\Mod}{{\rm Mod}}

\newcommand{\ra}{\rightarrow}
\renewcommand{\sp}{{\rm sp}}
\newcommand{\sub}{\subset}

\newcommand{\hsD}{{\widehat{\sD}}} 
\newcommand{\hD}{{\widehat{D}}} 
\newcommand{\hE}{{\widehat{E}}} 

\newcommand{\Sym}{{\rm Sym}}

\newcommand{\car}{\stackrel{\simeq}{\longrightarrow}}

\newcommand{\crofrac}[2]{\genfrac{\langle}{\rangle}{0pt}{}{#1}{#2}}

\newcommand{\parfrac}[2]{\genfrac{(}{)}{0pt}{}{#1}{#2}}

\newcommand{\accfrac}[2]{\genfrac{\{}{\}}{0pt}{}{#1}{#2}}
\newcommand{\uzero}{\underline{0}}
\newcommand{\uone}{\underline{1}}

\newcommand{\uder}{\underline{\partial}}

\newcommand{\lam}{\lambda}
\newcommand{\uxi}{\underline{\xi}}

\newcommand{\umu}{\underline{\mu}}
\newcommand{\unu}{\underline{\nu}}

\newcommand{\pr}{{\rm pr}}

\begin{document}

\title{arithmetic structures for differential operators on formal schemes}
\author{Christine Huyghe}
\address{IRMA, Universit\'e de Strasbourg, 7 rue Ren\'e Descartes, 67084 Strasbourg cedex, France}
\email{huyghe@math.unistra.fr}
\author{Tobias Schmidt}
\address{IRMAR, Universit\'e de Rennes 1, Campus Beaulieu, 35042 Rennes cedex, France}
\email{Tobias.Schmidt@univ-rennes1.fr}
\author{Matthias Strauch}
\address{Indiana University, Department of Mathematics, Rawles Hall, Bloomington, IN 47405, U.S.A.}
\email{mstrauch@indiana.edu}


\begin{abstract} Let $\fro$ be a complete discrete valuation ring of mixed characteristic $(0,p)$ and $\frX_0$ a smooth formal $\fro$-scheme. Let $\frX\ra \frX_0$ be an admissible blow-up. In the first part, we introduce sheaves of differential operators $\sD^\dagger_{\frX,k}$ on $\frX$, for every sufficiently large positive integer $k$, generalizing
Berthelot's arithmetic differential operators on the smooth formal scheme $\frX_0$. The coherence of these sheaves and several other basic properties are proven. In the second part, we study the projective limit sheaf $\sD_{\frX,\infty} = \varprojlim_k \sD^\dagger_{\frX,k}$ and introduce its abelian category of coadmissible modules. The inductive limit of the sheaves $\sD_{\frX,\infty}$, over all admissible blow-ups $\frX$, is a sheaf $\sD_{\langle \frX_0 \rangle}$ on the Zariski-Riemann space of $\frX_0$, which gives rise to an abelian category of coadmissible modules. Analogues of Theorems A and B  are shown to hold in each of these settings, i.e., for $\sD^\dagger_{\frX,k}$, $\sD_{\frX,\infty}$, and $\sD_{\langle \frX_0 \rangle}$.
\end{abstract}

\maketitle

\tableofcontents

\section{Introduction}
Let $\fro$ be a complete discrete valuation ring of mixed
 characteristic $(0,p)$, with uniformizer $\vpi$ and fraction field $L$.
 In \cite{PSS4} some of us (together
 with D. Patel) have introduced sheaves of arithmetic
differential operators $\sD^{\dagger}_{n,k}$ on certain semistable formal models $\frX_n$ of the rigid analytic projective line over
$L$ (for positive integers $k \ge n$). A key result of \cite{PSS4} is that $\frX_n$ is $\sD^{\dagger}_{n,k}$-affine.
When $n=0$, the formal model $\frX_0$ is formally smooth over
$\Spf(\fro)$, and the sheaf $\sD^{\dagger}_{0,0}$ equals
Berthelot's sheaf of arithmetic differential operators, as defined
in \cite{BerthelotDI}, and $\sD^{\dagger}_{0,0}$-affinity was known
before by a result of one of us \cite{NootHuyghe09}.

\vskip8pt

In this paper we generalize the construction of \cite{PSS4} and define and study sheaves of arithmetic differential operators on arbitrary admissible formal blow-ups of an arbitrary given smooth formal scheme $\frX_0$ over $\fro$.

\vskip8pt

At the moment, the main application of this generalization is the
localization theorem of \cite{HPSS}: in this context $\frX_0$ is
the smooth model of the flag variety of a connected split reductive
group $\bbG$ over $L$, and the main result of \cite{HPSS}
establishes then an anti-equivalence between the category of
admissible locally analytic $\bbG(L)$-representations (with trivial
character) \cite{ST03} and the category of so-called coadmissible
equivariant arithmetic $\sD$-modules on the system of all formal models of the rigid analytic flag variety of $\bbG$.

\vskip8pt

In the following we describe the construction and the main results of this article. Let $\frX_0$ be a smooth formal scheme over $\fro$ and let $\sD^{(m)}_{\frX_0}$ be Berthelot's sheaf of arithmetic differential operators of level $m$ on $\frX_0$
as defined in \cite{BerthelotDI}. For any number $k\geq 0$, we have the subalgebra $\sD^{(k,m)}_{\frX_0}$
consisting of those differential operators which are generated, locally where we have coordinates $x_1,...,x_M$ and corresponding derivations
$\der_1,...,\der_M$, by operators of the form

$$\vpi^{k|\unu|}\uder^{\lan \unu \ran_{(m)}}  = \vpi^{k(\nu_1 + \ldots + \nu_M)} \prod_{l=1}^M \partial_l^{\lan \nu_l \ran_{(m)}} \;, \;\; \mbox{where } \; \partial_l^{\lan \nu_l \ran_{(m)}} = \frac{\lfloor \frac{\nu_l}{p^m} \rfloor !}{\nu_l !} \partial_l^{\nu_l} \;.$$

\vskip8pt

Given an admissible blow-up $\pr: \frX\ra \frX_0$, we let $k_\frX$ be the minimal $k$ such that $\varpi^k\cO_\frX\subset\cI$
for any coherent ideal sheaf $\cI$ on $\frX_0$ whose blow-up is $\frX$. Our first basic result, cf. \ref{ringst1}, says that

$$\sD^{(k,m)}_{\frX}:=\pr^* \sD^{(k,m)}_{\frX_0}=\cO_{\frX}\otimes_{\pr^{-1}\cO_{\frX_0}} \pr^{-1}\sD^{(k,m)}_{\frX_0}$$

\vskip8pt

is naturally a sheaf of rings on $\frX$ whenever  $k\geq k_\frX$. We define

$$
 \hsD^{(k,m)}_{\frX}=\varprojlim_i \sD^{(k,m)}_{\frX}/\vpi^i \hskip5pt \textrm{ and }\hskip5pt
\sD^{\dagger}_{\frX,k}=\varinjlim_m \hsD^{(k,m)}_{\frX}\ot \bbQ \;,
$$

\vskip8pt

and call these sheaves {\it arithmetic differential operators of congruence level\footnote{The terminology is motivated by the relation to congruence subgroups in reductive groups in the case of formal models of flag varieties, cf. \cite{HPSS}.} $k$} on
$\frX$.

\vskip8pt

The structure theory of these differential
operators goes largely parallel to the classical smooth
setting (when $\frX = \frX_0$ and $k=0$), as developed by
Berthelot \cite{BerthelotDI}. In particular, the sheaves
$\sD^{(k,m)}_{\frX}, \hsD^{(k,m)}_{\frX}$ and
$\sD^{\dagger}_{\frX,k}$ are sheaves of coherent rings on
$\frX$. We then show that Cartan's theorems A and B hold for
the sheaf $\sD^{\dagger}_{\frX,k}$, when restricted to an affine open
subscheme $\frU$ of $\frX$, cf. \ref{easy_thmAB2}. This means that
the global sections functor $\Gamma(\frU,-)$ furnishes an
equivalence of categories between the coherent modules over $\sD^{\dagger}_{\frU,k}$ and over $\Gamma(\frU,\sD^{\dagger}_{\frX,k})$, respectively.
A key result (the 'invariance theorem') shows that in case of a morphism $\frX' \ra \frX$ between admissible blow-ups of $\frX_0$, the categories of coherent modules over $\sD^\dagger_{\frX',k}$ and over $\sD^\dagger_{\frX,k}$, respectively, are naturally equivalent, cf. \ref{prop-exactdirectimage}.
As a consequence, we obtain global versions of theorem A and B on the whole blow-up $\frX$ provided the base $\frX_0$ is affine, cf. \ref{thmA}.
\vskip8pt

Our next objective is to pass to the projective limit $$\sD_{\frX,\infty} = \varprojlim_k \sD^\dagger_{\frX,k}$$ and to define the category $\cC_{\frX}$ of
 {\it coadmissible $\sD_{\frX,\infty}$.modules}. This is a full abelian subcategory of the category of all $\sD_{\frX,\infty}$-modules. Its construction relies on the fact that the ring of local sections  $\Gamma(\frV,\sD_{\frX,\infty})$ over an open affine $\frV\subseteq\frX$ is a Fr\'echet-Stein algebra. Our terminology (as well as the general philosophy behind these constructions) goes back to the fundamental work of P. Schneider and J. Teitelbaum who introduced the concept of a Fr\'echet-Stein algebra and defined and studied the category of coadmissible modules over such a ring, cf. \cite{ST03}. In fact, we show that the global sections functor $\Gamma(\frV,-)$ induces an equivalence of categories
between $\cC_{\frV}$ and the category of coadmissible  $\Gamma(\frV,\sD_{\frX,\infty})$-modules, cf. \ref{thm_A_for_frX_infty}.
Moreover, any coadmissible $\sD_{\frX,\infty}$-module has vanishing higher cohomolgy. These results should be regarded as Cartan's theorems A and B in this setting, cf. \ref{thm_A_for_frX_infty} and \ref{thm_B_for_frX_infty}.

\vskip8pt

Finally we consider the Zariski-Riemann space of $\frX_0$, i.e., the projective limit $\lan \frX_0 \ran = \varprojlim \frX$ of all admissible formal blow-ups $\frX \ra \frX_0$, cf. \cite{BoschLectures}. One can then form the inductive limit $$\sD_{\langle \frX_0 \rangle} = \varinjlim_\frX \sp_\frX^{-1}\sD_{\frX,\infty} \;,$$ where $\sp_\frX: \lan \frX_0 \ran \ra \frX$ is the projection map. This is a sheaf of rings on $\lan \frX_0 \ran$ and we define the abelian category of coadmissible $\sD_{\langle \frX_0 \rangle}$-modules. We establish analogues of Theorems A and B in this setting, cf. \ref{thm_AB_for_frX_infty}.

\vskip8pt

After we developed much of the theory presented here (which began with \cite{PSS2,PSS4}) we became aware of the article \cite{Shiho15} by A. Shiho, where he introduces sheaves of $p$-adic differential operators of negative level $-m$, as they are called there. These are closely related to the sheaves considered here, where the congruence level $k$ corresponds to Shiho's level $-m$. We are currently investigating the implications that Shiho's work has in our context.

\vskip8pt

We also want to mention that K. Ardakov and S. Wadsley are developing a theory of $D$-modules on general smooth rigid-analytic
spaces, cf. \cite{ArdakovICM, AWDcapII, AWDcapI}. In their work they consider deformations of crystalline
differential operators (as in \cite{AW}), whereas we take as a starting point deformations of Berthelot's arithmetic differential operators.
Though we have not carried this out in the present paper, it is not too difficult to see that the category of coadmissible
$\sD_{\langle \frX_0 \rangle}$-modules as defined here, when pulled back to the site of the rigid-analytic space of classical points, coincides with the corresponding category studied in \cite{ArdakovICM, AWDcapII,
AWDcapI}.
\vskip8pt

{\it Acknowledgments.} {T.S. would like to acknowledge support of the Heisenberg programme of Deutsche Forschungsgemeinschaft (SCHM 3062/1-1). M.S. is grateful for the hospitality and support of the following institutions where work on this project has been accomplished: Institut de Recherche Math\'ematique Avanc\'ee (IRMA) of the University
of Strasbourg, Centre Henri Lebesgue, Institut de
Recherche Math\'ematique de Rennes (IRMAR).

\vskip8pt

{\it Notations and Conventions.} \label{notations}
We denote by $\vpi$ a uniformizer of the complete discrete valuation ring $\fro$, and we let
 $|.|_p$ be the absolute value on $L=Frac(\fro)$ which is normalized by $|p|_p = p^{-1}$. Throughout this paper $\frS = \Spf(\fro)$.
A formal scheme $\frX$ over $\frS$ such that $\vpi \cO_{\frX}$ is an ideal of definition and which is locally noetherian is called a {\it $\frS$-formal scheme}.
If the $\frS$-formal scheme $\frX$ is smooth over $\frS$ we denote by $\sT_{\frX}$ its relative tangent sheaf. A coherent sheaf of ideals $\cI\subset \cO_{\frX}$ is called {\it open} if for any open $\frU \sub \frX$ the restriction of $\cI$ to $\frU$ contains  $\varpi^k \cO_{\frU}$ (for some $k \in \bbN$ depending on $\frU$). A scheme which arises
from blowing up an open ideal sheaf on $\frX$ will be called {\it an admissible blow-up} of $\frX$. For an integer $i \geq 0$ we also denote $X_i$ the scheme

$$X_i= \frX\times_{\frS}\Spec \left(\fro/\varpi^{i+1}\fro\right) \;,$$

\vskip8pt

where the Cartesian product is taken in the category of locally ringed spaces.
Without further mentioning, all occuring modules will be left modules.
We let $\bbN=\{0,1,2,...\}$ be the set of non-negative integers.

\vskip12pt

\section{Arithmetic differential operators with congruence level}
\label{new_sheaves}
Let $\frX_0$ be a smooth and separated $\frS$-formal scheme, and let
$$X_{0,i}=\frX_0\times_{\frS}\Spec \left(\fro/\varpi^{i+1}\fro\right) \;.$$
Let us write $\cI_{\Delta}$ for the diagonal ideal of the closed immersion of formal schemes
$\frX_0 \hra \frX_0 \times \frX_0$ and $\cI_{\Delta,i}$ for the diagonal ideal : $X_{0,i}\hra X_{0,i} \times X_{0,i}$.
We also introduce $\frX_{0,\Q}$, the generic fiber of $\frX_0$. It is a rigid analytic space over $L$. We write $\ocI_{\Delta}$ for
 the diagonal ideal of the closed immersion of analytic spaces $\frX_{0,\Q} \hra \frX_{0,\Q}\times \frX_{0,\Q}$. We have
specialization maps $\frX_{0,\Q}\ra \frX_0$ and $\frX_{0,\Q}\times \frX_{0,\Q}\ra \frX_0 \times
\frX_0$ which we denote by $\sp$. There is a canonical isomorphism $\sp^*(\cI_{\Delta})\simeq \ocI_{\Delta} $.
Finally, the relative tangent sheaf $\sT_{\frX_0}$ is a locally free $\cO_{\frX_0}$-module of finite rank
equal to the relative dimension $M$ of $\frX_0$ over $\frS$.

\subsection{The main construction}\label{main_construction}
%
\subsubsection{Definitions over $\frX_0$}\label{def_X0}
 The sheaf of relative differential operators of $X_{0,i}$ over $\fro/\varpi^{i+1}$, as introduced in \cite[16.8]{EGA_IV_4}, will be denoted by $\sD_{X_{0,i}}$. It naturally acts on the structure sheaf $\cO_{X_{0,i}}$.
Its subsheaf $\sD_{X_{0,i}}^n$ of differential operators of order $\leq n$ is defined
by
\begin{numequation}\label{defDi} \sD_{X_{0,i}}^n=\cH om_{\cO_{X_{0,i}}}(\cO_{X_{0,i}\times
X_{0,i}}/\cI_{\Delta,i}^{n+1},\cO_{X_{0,i}}) \;.\end{numequation}
It is a sheaf of locally free $\cO_{X_{0,i}}$-modules of finite type and we have $\sD_{X_{0,i}}=\varinjlim_{n}  \sD_{X_{0,i}}^n$.

For fixed $n$, the projective limit $\varprojlim_i \sD_{X_{0,i}}^n$ is a locally free $\cO_{\frX_0}$-module of finite type. Taking the inductive limit produces a sheaf of rings
$$\sD_{\frX_0}=\varinjlim_n \left( \varprojlim_i \sD_{X_{0,i}}^n\right)$$
on $\frX_0$. It naturally acts on the structure sheaf
 $\cO_{\frX_0}$ and can be described in local coordinates as follows.
Let $\frU_{0}\subseteq \frX_0$ be an open affine endowed with \'etale coordinates
$x_1,\ldots,x_M$ and corresponding set of derivations $\der_1,\ldots, \der_M$. Write $\der_l^{[\nu]}\in \sD_{\frU_{0}}$ for the differential operator defined by $\nu!\der_l^{[\nu]}=\der_l^{\nu}$, and put $\unu=(\nu_1,\ldots,\nu_M)\in\bbN^M$, $\uder^{[\unu ]}=\prod_{l=1}^M \der_{l}^{[\nu_l]}$. One has the following description, involving finite sums,
\begin{gather*}\sD_{\frX_0}(\frU_{0})=\left\{\sum^{<\infty}_{\unu}a_{\unu}\uder^{[ \unu ]}\,|\, a_{\unu}\in
\cO_{\frX_0}(\frU_0)\right\} \;.\end{gather*}

Since $\frX_{0}$ is a smooth $\frS$-formal scheme, one also has the usual sheaves of arithmetic differential operators defined by
Berthelot in \cite{BerthelotDI}. In particular, for a fixed non-negative $m$,
$\sD^{(m)}_{\frX_{0}}$ will denote the sheaf of differential operators over $\frX_{0}$ of level $m$. Taking
$\frU_0$ to be endowed with local coordinates $x_1,\ldots,x_M$ as before,
one introduces the following differential operators
\begin{numequation}\label{notbracket}
\der_{l}^{\lan \nu \ran}=q_{\nu}!\der_{l}^{[\nu]} \;,
\end{numequation}
 where $q_{\nu}$ denotes the quotient of the euclidean division
of $\nu$ by $p^m$.
For $\unu=(\nu_1,\ldots,\nu_M)\in\bbN^M$, we also define
$\uder^{\lan \unu \ran}=\prod_{l=1}^M \der_{l}^{\lan \nu_l \ran}$.
Restricted to $\frU_{0}$, the sheaf $\sD^{(m)}_{\frX_{0}}$ is a sheaf of free $\cO_{\frU_{0}}$-modules
with basis given by the elements $\uder^{\lan \unu \ran}$. Berthelot introduces also the following
sheaves over $\frX_0$
$$ \hsD_{\frX_0}^{(m)}=\varprojlim_i \sD^{(m)}_{X_{0}}/\vpi^i \hskip5pt \textrm{ and }\hskip5pt
\sD^{\dagger}_{\frX_0}=\varinjlim_m \hsD_{\frX_0}^{(m)}\ot \bbQ \;.$$\\

Let $k$ be a non-negative integer, called a {\it congruence level}. We define subsheaves $\sD^{(k,m)}_{\frX_0}$
of subalgebras of the previous sheaves $\sD^{(m)}_{\frX_0}$ in the following way.
 Take $\frU_0$ endowed with local coordinates $x_1,\ldots,x_M$ as before. Then the sheaf $\sD^{(k,m)}_{\frX_0}$
is free over $\frU_0$ as a sheaf of $\cO_{\frU_0}$-modules with a basis given by the elements $\vpi^{k|\unu|}\uder^{\lan \unu \ran}$.
In particular, one has
\begin{numequation}\label{coord1}\sD^{(k,m)}_{\frX_0}(\frU_{0})=
\left\{\sum^{<\infty}_{\unu}\vpi^{k|\unu|}a_{\unu}\uder^{\lan \unu \ran}\,|\, a_{\unu}\in
\cO_{\frX_0}(\frU_0)\right\} \;.\end{numequation}
It is easy to check that these sheaves define a subsheaf of $\sD^{(m)}_{\frX_0}$ and we call them
{\it level $m$ arithmetic differential operators of congruence level $k$ on $\frX_0$}. We then define as before
$$  \sD^{(k,m)}_{X_{0,i}}=\sD^{(k,m)}_{\frX_{0}}/\varpi^{i+1}, \;
\hsD_{\frX_0}^{(k,m)}=\varprojlim_i \sD^{(k,m)}_{X_{0,i}} \hskip5pt \textrm{ and }\hskip5pt
\sD^{\dagger}_{\frX_0,k}=\varinjlim_m \hsD_{\frX_0}^{(k,m)}\ot \bbQ \;.$$

Of course, for $k=0$ one recovers
the sheaves of Berthelot $\sD^{(0,m)}_{\frX_{0}}=\sD^{(m)}_{\frX_{0}}$. Note also, by definition,
the sheaf $\sD^{\dagger}_{\frX_0,k}$ is a sheaf of $\Q$-algebras.

\vskip5pt
 We have the following
description over an affine open $\frU_0$ of $\frX_0$
endowed with coordinates $x_1,\ldots,x_M$,

\begin{gather*} \label{ddagkX0}
\sD^{\dagger}_{\frX_0,k}(\frU_0)=\left\{\sum_{\unu}\vpi^{k|\unu|}a_{\unu}\uder^{[ \unu ]}\,|\, a_{\unu}\in
\cO_{\frX_0,\Q}(\frU_0),
\textrm{ and } \exists C>0, \eta<1 \, | \, |a_{\unu}|< C \eta^{|\unu|} \right\},
\end{gather*}
where $|\cdot|$ is any Banach norm on the affinoid algebra $\cO_{\frX_0,\Q}(\frU_0)$.

\vskip8pt

\subsubsection{Definitions over $\frX_{0,\Q}$}\label{def_X0K}

We refer to~\cite[1.1.1]{ChiarLeStum} for a basic discussion of the sheaf of algebraic differential operators over
a smooth rigid analytic space such as $\frX_{0,\Q}$. It is defined in the following way, analogously
to definition given in \cite[16.8]{EGA_IV_4} which we have recalled in~\ref{def_X0}. As before, $\ocI_{\Delta}$ denotes the diagonal ideal of the immersion
$\frX_{0,\Q} \hra \frX_{0,\Q}\times \frX_{0,\Q}$. One puts
\begin{numequation}\label{defDrig}
 \sD_{\frX_{0,\Q}}^n=
\cH om_{\cO_{\frX_{0,\Q}}}(\cO_{\frX_{0,\Q}\times \frX_{0,\Q}}/\ocI_{\Delta}^{n+1},\cO_{\frX_{0,\Q}}),
\end{numequation}
which is a sheaf of locally free $\cO_{\frX_{0,\Q}}$-modules of finite type and $ \sD_{\frX_{0,\Q}}=\varinjlim_n \sD_{\frX_{0,\Q}}^n$. The latter is a sheaf of
rings acting naturally on the structure sheaf $\cO_{\frX_{0,\Q}}$ of the rigid analyic space $\frX_{0,\Q}$.

\vskip8pt

 Suppose now that  $$\pr:\frX\rightarrow \frX_0$$ is an admissible blow-up of the formal scheme $\frX_0$
defined by a coherent sheaf of open ideals
$\cI$ of $\cO_{\frX_0}$ containing $\vpi^k$. We remark that the ideal $\cI$ is not determined by the blow-up
$\pr:\frX\rightarrow \frX_0$, i.e., different open ideal sheaves can give rise to isomorphic blow-ups. (For example, the blow-ups
defined by $\cI$ and by $\varpi^n\cI$ are isomorphic as formal schemes over $\frX_0$. The same holds for the ideals $\cI$ and $\cI^r$.)
In the sequel we denote by $k_{\cI}$ the minimal $k$ such that $\vpi^k\in \cI$ and put
\begin{numequation}\label{defkX} k_{\frX}=\min\{k_{\cI}\, | \text{ $\frX$ is~the~blowing-up~ of~$\cI$ on $\frX_0$} \}.
\end{numequation}
Let us define now for $k \geq k_{\frX}$ the $\cO_{\frX}$-module
$$ \sD_{\frX}^{(k,m)}=\pr^* \sD_{\frX_0}^{(k,m)} \;.$$

We have the following commutative diagram of ringed spaces, involving the specialization maps
$\frX_{0,\Q} \ra \frX_0$ and $\frX_\Q \ra \frX$, which we denote both by $\sp$, and the isomorphism
$\tpr: \frX_\Q \stackrel{\simeq}{\lra}\frX_{0,\Q}$ induced by the morphism $\pr$ on generic fibres:

$$\xymatrix {\frX_\Q \ar@{->}[r]^{\sim}_{\tpr}\ar@{->}[d]^{\sp}  & \frX_{0,\Q}\ar@{->}[d]^{\sp}\\
          \frX \ar@{->}[r]_{\pr}& \frX_0. }$$

Note that $\sD_{\frX_{\Q}}\simeq \tpr^*\sD_{\frX_{0,\Q}}$ for the corresponding sheaves of differential operators on $\frX_{\Q}$ respectively $\frX_{0,\Q}$, as follows from the definition of these sheaves.

 \begin{lemma}\label{key-lem3}\hskip0pt
  \begin{enumerate} \item There is a canonical isomorphism
              $\sD_{\frX_{0,\Q}}\simeq \sp^*\sD_{\frX_0}$ inducing an injective morphism of sheaves of rings
                         $$\sD_{\frX_0} \hra \sp_*\sD_{\frX_{0,\Q}} \;.$$
                        \item There is a canonical injective map of sheaves of abelian groups
                      $$ \sD_{\frX}^{(k,m)} \hra \sp_* \sD_{\frX_{\Q}} \;,$$ which becomes an isomorphism upon tensoring with $\Q$.
                               \end{enumerate}
\end{lemma}
\begin{proof} We have a canonical map
$ \cO_{\frX,\Q} \ra \sp_* \cO_{\frX_{\Q}} $, that is locally an isomorphism
over the formal scheme $\frX$ (resp. $\frX_0$) and is thus an isomorphism of sheaves.
 Let us begin by (i).
We have a canonical map
$\sD_{\frX_{0,\Q}}\ra \sp^*\sD_{\frX_0}$. To check that this is an isomorphism, we can work locally on $\frX_0$ and
		assume that $\frX_0$ is affine, endowed with local coordinates $x_1,\ldots,x_M$. Then, using
notations ~\ref{notbracket} we see that both sheaves are free
$\cO_{\frX_{0,\Q}}$-modules with basis $\uder^{\unu }$ and $\uder^{[ \unu ]}$ respectively. The previous map
takes $\uder^{\unu }$ to $\uder^{[ \unu ]}\unu !$ and is an isomorphism of sheaves of $\cO_{\frX_{0,\Q}}$-modules.
We obtain the second map of (i) by adjunction and its injectivity follows again using local coordinates.
As in \cite[16.8.9]{EGA_IV_4}, the ring structure on both sheaves makes use of the descriptions~\ref{defDi}, resp. ~\ref{defDrig}.
That the map $\sD_{\frX_0} \hra \sp_*\sD_{\frX_{0,\Q}}$ is compatible with ring structures comes then from
the fact that $\sp^*(\cI_{\Delta})\simeq \ocI_{\Delta}$ and
$$\sp^*(\cO_{\frX_0}\ot\cO_{\frX_0}/\cI_{\Delta}^{n+1})\simeq \cO_{\frX_{0,\Q}}\ot\cO_{\frX_{0,\Q}}/\ocI_{\Delta}^{n+1} \;.$$

Let us prove (ii). From the previous isomorphism, we get an isomorphism
$$ \sp^* \pr^* \sD_{\frX_{0}}\simeq \tpr^* \sp^* \sD_{\frX_{0}} \simeq \tpr^* \sD_{\frX_{0,\Q}} \simeq \sD_{\frX_{\Q}} \;,$$
that induces a canonical map
$ \pr^*\sD_{\frX_{0}} \ra \sp_*\sD_{\frX_{\Q}} \;.$ Composing with the map
$\sD^{(k,m)}_{\frX_{0}}\ra \sD_{\frX_{0}}$ we get a map
$$ \pr^*\sD^{(k,m)}_{\frX_{0}} \ra \sp_*\sD_{\frX_{\Q}} \;.$$ It is a local question to prove that this map is
injective. Let $\frU\subset \frX_0$ be an open affine formal scheme of $\frX_0$ such that
$\frU \subset \pr^{-1}\frU_0$, with $\frU_0$ endowed with local coordinates $x_1, \ldots,x_M$. Then these
local coordinates give local coordinates always denoted $x_1, \ldots,x_M$ over the generic fiber $\frU_\Q$ of $\frU$.
Using notations  ~\ref{notbracket}, the sheaf $\pr^*\sD^{(k,m)}_{\frU_{0}}$ is a free $\cO_{\frU}$-module with
basis the operators $\varpi^{k|\unu|}\uder^{\lan \unu \ran}$ whereas the sheaf $\sp_*\sD_{\frU_{\Q}}$ is a free
$\cO_{\frU}\ot \Q$-module with basis $\uder^{\unu}$. The map we consider takes $\varpi^{k|\unu|}\uder^{\lan \unu \ran}$
to $\varpi^{k|\unu|}q_{\unu}!/\unu! \uder^{\unu}$. Since $\frU$ is flat over $\fro$, the map $\cO_{\frU}\ra \cO_{\frU}\ot \Q$ is
injective, and this proves that the canonical map $ \pr^*\sD^{(k,m)}_{\frX_{0}} \ra \sp_*\sD_{\frX_{\Q}}$
is injective as well. The same argument shows that this map becomes an isomorphism upon tensoring with $\Q$.

\end{proof}

\vskip8pt

Following ~\cite{PSS4}, given $k \geq k_{\frX}$, we will construct a $p$-adically complete sheaf of arithmetic differential operators $\hsD^{(k,m)}_{\frX}$ over the (usually non-smooth) formal scheme $\frX$.

\subsubsection{Construction of the ring of differential operators of level $k$ over $\frX$.}
\label{subsec-const}
We first observe that the sheaf $\sp_* \sD_{\frX_{\Q}} $ acts on $\sp_*\cO_{\frX_\Q}\simeq \cO_{\frX,\Q}$. By
~\ref{key-lem3}, the sheaf $\sD_{\frX}^{(k,m)}$ is a subsheaf of $\sp_* \sD_{\frX_{\Q}}$ and
$\sD_{\frX,\Q}^{(k,m)}\simeq \sp_* \sD_{\frX_{\Q}}$. We will first
check that if $k \geq k_{\frX}$,
the action of $\sp_* \sD_{\frX_{\Q}}$ on $\cO_{\frX,\Q}$ restricts to an action of
$\sD_{\frX}^{(k,m)}$ on $\cO_{\frX}$.
 This can be checked locally on $\frX$. For this, we assume that $\frX_0=\Spf A$, where $A$ is a smooth, complete,
$\fro$-algebra, endowed with local coordinates $x_1,\ldots,x_M$. Since $\frX_0$ is smooth over $\Spf \fro$, both rings $A$ and $A/\varpi A$
are integral domains. We also introduce the differential operators $\uder^{\lan \unu \ran}$ and
$\varpi^{k|\unu|}\uder^{\lan \unu \ran}$ according to ~\ref{notbracket}, of $D^{(m)}=\Ga(\frX_0,\sD^{(m)}_{\frX_{0}})$
and
$D^{(k,m)}=\Ga(\frX_0,\sD^{(k,m)}_{\frX_{0}})$ respectively. We also denote $I=\Ga(\frX_0,\cI)$ where $V(\cI)$ is the
center of the blowing-up $\frX$.

\vskip8pt

Consider the $\bbN$-graded $A$-algebra $$B=\bigoplus_n B_n$$
where the degree $n$-part $B_n$ equals the $n$-th power $I^n$ of the ideal $I$. In particular, $I^0=A$.
This means that $$\frX=\what{{\rm Proj}}(B) \;,$$ the formal completion of ${\rm Proj}(B)$.
  The algebra $B$ is integral, as $A$ is integral.
Let $t\in B_D$ be a homogeneous element of
degree $D>0$, and let $C_t=B[1/t]_0$ be the algebra of degree $0$ elements in the homogeneous localization $B[1/t]$. Then $C_t$
is non-zero, since $B$ is integral. Put $D_+(t)=\Spf \;\what{C}_t$. These open sets form a basis
for the Zariski topology of $\frX$.

Let us observe that the algebra $B$ is a graded subalgebra of $A[T]$. Indeed, there is a graded injective ring morphism
$$\label{defphi}\xymatrix @R-25pt {B \ar@{->}[r]^(.4){\varphi}& A[T]\\
          x_n \in B_n  \ar@{->}[r]& x_n T^n.  }$$
By definition, $\varphi(t)=tT^D$.
Since localization is flat, we get from this an injective graded morphism
$B[1/t]\rightarrow A[1/t][T^{\pm 1}]$, where $A[1/t][T^{\pm 1}]$ is graded by the degree of
$T$. Because $C_t=B[1/t]_0$ is the subring of degree zero elements in $B[1/t]$, we get an injection
\begin{numequation}\label{key-inject}C_t \hra A[1/t] \;.\end{numequation}

Since $A\{1/t\}=\Ga(D(t),\cO_{\frX_0})$, and because $\sD_{\frX_0}^{(k,m)}$ acts on the structural
sheaf $\cO_{\frX_0}$, we get that $A\{1/t\}$ is a $D^{(m)}$-module and thus a $D^{(k,m)}$-module.
Moreover, as $A[1/t]$ is integral and noetherian,
it embeds into its $p$-adic completion $A\{1/t\}$.
This leads us to the following

\begin{lemma}\label{key-lem1}\hskip0pt
\begin{enumerate} \item $A[1/t]$ is a $D^{(m)}$-submodule of $A\{1/t\}$,
         \item If $k \geq k_{\frX}$, then $C_t$ is a $D^{(k,m)}$-submodule of $A[1/t]$.
\end{enumerate}
\end{lemma}

\begin{proof}
Before giving the proof, we need some notation.
Given a fixed nonnegative integer $m$ and a nonnegative integer $\nu$, one denotes as before by $q$ the quotient of the Euclidean division
of $\nu$ by $p^m$. Let $\nu\geq \nu'$ be two nonnegative integers and $\nu'':=\nu-\nu'$; then for the corresponding numbers $q,q'$ and $q''$, we define
\begin{numequation}\label{notacc}\accfrac{\nu}{\nu'} =\frac{q!}{q'!q''!} \;,\end{numequation}

which is an integer because $q \ge q' + q''$. Let us begin with the proof of (i). It is enough to prove that for each $i \leq M$, for each invertible $s\in A[1/t]$, for
each $\nu\in \Ne$, $$\der_{i}^{\lan \nu \ran}(s^{-1})\in \frac{A}{s^{\nu+1}} \;.$$
 We will prove this by induction on $\nu$, the case $\nu=0$ being straightforward.
We have
$$ \der_{i}^{\lan \nu +1 \ran}(s^{-1})=-\sum_{\mu =0}^{\nu}\accfrac{\nu+1}{\mu }s^{-1}
\der_{i}^{\lan \nu+1 -\mu  \ran}(s) \der_{i}^{\lan \mu  \ran}(s^{-1}) \;.$$
cf. \cite[(iv) of 2.2.4]{BerthelotDI}. By induction hypothesis the elements $\der_{i}^{\lan \mu  \ran}(s^{-1})$ lie in $s^{-(\nu+1)}A$, which
proves that $\der_{i}^{\lan \nu+1  \ran}(s^{-1})\in s^{-(\nu+2)}A$. By applying this to $s=t^{-n}$ for $n >0$,
we see that (i) holds.

Let us prove now (ii).
We begin the proof with an auxiliary assertion (it is here where we use the assumption $k \geq k_{\frX}$).

\vskip5pt

{\it Assertion}. Let $f\in I^{r}$, $l\in \mathbb{N}$, then $\varpi^{kl}\der_{i}^{\lan l \ran}(f)\in I^r$.

\vskip5pt

{\it Proof of the assertion.} The proof relies on the Leibniz formula \cite[2.3.4.1]{BerthelotDI}. We proceed by induction on $r$. For
$r=0$ the assertion is trivial and for $r=1$, it is true if $l \geq 1$ since $\varpi^k\in I$. For $r=1$, it is also
true if $l=0$ since $f\in I$. Let us assume that the result is true for $s\leq r$. It is enough to prove that
$$ \forall g \in I \; \forall h\in I^r \,: \; \varpi^{kl}\der_{i}^{\lan l \ran}(hg)\in I^{r+1} \;.$$
Denote $f=hg$, the Leibniz formula of loc. cit. states that
\begin{gather*} \varpi^{kl} \der_{i}^{\lan l \ran}(f)=\sum_{j=0}^l \accfrac{l}{j}\varpi^{kj}\der_{i}^{\lan j \ran}(h)
\varpi^{k(l-j)}\der_{i}^{\lan l-j \ran}(g) \;.
\end{gather*}
By induction hypothesis, for all $j \leq l$, $\varpi^{kj}\der_{i}^{\lan j \ran}(h)\in I^r$ and
$\varpi^{k(l-j)}\der_{i}^{\lan l-j \ran}(g)\in I$, which implies that $ \varpi^{kl}\der_{i}^{\lan l
\ran}(f)\in I^{r+1}$. This establishes the assertion.

\vskip5pt

After this preliminary discussion, let $d>0$. Let us first prove by induction on $\nu$ that for an arbitrary element $s\in I^d$ which becomes invertible in $A[1/t]$, one has
\begin{numequation} \vpi^{k\nu}\der_{i}^{\lan \nu \ran} \left(s^{-1}\right)\in \frac{I^{\nu d}}{s^{\nu+1}} \;.
\end{numequation}
This is true for $\nu=0$. Consider then the formula \cite[(iv) of 2.2.4]{BerthelotDI} with notation~\ref{notacc}
$$ \vpi^{k(\nu+1)}\der_{i}^{\lan \nu +1\ran}(s^{-1})=-\sum_{\mu =0}^{\nu}\accfrac{\nu+1}{\mu }s^{-1}
\vpi^{k(\nu+1-\mu )}\der_{i}^{\lan \nu+1 -\mu  \ran}(s)
\vpi^{k\mu }\der_{i}^{\lan \mu  \ran}(s^{-1}) \;.$$ By the induction hypothesis, one knows for any integer $\mu\leq \nu$,
$$\vpi^{k\mu}\der_{i}^{\lan \mu \ran}(s^{-1})\in\frac{I^{\mu d}}{s^{\mu+1}}$$
and, by our auxiliary assertion above, one knows

$$ s^{-1}\vpi^{k(\nu+1-\mu)}\der_{i}^{\lan \nu+1 -\mu \ran}(s)
\in \frac{I^d}{s} \;.$$ This implies $$s^{-1}
\vpi^{k(\nu+1-\mu )}\der_{i}^{\lan \nu+1 -\mu  \ran}(s)
\vpi^{k\mu }\der_{i}^{\lan \mu  \ran}(s^{-1}) \in \frac{I^d}{s}\frac{I^{\mu d}}{s^{\mu+1}} = \frac{I^{d(\mu +1)}}{s^{\mu+2}} = \frac{s^{\nu-\mu} I^{d(\mu +1)}}{s^{\nu+2}} \subseteq\frac{I^{(\nu+1)d}}{s^{\nu+2}} \;,$$
which proves our claim. Applying this claim to the element $s=t^d\in I^{dD}$ gives for $\mu\leq p^m$
$$\vpi^{k\mu }\der_{i}^{\lan \mu \ran}(t^{-d})\in \frac{I^{\mu dD}}{t^{d(\mu +1)}} \;.$$
Then, using again the Leibniz formula, we deduce from this and the auxiliary assertion, for a given homogeneous element
$g\in B$ of degree $dD$, i.e $\frac{g}{t^d}\in C_t$, the identity
$$\vpi^{k\nu}\der_{i}^{\lan \nu \ran}\left(\frac{g}{t^d}\right)=\sum_{\mu =0}^{\nu}
\accfrac{\nu}{\mu }\vpi^{k(\nu- \mu)}\der_{i}^{\lan \nu- \mu \ran}(g)\vpi^{k \mu}\der_{i}^{\lan \mu \ran}(t^{-d}) \;,$$
whose right-hand terms are contained in $$I^{dD}\frac{I^{ \mu dD}}{t^{d(\mu +1)}}\subset C_t \;.$$
This completes the proof of the lemma.
\end{proof}
From this we get the
\begin{cor}\label{key-coro} $\what{C_t}$ is a $D^{(k,m)}$-submodule of $A\{1/t\}_{\Q}$ if $k \geq k_{\frX}$.
\end{cor}
\begin{proof} Using previous notations, we see that $\vpi^{k|\unu|}\der_{i}^{\lan \unu \ran}$ acts continuously
 on $C_t$. We can thus extend this action by continuity to get an action on $\what{C_t}$. By construction
this action is induced by the action of $\Ga(\frX_0,\sp_*\sD_{\frX_{0,\Q}})$ on $\what{A\{1/t\}}_{\Q}$.
\end{proof}
%

%
After these local considerations, we come back now to the general situation.
\begin{cor}\label{ringst1} Let $k\geq k_{\frX}$. The sheaf $\sD^{(k,m)}_{\frX}=\pr^*\sD^{(k,m)}_{\frX_0}$
is a subsheaf of rings of the sheaf $\sp_*\sD_{\frX_\Q}$. Moreover it is locally free over $\cO_{\frX}$.
   \end{cor}
\begin{proof} The assertion is local on $\frX$ and we can assume that $\frX_0=\Spf \;A$ is affine, endowed with
local coordinates $x_1,\ldots,x_M$. Then, the sheaf $\sD^{(k,m)}_{\frX}$
is a free $\cO_{\frX}$-module generated by the operators $\varpi^{k|\unu|}\uder^{\lan \unu \ran}$
(using notations~\ref{notbracket}).
By the formula of Berthelot \cite[2.2.4]{BerthelotDI}, one has
\begin{gather*}
    \uder^{\lan \unu \ran}\cdot \uder^{\lan \unu' \ran}= \crofrac{\unu+\unu'}{\unu}\uder^{\lan \unu +\unu' \ran} \;,
\end{gather*}
where
  $$\crofrac{\unu+\unu'}{\unu} = \parfrac{\unu+\unu'}{\unu}\accfrac{\unu+\unu'}{\unu}^{-1}\, \in \Zp \;.$$
To check that $\sD^{(k,m)}_{\frX}$ is a subsheaf of rings of $\sp_*\sD_{\frX_\Q}$,
we thus only have to check, by linearity, that
 if $b$ is a section in $\cO_{\frX}(\frV)$ where $\frV=\Spf\;\what{C_t[1/h]}$, cf. \ref{key-coro}, for some non zero $t\in A$,
and non zero $h\in C_t$,
 and if $\unu\in \Ne^M$, then the element $\uder^{\lan \unu \ran}\cdot b $ lies in $\sD^{(k,m)}_{\frX}(\frV)$. By
\cite[2.2.4]{BerthelotDI}, one has
$$  \vpi^{k|\unu|}\uder^{\lan \unu \ran}\cdot b=\sum_{\unu'+\unu''=\unu}\accfrac{\unu}{\unu'}
\vpi^{k|\unu'|}\uder^{\lan \unu' \ran}(b)
\vpi^{k|\unu''| }\uder^{\lan \unu''  \ran} \;.$$
Since $\vpi^{k|\unu'|}\uder^{\lan \unu' \ran}(b)\in \cO_{\frX}({\frV})$ by~\ref{key-coro}, this proves that
$\sD^{(k,m)}_{{\frX}}$ is a subring
of $\sD_{{\frX}_\Q}$. \end{proof}

\vskip8pt

We finally define the following sheaves of differential operators over $\frX$ and $X_i$
\begin{numequation}\label{def1D}
  \sD^{(k,m)}_{X_{i}}=\sD^{(k,m)}_{\frX}/\varpi^{i+1}, \;
\hsD_{\frX}^{(k,m)}=\varprojlim_i \sD^{(k,m)}_{X_{i}} \hskip5pt \textrm{ and }\hskip5pt
\sD^{\dagger}_{\frX,k}=\varinjlim_m \hsD_{\frX}^{(k,m)}\ot \bbQ.
\end{numequation}
We have the following
local description over an affine open $\frV\subseteq \pr^{-1}(\frU_0)$ where $\frU_0$ is an affine open of $\frX_0$
endowed with coordinates $x_1,\ldots,x_M$ and derivations $\der_1,\ldots,\der_M$:

\begin{numequation} \label{ddagkX}
\sD^{\dagger}_{\frX,k}(\frV)=
\left\{\sum_{\unu}a_{\unu}\vpi^{k|\unu|}\uder^{[ \unu ]}\,|\, a_{\unu}\in \cO_{\frX,\Q}(\frV),
\textrm{ and } \exists C>0, \eta<1 \, | \, |a_{\unu}|< C \eta^{|\unu|} \right\},
\end{numequation}
where $|.|$ denotes any Banach norm on $\cO_{\frX,\Q}(\frV)$.

\subsection{First properties}
We keep here the hypothesis from the previous section. In particular, $\frX_0$ denotes a smooth formal $\frS$-scheme and $$\pr:\frX\ra\frX_0$$ denotes an admissible blow-up. For a given natural number $k\geq 0$ we let $$\sT_{\frX,k} := \vpi^k (\pr)^*(\sT_{\frX_0}) \;,$$
where $\sT_{\frX_0}$ is the relative tangent sheaf of $\frX_0$ over $\frS$.
\begin{lemma}\label{tcT_lemma1} \hskip0pt \begin{enumerate}
\item The sheaf $\sT_{\frX,k}$ is a locally free $\cO_{\frX}$-module of rank equal to the relative dimension
 of $\frX_0$ over $\frS$.
\item Suppose $\pi: \frX'\rightarrow \frX$ is a morphism over $\frX_0$ from another admissible blow-up $\pr':\frX'\ra \frX_0$. Let $k',k\geq 0 $. One has as subsheaves of $\sT_{\frX'}\otimes L$
$$\sT_{\frX',k'} = \vpi^{k'-k} {\rm \pi}^*(\sT_{\frX,k}) $$
\end{enumerate}
\end{lemma}

\begin{proof} This follows directly from the definitions. Note that
$(\pr')^*=\pi^*\circ \pr^*$ in (ii).
\end{proof}

\vskip8pt

Before stating the next proposition, let us recall that $\Sym^{(m)}(\sT_{\frX,k})$ denotes the graded level $m$ symmetric algebra generated by the sheaf $\sT_{\frX,k}$ defined in~\cite[sec. 1]{Huyghe97}. This is a graded $\cO_{\frX}$-algebra

$$\Sym^{(m)}(\sT_{\frX,k})=
\bigoplus_{d\geq 0}\Sym^{(m)}_d(\sT_{\frX,k}) \;.$$

Over some sufficiently small open affine set $\frU \subseteq \pr^{-1}(\frU_0)$ such that $\sT_{\frU_0}$ is free with basis
$\xi_1,\ldots,\xi_M$, one has the description using notation \ref{notbracket} (i.e. $\nu! \xi_{l}^{\lan \nu
\ran}=q_{\nu}!\xi_{l}^{\nu}$)

$$\Sym^{(m)}_d(\sT_{\frX,k})(\frU) = \bigoplus_{|\unu|=d}
\cO_{\frX}(\frU) \vpi^{kd}\uxi^{\lan \unu \ran} \;,$$

where the right hand side is a free $\cO_{\frX}(\frU)$-module. For the rest of this subsection we fix a number $k\geq k_\frX$ (\ref{defkX}).

\begin{prop}\label{finite_tDm} \hskip0pt
\begin{enumerate} \item  The sheaves $\sD^{(k,m)}_{\frX}$ are filtered by the sheaves of differential operators of order $\le d$, which are $\cO_{\frX}$-coherent modules, and which we denote by $\sD^{(k,m)}_{\frX,d}$.
There is a canonical isomorphism of sheaves of graded algebras

$${\rm gr}\left(\sD^{(k,m)}_{\frX}\right) \simeq \Sym^{(m)}(\sT_{\frX,k}) \;.$$

Moreover the sheaves $\sD^{(k,m)}_{X_i}$ are quasi-coherent $\cO_{X_i}$-modules.

\item There is a basis of the topology $\sB$ of $\frX$ (resp. $X_i$), consisting of open affine subsets, such that for
any $\frU \in \sB$ (resp. $U_i \in \sB$), the ring $\sD^{(k,m)}_{\frX}(\frU)$ (resp. $\sD^{(k,m)}_{X_i}(U_i)$) is two-sided noetherian.

\item For every formal affine open $\frU\subseteq \frX$ (resp. affine open $U_i\subseteq X_i$), the ring $\sD^{(k,m)}_{\frX}(\frU)$ (resp. $\sD^{(k,m)}_{X_i}(U_i)$) is two-sided noetherian.

\item The sheaves of rings $\sD^{(k,m)}_{\frX}$ (resp.  $\sD^{(k,m)}_{X_i}$) are coherent.

\item For every formal affine open $\frU\subset \frX$, the ring
$\hsD^{(k,m)}_{\frX}(\frU)$ is two-sided noetherian.
\item The sheaf of rings $\hsD^{(k,m)}_{\frX}$ is coherent.
\end{enumerate}
\end{prop}
\begin{proof}
We only do the proof of (i) to (iv) in the case of $\sD^{(k,m)}_{\frX}$, since the same proof works for the sheaf
$\sD^{(k,m)}_{X_i}$. Denote by $\sD^{(k,m)}_{\frX_0,d}$ the sheaf of differential operators of $\sD^{(k,m)}_{\frX_0}$ of
order $\le d$, and $\sD^{(k,m)}_{\frX,d}=\pr^*\sD^{(k,m)}_{\frX_0,d} \;.$
It is straightforward that we have an exact sequence of $\cO_{\frX_0}$-modules on $\frX_0$
\begin{numequation}\label{fil_gr_ex_seq_n0} 0 \lra \sD^{(k,m)}_{\frX_0,d-1} \lra \sD^{(k,m)}_{\frX_0,d} \lra
		\Sym^{(m)}_d(\sT_{\frX_0,k}) \lra 0 \;.
\end{numequation}
 Now we apply $\pr^*$ and get an exact sequence since
 $\Sym^{(m)}_d(\sT_{\frX_0,k})$ is a locally free $\cO_{\frX_0}$-module of finite rank. This gives (i).
Let $\frU$ be an affine subset of $ \pr^{-1}(\frU_0)$,
 where $\frU_0\subset \frX_0$ has some coordinates $x_1,\ldots,x_M$. One has the following description
$$\sD^{(k,m)}_{\frX}(\frU)=\left\{\sum^{<\infty}_{\unu}\vpi^{k|\unu|}a_{\unu}\uder^{\lan \unu \ran}\,|\, a_{\unu}\in \cO_{\frX}(\frU)\right\} \;.$$
Since $\frU$ is affine and the filtration steps $\sD^{(k,m)}_{\frX,d}$ are coherent $\cO_{\frX}$-modules for all $d$,
the previous exact sequences gives us the following isomorphism

$${\rm gr}\left(\sD^{(k,m)}_{\frX}\right)(\frU) \simeq \Sym^{(m)}_{\cO_{\frX}}(\sT_{\frX,k})(\frU) \;.$$

Since the latter level $m$ symmetric algebra is known to be noetherian \cite[Prop. 1.3.6]{Huyghe97},
this proves (ii).
 As $\sB$ we may take the set of open affine subsets
of $\frX$ that are contained in some $\pr^{-1}(\frU_0)$, for some open $\frU_0\subset \frX_0$ endowed with global
coordinates.
Let now $\frV, \frU \in \sB$ such that $\frV\subset \frU$. Since the sheaf $\sD^{(k,m)}_{\frX}$ is
an inductive limit of coherent $\cO_{\frX}$-modules, one has
\begin{numequation}
\cO_{\frX}(\frV)\ot_{\cO_{\frX}(\frU)}\sD_\frX^{(k,m)}(\frU)\simeq \sD_\frX^{(k,m)}(\frV) \;.
\end{numequation}
In particular $\sD_\frX^{(k,m)}(\frV)$ is flat over $\sD_\frX^{(k,m)}(\frU)$.
This remark and (ii) prove the coherence of the sheaves $\sD^{(k,m)}_{\frX}$
 exactly as in the proof \cite[3.1.3]{BerthelotDI}.

Let us now prove (iii) for 'left noetherian' (the proof of the right version is similar). Let
$\frU$ be an affine open of $\frX$, $\sD=\sD^{(k,m)}_{\frU}$, $D=\Ga(\frU,\sD)$, $A=\Ga(\frU,\cO_{\frX})$
 and $\frU =\bigcup \frU_l$ a finite cover of $\frU$ by open
$\frU_l \in \sB$. Since the sheaf $\sD$ is an inductive limit of coherent $\cO_{\frX}$-modules,
one has
\begin{numequation}\label{Dflat}
\sD=\cO_{\frU}\ot_A D \;,
\end{numequation}
and $\sD$ is a flat $D$-module.
Moreover, thanks to (ii), we know that $\sD_{\frX}(\frU_l)$ is noetherian for each $l$.
 Let
$(M_i)$ be an increasing sequence of left ideals of $D$, and consider
$$\cM_i=\sD\ot_D M_i=\cO_{\frU}\ot_A M_i\;,$$ which form an increasing sequence of sheaves of ideals
of $\sD$ by flatness of $\sD$ over $D$. The sequence $\Ga(\frU_l,\cM_i)$
is thus an increasing sequence of ideals of $\Ga(\frU_l,\sD)$,
that is stationary since this algebra is noetherian.
Since $M_i$ is an inductive limit of finite type $A$-modules,
$\cM_i$ is an inductive limit of coherent $\cO_{\frX}$-modules, thus
$$\forall l, \;\cM_{i|\frU_l}\simeq \cO_{\frU_l}\ot_{\cO_{\frX}(\frU_l)}\Ga(\frU_l,\cM_i)\,
\textrm{ and } \, \Ga(\frU,\cM_i)=M_i\;.$$
Finally we see that $\cM_{i|\frU_l}$ is stationary for each $l$. Since there are finitely many
affine open $\frU_l$, we see that the sequence $(\cM_i)$ and thus $(M_i)$ are stationary. This proves (iii).
Since $\hsD^{(k,m)}_{\frX}(\frU)$ is the $p$-adic completion of $\hsD^{(k,m)}_{\frX}(\frU)$, it is also left
and right noetherian \cite[3.2.3]{BerthelotDI}, which proves (v).

The coherence of $\hsD^{(k,m)}_{\frX}$ follows from (iii), and the fact that
$\sD^{(k,m)}_{X_i}$ is a quasi-coherent $\cO_{X_i}$-module, literally as in \cite[3.3.3]{BerthelotDI}.
\end{proof}

\vskip8pt

From these considerations, and under our initial condition $k\geq k_\frX$, we have the
following local versions of Cartan's Theorem A
and B for the restrictions of the sheaves $\sD^{(k,m)}_{X_i}$ and $\hsD^{(k,m)}_{\frX}$ to an open affine (formal)
subscheme. 

\begin{prop}\label{easy_thmAB} {\rm (Local theorem A and B for fixed $m$)} \hskip0pt \begin{enumerate}
\item Let $U_i\subset X_i$ be an open affine subscheme of $X_i$. The functor $\Ga(U_i,.)$
establishes an equivalence of categories
between coherent $\sD^{(k,m)}_{U_i}$-modules and finite type $\Ga(U_i,\sD^{(k,m)}_{X_i})$-modules. In particular,
the functor $\Ga(U_i,.)$ is exact on coherent $\sD^{(k,m)}_{U_i}$-modules. Moreover, for any coherent $\sD^{(k,m)}_{U_i}$-module $\sM$ and any $q>0$ one has $H^q(U_i,\sM) = 0$.
\item Let $\frU\subset \frX$ be an open affine formal subscheme of $\frX$. The functors
$\cM \mapsto \Ga(\frU,\cM)$ and $M \mapsto \hsD^{(k,m)}_{{\frU}}\ot M$ are quasi-inverse equivalences of categories
				between the category of coherent left $\hsD^{(k,m)}_{{\frU}}$-modules and the category of left
				$\Ga(\frU,\hsD^{(k,m)}_{{\frU}})$-modules of finite type.
\item Let $\frU$ be as in (ii). The functor $\Ga(\frU,.)$ is exact on coherent $\hsD^{(k,m)}_{\frU}$-modules. Moreover, for any coherent $\hsD^{(k,m)}_\frU$-module $\sM$ and any $q>0$ one has $H^q(\frU,\sM) = 0$.
\end{enumerate}
\end{prop}
Note that (ii) of the proposition remains true for the sheaf $\hsD^{(k,m)}_{\frU,\bbQ}$ and coherent modules
over this sheaf by \cite[3.4.6]{BerthelotDI}.

\begin{proof} For the convenience of the reader, we start by recalling the following result

\begin{aux}\label{aux1} (cf. \cite[Prop. 3.1.3]{BerthelotDI}) Let $X$ be a scheme, ${\sD} $ be a sheaf of rings over $X$ such that, for all affine open $U \subset X$, $\Ga(U,\sD)$ is a noetherian ring. We fix an homomorphism
$\cO_X \ra \sD$ such that the left multiplication by the sections of $\cO_X$ induces a structure of
$\cO_X$-coherent ring over $\sD$.
		\begin{enumerate} \item The sheaf $\sD$ is a left coherent sheaf of rings.
                         \item A left $\sD$-module $\cM$ is coherent if and only if it is a quasi-coherent
                $\cO_X$-module and, for all affine open $U$ of an affine cover of $X$, $\Ga(U,\cM)$ is a
                    left $\Ga(U,\sD)$-module of finite type.
                  \item Assume that $X$ is affine and let $D=\Ga(X,\sD)$. The functors
$\cM \mapsto \Ga(X,\cM)$ and $M \mapsto \sD\ot M$ are quasi-inverse equivalences of categories between the
category of coherent left $\sD$-modules and the category of left $D$-modules of finite type. \qed
		\end{enumerate}
\end{aux}

Consider now the following situation, compare \cite[3.3.3]{BerthelotDI}. Let $\frX'$ be  an $\frS$-formal scheme and let
$\sD$ be a sheaf of rings over $\frX'$, endowed with a homomorphism $\cO_{\frX'}\ra \sD$, $\sD_i=\sD/p^{i+1}$,
$\hsD=\varprojlim_i \sD_i$. In addition, assume the following conditions (Berthelot's conditions)
		\begin{enumerate}\item[(1)]
				As an $\cO_{\frX'}$-module, $\sD$ is filtered inductive limit of a family of $\cO_{\frX'}$-module
                 $\sD_{\lambda}$ such that $\sD_{\lambda}/p^i \sD_{\lambda}$ are $\cO_{X'_i}$-quasi-coherent and
				 $\sD_{\lambda}\simeq \varprojlim_i \sD_{\lambda}/p^i \sD_{\lambda}$.
				\item[(2)] For every open set $\frU \subset \frX'$, the ring $\Ga(\frU,\sD)$ is left noetherian.
		\end{enumerate}
\begin{aux}\label{aux2} With the previous hypotheses, suppose that
$\frX'$ is affine, and let $\hD=\Ga(\frX',\hsD)$. Then $\hD$ is left noetherian.
		If $M$ is a $\hD$-module one defines
a $\hsD$-module
$$ M^{\Delta}=\varprojlim_i \sD_i \ot_{\cO_{X'_i}}M/p^{i+1}M \;.$$
 For a $\hsD$-module $\cM$, the following statements are equivalent
\begin{enumerate} \item For all $i$, the $\sD_i$-module $\cM/p^{i+1}\cM$ is coherent and the canonical
homomorphism $\cM\ra \varprojlim_i \cM/p^{i+1}\cM$ is an isomorphism.
\item There exists an isomorphism $\cM \simeq \varprojlim_i \cM_i$, where $(\cM_i)$ is a projective system
of coherent $\sD_i$-modules, such that the transition morphisms factorize via isomorphisms
$\cM_i/p^i \cM_i \simeq \cM_{i-1}$.
\item There exists a finite type $\hD$-module $M$ and an isomorphism $\cM \simeq M^{\Delta} \;.$
\item $\cM$ is a coherent $\hsD$-module.
\end{enumerate}
\end{aux}
\begin{proof} The ring $\hD$ is noetherian by \cite[Prop. 3.3.4]{BerthelotDI} and the other results
		come from \cite[Prop. 3.3.9]{BerthelotDI}.
\end{proof}

Under the same hypotheses, Berthelot proves in addition the following.
\begin{aux}\label{aux3} (cf. \cite[Prop. 3.3.10/11]{BerthelotDI}) With the previous hypotheses, suppose that
$\frX'$ is affine. The functors
$\Ga(\frX',\cdot)$ and $M \mapsto M^{\Delta}$ are equivalences between the category of
coherent $\hsD$-modules and the category of finite type $\hD$-modules. If
$\cM$ is a coherent $\hsD$-module, then $\forall q\geq 1$, $H^q(\frX',\cM)=0$. \qed
\end{aux}

Now it is clear that part (i) of our proposition \ref{easy_thmAB} follows from auxiliary result \ref{aux1},
since by (iii) of ~\ref{finite_tDm}, the rings $\sD^{(k,m)}_{X_i}(U_i)$ are indeed noetherian. Again, from (iii) of
~\ref{finite_tDm}, we see that the rings $\sD^{(k,m)}_{\frX}(\frU)$ are noetherian. Moreover the sheaf $\sD^{(k,m)}_{\frX}$
is a filtered inductive limit of the $\cO_{\frX}$-coherent sheaves $\sD^{(k,m)}_{\frX,d}$ defined in the proof
of~\ref{finite_tDm}. This means that Berthelot's
conditions (1) and (2) are satisfied for $\frX'=\frX$ and $\sD=\sD^{(k,m)}_{\frX}$. Hence, the auxiliary results \ref{aux2} and \ref{aux3} can be applied in our context, proving (ii) of the proposition. The point (iii) is a direct consequence of (ii). This ends the proof of the proposition \ref{easy_thmAB}.
\end{proof}

\begin{prop}\label{resolution-m}  Let $\frU\subset \frX$ be an open affine formal subscheme of $\frX$, and $\sM$ a coherent
      $\hsD^{(k,m)}_{\frU}$-module. Then there are integers $a,b \geq 0$ and a short exact
		sequence of coherent $\hsD^{(k,m)}_{\frU}$-modules:
          $$    \left(\hsD^{(k,m)}_{\frU}\right)^a \ra \left(\hsD^{(k,m)}_{\frU}\right)^b \ra \sM \ra 0 \;.$$
\end{prop}
\begin{proof} Denote $\wD^{(m)}=\Ga(\frU,\hsD^{(k,m)}_{\frU})$ which is a noetherian ring by~\ref{finite_tDm}
   and $M=\Ga(\frU,\sM)$, which is a finite type $\wD^{(m)}$-module by the previous proposition ~\ref{easy_thmAB}.
Since the algebra $\wD^{(m)}$ is noetherian, there exists a finite presentation of $\wD^{(m)}$-modules
        $$ \left(\wD^{(m)}\right)^a \ra \left(\wD^{(m)}\right)^b \ra \ M \ra 0\;.$$ Tensoring this presentation by
$\hsD^{(k,m)}_{\frU}$ over the ring $\wD^{(m)}$ and observing that $$\sM \simeq \hsD^{(k,m)}_{\frU}\ot_{\wD^{(m)}}M \;,$$ again by  ~\ref{easy_thmAB}, gives the result.
\end{proof}
\begin{prop}  \label{prop-coherence_D-dagger2} We have:
\begin{enumerate} \item The morphism of sheaves $\hsD^{(k,m)}_{\frX,\Q}\ra  \hsD^{(k,m+1)}_{\frX,\Q}$
                     is left and right flat.
\item The sheaf of rings $\sD^{\dagger}_{\frX,k}$ is coherent.

\item For any affine open $\frU\subset\frX$,
$\Ga(\frU,\hsD^{(k,m+1)}_{\frX,\Q})$ is left and right flat over $\Ga(\frU,\hsD^{(k,m)}_{\frX,\Q})$.

\item For any affine open $\frU\subset\frX$, $\Ga(\frU,\sD^{\dagger}_{\frX,k})$ is left and right flat over $\Ga(\frU,\hsD^{(k,m)}_{\frX,\Q})$.
\end{enumerate}

\end{prop}

\begin{proof}
Let us first observe that (ii) follows from the flatness statement of (i) and the last part of \ref{finite_tDm} thanks to \cite[3.6.1]{BerthelotDI}.
For (i), we follow Berthelot's method described in \cite[3.5.3]{BerthelotDI}. For the proof we can restrict ourselves to proving that
if $\frU$ is an affine open of $\frX$ contained in the basis of open sets $\sB$ from \ref{finite_tDm}, then
the map $\hsD^{(k,m)}_{\frX,\bbQ}(\frU)\ra  \hsD^{(k,m+1)}_{\frX,\bbQ}(\frU)$ is left and right flat. In this
situation, we have the following explicit description, assuming that $\pr^*\sT_{\frX_0}$ is free
restricted to $\frU$, with basis $\der_1,\ldots,\der_M$ as in~\ref{ddagkX},
$$\sD^{(k,m)}_{\frX}(\frU)=\left\{\sum^{<\infty}_{\unu}\vpi^{k|\unu|}b_{\unu}\uder^{\lan \unu \ran}\,|\,b_{\unu}\in B\right\}\subset
		\widehat{\sD}^{(k,m)}_{\frX}(\frU)=\left\{\sum_{\unu}\vpi^{k|\unu|}b_{\unu}\uder^{\lan \unu \ran}\,|\,b_{\unu}\in B, b_{\unu}\ra 0\right\} $$
where the convergence is in the $\vpi$-adic topology of $B=\cO_{\frX}(\frU)$.
With this description, we can copy Berthelot's proof of ~\cite[3.5.3]{BerthelotDI}, replacing everywhere
the operators $\uder^{\lan \unu \ran}$ by $\vpi^{k|\unu|}\uder^{\lan \unu \ran}$, as follows. First of all, by inserting suitable powers of $\vpi$ into the formula \cite[(2.2.5.1)]{BerthelotDI} we see that the $\fro$-algebra $\sD^{(k,m)}_{\frX}(\frU)$ is generated by $B$ and the operators $\vpi^{kp^r}\der_l^{[p^r]}$ for $1 \leq r \leq m$ and $1 \leq l \leq M$.
Now write
$$ D^{(k,m)}=\sD^{(k,m)}_{\frX}(\frU) ,\textrm{ and } \hD^{(k,m)}=\hsD^{(k,m)}_{\frX}(\frU)\;,$$
the latter being the $\vpi$-adic completion of $D^{(k,m)}$~\cite[4.3.3.2]{BerthelotDI}. By the above explicit description, the two canonical maps
$ D^{(k,m)}\ra \hD^{(k,m)}\ra \hD^{(k,m)}_{\Q}$ are injective and this is also true for the canonical map $\hD^{(k,m)}\ra \hD^{(k,m+1)}$. Indeed, the latter is induced by mapping \begin{numequation}\label{formula_transition} \vpi^{k|\unu|}\uder^{\lan\unu\ran_{(m)}}\mapsto \vpi^{k|\unu|}\frac{\underline{q}_{(m)}!}{\underline{q}_{(m+1)}!}\uder^{\lan\unu\ran_{(m+1)}}\end{numequation}
and is therefore injective by the above explicit description and the fact that $\fro$ is torsionfree. Let us consider
the subring $E$ of $\hD^{(k,m)}_{\Q}$ generated by the subsets $\hD^{(k,m)}$ and $D^{(k,m+1)}$. Since
$D^{(k,m+1)}_{\Q}=D^{(k,m)}_{\Q}$, we see that $$E_{\Q}=\hD^{(k,m)}_{\Q} \;.$$
 As in  Berthelot's proof, we have the following
\begin{aux} $ E = \hD^{(k,m)} + D^{(k,m+1)} \;.$
\end{aux}
\begin{proof} Denote by $E'=\hD^{(k,m)} + D^{(k,m+1)}$. We have to prove that $E'$ is a subalgebra of $\hD^{(k,m+1)}$.
This is enough to prove that if $(P,Q)\in \hD^{(k,m)} \times D^{(k,m+1)}$, then $P\cdot Q\in E'$ and $Q\cdot P\in E'$.
Since the proof is the same in both cases, we only treat the  product $P\cdot Q$. As $Q\in D^{(k,m+1)}$, there
exists $c>0$ such that $\varpi^cQ \in D^{(k,m)}$. As $P\in \hD^{(k,m)}$, there exist $(P_1,R_1)\in D^{(k,m)}\times
\hD^{(k,m)}$ such that $P=P_1+\varpi^c R_1$, then $PQ=\varpi^cR_1Q+P_1Q\in E'$, as $\varpi^cR_1Q\in \hD^{(k,m)}$ and
$P_1Q\in D^{(k,m+1)}$.
\end{proof}
Denote by $\widehat{E}$
the $\varpi$-adic completion of $E$. We can then prove the
\begin{aux} $\widehat{E}_{\Q}=\hD^{(k,m+1)}_{\Q} \;.$
\end{aux}
\begin{proof} By construction, there are maps $$\lam_i \, \colon \,  D^{(k,m+1)}/\varpi^i D^{(k,m+1)}  \ra E/\varpi^i
E \;.$$ Let us prove that these maps are bijective. If $R\in D^{(k,m+1)}$ is such that $\lam_i(R)\in \varpi^i E$ then
there exist $(P,Q)\in \hD^{(k,m)}\times D^{(k,m+1)}$ such that $R=\varpi^i(P+Q)$, thus $\varpi^i P=R-\varpi^i Q$
has finite order, and $P\in D^{(k,m)}$, that implies that $R\in D^{(k,m+1)}$ and $\lam_i$ is injective. Pick
now $R\in E$, and $(P,Q)\in \hD^{(k,m)}\times D^{(k,m+1)}$ such that $R=P+Q$. The operator $P$ can be written
$P=P_1+\varpi^i R_1$, with $P_1\in D^{(k,m)}$ and $R_1\in \hD^{(k,m)}$, then $P=P_1+R+\varpi^i R_1$
and $\lam_i(P_1+R \textrm{ mod }\varpi^i D^{(k,m+1)})=P \textrm{ mod }\varpi^i E$, so that $\lam_i$ is surjective.
We finally see that $\lam_i$ is bijective, which proves the auxiliary result.
\end{proof}
 Now, the remaining thing to prove
is that $E$ is noetherian, since this result implies that $\widehat{E}$ is flat over $E$, thus
that $\hD^{(k,m+1)}_{\Q}$ is flat over $\hD^{(k,m)}_{\Q}$. The proof that $E$ is noetherian
proceeds by induction. By our above remark, $E$ is generated as (left) $D^{(k,m)}$-module by the elements
$\vpi^{k|\unu|p^{m+1}}(\uder^{[p^{m+1}]})^{\unu}$ for $\unu\in\bbN^M$. Let $1\leq l\leq M$. Inserting appropriate powers of $\vpi$ into the corresponding formula in Berthelot's proof one finds $[\vpi^{kp^{m+1}}\der_l^{[p^{m+1}]},b]\in D^{(k,m)}$ for any $b\in B$ and
so $[\vpi^{kp^{m+1}}\der_l^{[p^{m+1}]},P]\in \hD^{(k,m)}$ for any $P\in \hD^{(k,m)}.$ Using the general commutator identity (valid in any associative ring) $[Q^\nu,P]=[Q^{\nu-1},P]Q+Q^{\nu-1}[Q,P]$
one deduces from this inductively
$$ [ (\vpi^{kp^{m+1}}\der_l^{[p^{m+1}]})^\nu,P]\in\sum_{\mu<\nu} \hD^{(k,m)}(\vpi^{kp^{m+1}}\der_l^{[p^{m+1}]})^\mu\;.$$
This is the analogue of the key formula \cite[(3.5.3.2)]{BerthelotDI}. Now let $1\leq j\leq M$ and consider the subring $E_j$ of $E$ generated by
$E_0=\hD^{(k,m)}$ and the operators $\vpi^{kp^{m+1}}\der_l^{[p^{m+1}]}$ for $1 \leq l \leq j$.
Then $E_0$ is noetherian by~prop. \ref{finite_tDm} and, by our above discussion, $E_M=E$. With the key formula at hand, one may now follow Berthelot's proof word for word to obtain that $E_{j-1}$ noetherian implies $E_j$ noetherian. This proves (i).

\vskip8pt

Let us now prove (iii). Denote $\wD^{(m)}_{k,\Q}=\Ga(\frU,\hsD^{(k,m)}_{\frU,\Q})$ and consider
 $\alpha$ an injective map of coherent $\wD^{(m)}_{k,\Q}$-modules $\alpha$ : $M \hra N$. As a consequence of
(ii) and (iii) of~\ref{easy_thmAB}, we know that the sheaf $\hsD^{(k,m)}_{{\frU},\Q}$ is a flat
$\wD^{(m)}_{k,\Q}$-module. In particular, the map $\alpha $ provides an injection of
coherent $\hsD^{(k,m)}_{{\frU},\Q}$-modules $$ \hsD^{(k,m)}_{{\frU},\Q}\ot_{\wD^{(m)}_{k,\Q}}M \hra
\hsD^{(k,m)}_{{\frU},\Q}\ot_{\wD^{(m)}_{k,\Q}}N\;.$$ Using the flatness result (i) we also have an injection of coherent
$\hsD^{(k,m+1)}_{{\frU},\Q}$-modules  $$ \hsD^{(k,m+1)}_{{\frU},\Q}\ot_{\wD^{(m)}_{k,\Q}}M \hra
\hsD^{(k,m+1)}_{{\frU},\Q}\ot_{\wD^{(m)}_{k,\Q}}N\;.$$ Then we identify (resp. for $N$)
$$  \hsD^{(k,m+1)}_{{\frU},\Q}\ot_{\wD^{(m)}_{k,\Q}}M \simeq
\hsD^{(k,m+1)}_{{\frU},\Q}\ot_{\wD^{(m+1)}_{k,\Q}}\wD^{(m+1)}_{k,\Q}\ot_{\wD^{(m)}_{k,\Q}}M \;.$$
Finally taking global sections of the previous injection and using again ~\ref{easy_thmAB},
we get an injection of coherent $\wD^{(m+1)}_{k,\Q}$-modules
$$ \wD^{(m+1)}_{k,\Q}\ot_{\wD^{(m)}_{k,\Q}}M \hra \wD^{(m+1)}_{k,\Q}\ot_{\wD^{(m)}_{k,\Q}} N \;,$$ that proves (iii). Assertion (iv) follows from the previous one, since, as $\frU$ is quasi-compact,
$$\Ga(\frU,\sD^{\dagger}_{\frX,k})=\varinjlim_m \Ga(\frU,\hsD^{(k,m)}_{\frX,\Q})\;,$$
and we deduce from (iii) that, for all integer $m'\geq m$, the module
 $\Ga(\frU,\hsD^{(k,m')}_{\frX,\Q})$ is left and right flat over $\Ga(\frU,\hsD^{(k,m)}_{\frX,\Q})$. We obtain thus
(iv) by passing to the inductive limit over $m$. This ends the proof of the proposition \ref{prop-coherence_D-dagger2}.
\end{proof}

We deduce from this the corresponding version of proposition \ref{easy_thmAB} for the sheaf $\sD^{\dagger}_{\frX,k}$.
\begin{cor}\label{easy_thmAB2}{\rm (Local theorem A and B for varying $m$)}
Let $\frU \subset \frX$ be an open affine formal subscheme of $\frX$. Then :

\begin{enumerate}
\item The algebra $D^{\dagger}_k=\Ga(\frU,\sD^{\dagger}_{\frX,k})$ is coherent.
\item For any open affine subset $\frU \sub \frX$, any coherent $\sD^{\dagger}_{\frU,k}$-module $\sM$, and any $q>0$ one has $H^q(\frU,\sM) = 0$.
\item The functor $\Ga(\frU,.)$ establishes an equivalence of categories
between coherent $\sD^{\dagger}_{\frU,k}$-modules and coherent $D^{\dagger}_k$-modules. In particular, the functor $\Ga(\frU,.)$ is exact on coherent $\sD^{\dagger}_{\frU,k}$-modules.
\end{enumerate}
\end{cor}
\begin{proof}
Denote $\wD^{(m)}_{k,\bbQ}=\Ga(\frU,\hsD^{(k,m)}_{\frU,\bbQ})$.
 Since the
scheme $\frU$ is quasi-compact, the functors $H^q(\frU,.)$ commute with inductive limits and we have
$$D^{\dagger}_k=\varinjlim_m \wD^{(m)}_{k,\bbQ}\;.$$
By (iii) of the ~\ref{prop-coherence_D-dagger2}, the maps $\wD^{(m)}_{k,\bbQ}\ra \wD^{(m+1)}_{k,\bbQ}$ are flat,
and by ~\ref{finite_tDm} these algebras $\wD^{(m)}_{k,\bbQ}$ are noetherian. This showes that the algebra $D^{\dagger}_k$ is coherent \cite[3.6.1]{BerthelotDI}.

Let $\sM$ be a coherent $\sD^{\dagger}_{\frU,k}$-module. The proof of \cite[3.6.2]{BerthelotDI} literally applies in our situation and shows
that there is a non-negative integer $m_0$ and a coherent
$\hsD^{(k,m_0)}_{{\frU},\Q}$-module $\sN$ such that
$$ \sM \simeq \sD^{\dagger}_{\frU,k}\ot_{\hsD^{(k,m_0)}_{{\frU},\Q}}\sN \;.$$
Denote $M=\Ga(\frU,\sM)$ and for $m \geq m_0$  $$ \sM^{(m)}=\hsD^{(k,m)}_{\frU,\bbQ}\ot_{\hsD^{(k,m_0)}_{{\frU},\Q}} \sN\;,$$
so that $ \sM \simeq \varinjlim_m \sM^{(m)}.$ Then, $\forall m \geq m_0$, $H^q(\frU,\sM^{(m)})=0$ by~\ref{easy_thmAB}, and by passing to the inductive limit
we see that $H^q(\frU,\sM)=0$, which proves (ii).

The rest of the proof follows 2.3.7, 2.4.1, 2.4.2 of~\cite{NootHuyghe09}. For the convenience of the reader, let us
summmarize the arguments here. From~\ref{resolution-m}, the $\hsD^{(k,m_0)}_{{\frU},\Q}$-coherent module $\sN$ admits a resolution
		$$    {\left(\hsD^{(k,m_0)}_{\frU}\right)}^{a} \ra {\left(\hsD^{(k,m_0)}_{\frU}\right)}^b \ra \sN \ra 0 \;.$$
Tensoring this resolution with the sheaf $\sD^{\dagger}_{\frU,k}$ gives us a resolution of coherent
$\sD^{\dagger}_{\frU,k}$-modules. Since the global section functor is exact on the category of coherent
$\sD^{\dagger}_{\frU,k}$-modules because of (ii), we get an exact sequence of coherent $D^{\dagger}_k$-modules
		$$ \left(D^{\dagger}_k\right)^a \ra \left(D^{\dagger}_k\right)^b \ra M \ra 0 \;.$$
To see that $\sD^{\dagger}_k\ot_{D^{\dagger}_k}M \simeq \sM$, we are thus reduced to the case $\sM=\sD^{\dagger}_{\frU,k}$, for
which it is obvious. We prove similarly that if $M$ is a $D^{\dagger}_k$-coherent module, then
$$M \simeq \Ga(\frU,\sD^{\dagger}_{\frU,k}\ot_{D^{\dagger}_k} M) \;,$$
by reducing to the case where $M=D^{\dagger}_k$. This proves the proposition.
\end{proof}

We now give a flatness result when the congruence niveau $k$ varies.
\begin{prop}\label{flatFS} Let $k,k'$ be nonnegative integers such that $k'\geq k \geq k_{\frX}$, then the morphism
		of sheaves of rings $\hsD^{(k',m)}_{\frX,\Q}\hra \hsD^{(k,m)}_{\frX,\Q}$ is left and right flat.
\end{prop}
\begin{proof} By induction, it is enough to prove that the morphism $\hsD^{(k+1,m)}_{\frX,\Q}\hra \hsD^{(k,m)}_{\frX,\Q}$ is
		left and right flat.
		It is also enough to prove the following statement : if $\frU$ is an affine open of the basis
		of open sets $\sB$ from \ref{finite_tDm}, then
the map $\hsD^{(k+1,m)}_{\frX,\bbQ}(\frU)\hra  \hsD^{(k,m)}_{\frX,\bbQ}(\frU)$ is left and right flat.
Denote $D_{k}=\Ga(\frU,\sD^{(k,m)}_{\frX})$
		and $\hD_{k}=\Ga(\frU,\hsD^{(k,m)}_{\frX})$ (resp. with $k+1$).
In our situation, we have the following explicit description (resp. with $k+1$), where $B=\Ga(\frU,\cO_{\frX})$,
		assuming that $\pr^*\sT_{\frX_0}$ is free
        restricted to $\frU$, with basis $\der_1,\ldots,\der_M$ as in~\ref{ddagkX},

				$$D_{k} =\left\{\sum^{<\infty}_{\unu}\vpi^{k|\unu|}b_{\unu}\uder^{\lan \unu \ran}\,|\,b_{\unu}\in
		B\right\}\hskip10pt {\rm and}\hskip10pt \hD_{k}=\left\{\sum_{\unu}\vpi^{k|\unu|}b_{\unu}\uder^{\lan \unu
				\ran}\,|\,b_{\unu}\in B, b_{\unu}\ra 0  \right\} \;.$$
  Here, convergence is with respect to the $p$-adic topology on $B$. Moreover we have the inclusion $D_k \subset \hD_k$.
		As in the proof of the proposition \ref{prop-coherence_D-dagger2} above, we will use Berthelot's method \cite[3.5.3]{BerthelotDI}, and introduce
		$$ E=\hD_{k+1}+ D_{k}\;.$$ Then, since $D_{k,\Q}=D_{k+1,\Q}$, it is clear that
		$ E_{\Q}=\hD_{k+1,\Q}.$ Moreover, we have the following
		\begin{aux} The $B$-module $E$ is a $B$-algebra and its $p$-adic completion $\hE$ is canonically isomorphic
				to $\hD_{k}$.
		\end{aux}
		\begin{proof}
Let us first prove that $E$ is a ring. Any element $P$ of $E$ can be written as
		$$ P =Q+R,\; \textrm{ with  } Q\in \hD_{k+1} \textrm{ and } R\in D_{k}\;.$$ To prove that $E$ is a ring, it is enough
		to prove that
		$$ \forall \unu , \, \varpi^{k|\unu|}\uder^{\lan \unu \ran}\cdot \hD_{k+1}\subset E,\; \textrm{ and } \hD_{k+1} \cdot
		\varpi^{k|\unu|}\uder^{\lan \unu \ran} \subset E\;.$$
		Fix $\unu$, and take $Q\in \hD_{k+1}$. We can then write $Q=Q_1+\varpi^{|\unu|}Q_2$, with $Q_1 \in D_{k+1}$ and
		$Q_2\in \hD_{k+1}$. Since $\vpi^{k|\unu|}\uder^{\lan \unu \ran}\cdot Q_1 \in D_{k}$, and
		$Q_1 \cdot \vpi^{k|\unu|}\uder^{\lan \unu \ran}\in D_{k}$, it remains to prove that $\vpi^{k|\unu|}\uder^{\lan \unu \ran}\cdot
		\varpi^{|\unu|} Q_2 \in E$ (resp. $\varpi^{|\unu|}Q_2 \cdot \vpi^{k|\unu|}\uder^{\lan \unu \ran}\in E$). Let us write
		$$Q_2= \sum_{\unu'}\vpi^{(k+1)|\unu'|}b_{\unu'}\uder^{\lan \unu' \ran} \;,$$ with coefficients $b_{\unu'}\in B$ tending $p$-adically to zero in $B$. Besides the coefficients appearing in ~\ref{notacc}, we need other modified binomial coefficients \cite[(1.1.2.2)]{BerthelotDI}
		$$\crofrac{\unu}{\unu'} = \parfrac{\unu}{\unu'}\accfrac{\unu}{\unu'}^{-1}\, \in \Zp \;.$$
				Then, following \cite[(2.2.4)]{BerthelotDI} we have the following formulas
				
$$\begin{array}{rl} & \varpi^{k|\unu|}\uder^{\lan \unu \ran} \cdot \vpi^{((k+1)|\unu'|+|\unu|)}b_{\unu'}\uder^{\lan \unu' \ran} \\
& \\
= & \varpi^{(k+1)(|\unu|+|\unu'|)}\sum_{\umu \leq \unu} \accfrac{\unu}{\umu}\crofrac{\unu+\unu'-\umu}{\unu'}\uder^{\lan \umu \ran}(b_{\unu'})
			 \uder^{\lan \unu+\unu'-\umu \ran} \, \in \hD_{k+1}  \;.
\end{array}$$

Passing to the limit in the complete ring $\hD_{k+1}$ we see that $ \vpi^{k|\unu|}\uder^{\lan \unu \ran}\cdot
		\varpi^{|\unu|} Q_2 \in \hD_{k+1}$. Similarly, we have
		$$\vpi^{((k+1)|\unu'|+|\unu|)}b_{\unu'}\uder^{\lan \unu' \ran}\cdot \varpi^{k|\unu|}\uder^{\lan \unu \ran}=
		b_{\unu'}\vpi^{((k+1)(|\unu'|+|\unu|)}\crofrac{\unu+\unu'}{\unu}\uder^{\lan \unu+\unu'\ran}\;,$$
		which proves that $\varpi^{|\unu|}Q_2 \cdot \vpi^{k|\unu|}\uder^{\lan \unu \ran}\in E$ and that $E$ is a ring.

		Let $i \geq 0$ be an integer and consider now the canonical map $$\lam : D_{k}/\vpi^iD_{k}\ra E/\vpi^i E \;.$$
		Let $P\in E$, and $Q\in \hD_{k+1}, R\in D_{k}$ such that $P=Q+R$. There exist $Q_1\in D_{k+1}$ and $Q_2\in \hD_{k+1}$ such
		that $Q=Q_1+\vpi^i Q_2$. Then $$Q=\lam(\overline{Q_1}+\overline{R})\, {\rm mod} \, p^i E \;,$$ where $\overline{Q_1}$ and $\overline{R}$
		are the class of $Q_1$ and $R$ in the quotient $D_{k}/\vpi^iD_{k}$. This proves that the map $\lam$ is
		surjective. Suppose now that $\lam(\ovP)=0 $ for some $P\in D_{k}$, then there exist $Q\in \hD_{k+1}$ and
		$R\in D_{k}$ such that $P=\vpi^i(Q+R)$. As $$\vpi^iQ=P-\vpi^iR\in D_{k} \;,$$ we see from the explicit description of $D_{k+1}$ above
		that $Q\in D_{k+1}$
		and finally that $P\in \vpi^i D_{k}$, which proves that $\lambda $ is injective and thus an isomorphism. This completes the proof of the auxiliary result.
		\end{proof}

Now, to prove the proposition, it is enough to prove that $E$ is noetherian. If this is the case, then $\hE$ is flat
		over $E$, thus $\hE_{\Q}$ is flat over $E_{\Q}$ and $\hD_{k,\Q}$ is flat over $\hD_{k+1,\Q}$.
	
		Recall ~\cite[(2.2.5)]{BerthelotDI} that the ring $D_{k}$ is generated by the algebra $B$ and the elements
		$\vpi^{kp^a}\der_i^{\lan p^a \ran}=\vpi^{kp^a}\der_i^{[ p^a ]}$ with
          $1 \leq i \leq M$, $1\leq a \leq m$.
		Let us define the following algebras: let $E_0=\hD_{k+1}$, and for $j \geq 1$
		 $E_j$ be the $B$-algebra of $E $ generated by $\hD_{k+1}$ and the
		$\vpi^{k\nu_i}\der_i^{\lan \nu_i \ran}$ with $1 \leq i \leq j  $ and $\nu_i \in \mathbb{N}$.
		We also introduce for each $j$ and $s$ an integer satisfying $1 \leq s \leq m$
		the subalgebra $E_{j,s}$ of $E_j$ generated by  $E_{j-1}$ and
		the $\vpi^{k p^a}\der_j^{[ p^a ]}$ with $1 \leq a \leq s $. We define $E_{j,0}=E_{j-1}$ for $j \geq 1$.
		By definition, we have $E_{j,m}=E_j$. Now we use the

\begin{aux} For each $j\leq M$, $s\leq m$, the algebra $E_{j,s}$ is two-sided noetherian. The algebra $E$ is two-sided noetherian.
\end{aux}
\begin{proof}
We will prove the result by induction on both $j$ and $s$.
 Note that $E_0=\hD_{k+1}$ is noetherian by~\ref{finite_tDm}. By definition, $E_{0,0}=E_0$ and is thus noetherian.

Next, let us prove that if $E_{j,s-1}$ is noetherian, then $E_{j,s}$ is noetherian if $1 \leq s\leq m$ and
$1 \leq j \leq M$. For this, note that, if $b\in B$, we have as in \cite[3.5.3.2]{BerthelotDI},
$$ [\vpi^{kp^s}\der_j^{[p^s]},b] = \sum_{i<p^s} \parfrac{p^s}{i} \der_j^{[p^s-i]}(b)\vpi^{kp^s}\der_j^{[i]}\;,$$
thus noticing that if $i <p^s$, $\vpi^{ki}\der_j^{[i]}\in E_{j,s-1}$, we have
\begin{gather*} [\vpi^{kp^s}\der_j^{[p^s]},b]\in \sum_{i<p^s}\hD^{(k+1)}\vpi^{kp^i}\der_j^{[i]}\, \subset E_{j,s-1} \;.
\end{gather*}
Consider the finite type $\hD^{(k+1)}$-module $F:=\sum_{i<p^s}\hD^{(k+1)}\vpi^{kp^i}\der_j^{[i]}$. Then
for each finite sum $P=\sum_{\unu}b_{\unu}\vpi^{(k+1)|\unu|}\uder^{\lan \unu \ran}$, we observe that
$[\vpi^{kp^s}\der_j^{[p^s]},P]\in F$. Since $F$ is a finite type $\hD^{(k+1)}$-module, it is $p$-adically complete,
and thus also for each infinite sum $P\in \hD^{(k+1)}$, we have $[\vpi^{kp^s}\der_j^{[p^s]},P]\in F \subset E_{j,s-1}$.
Moreover, $[\vpi^{kp^s}\der_j^{[p^s]},\vpi^{kp^a}\der_j^{[p^a]}]=0$ for $a \leq s-1$. As $E_{j,s-1}$ is an
algebra and as we have the formula for $P,Q\in E_j$
$$ [\vpi^{kp^s}\der_j^{[p^s]}, PQ]=[\vpi^{kp^s}\der_j^{[p^s]}, P]Q+ P[\vpi^{kp^s}\der_j^{[p^s]}, Q]\;,$$
we see that $\forall P \in E_{j,s-1}$
\begin{gather*} [\vpi^{kp^s}\der_j^{[p^s]},P]\in  E_{j,s-1} \;.
\end{gather*}
Then, by induction on $l$, we deduce that
\begin{numequation} \label{formcle}\forall l\geq 1, \forall P\in E_{j,s-1}, \,[(\vpi^{kp^s}\der_j^{[p^s]})^l,P]\in
		\sum_{i\leq l-1}E_{j,s-1}(\vpi^{kp^s}\der_j^{[p^s]})^i.
\end{numequation}
Define $\Delta:=\vpi^{kp^s}\der_j^{[p^s]}$ for the rest of the proof. We follow now
Berthelot's argument in the proof of~\cite[3.5.3]{BerthelotDI}. We do the proof for 'left noetherian'
(the right version is similar). Let $I$ be a left ideal of $E_{j,s}$ and  $J$ be the set of elements $R$ of $E_{j,s-1}$
such that there exists $P\in I$
that can be written $$P=R \Delta^l+\sum_{i \leq l-1} R_i \Delta^i\;,$$
with $R_i \in E_{j,s-1}$.
 If $R_1$ and $R_2$ are in $J$, write
$$P_1=R_1 \Delta^{l_1}+\sum_{i \leq l_1-1} R_i \Delta^i, \textrm{ and }
P_2=R_2 \Delta^{l_2}+\sum_{i \leq l_2-1} R'_i \Delta^i\;,$$
with $R_i$ and $R'_i$ elements of $E_{j,s-1}$.
Assume $l_1 \geq l_2$, then, using ~\eqref{formcle}, we can write
$$P_1+\Delta^{l_1-l_2}P_2=(R_1+R_2)\Delta^{l_1}+\sum_{i\leq l_1-1} R''_i \Delta^i\, \in I\;,$$
with elements $R''_i\in E_{j,s-1}$. In particular, we deduce from this that $J$ is a left ideal of $E_{j,s-1}$
generated by a finite set of elements $R_1,\ldots , R_a$. Moreover $I \bigcap E_{j,s-1} $ is a left ideal of $E_{j,s-1}$
generated by a finite set of elements $Q_1,\ldots, Q_b$. We see easily then that $I$ is generated by
the elements $R_1, \ldots, R_a, Q_1, \ldots, Q_b$. This proves that $E_{j,s}$ is noetherian and ends
the proof of the auxiliary result. \end{proof}

As we have remarked above, the proof of the proposition \ref{flatFS} is now complete.
\end{proof}

\begin{cor}\label{flatFSaffine} Let $\frU\subset \frX$ be an open affine formal subscheme of $\frX$, $k,k'\geq k_{\frX}$, such
		that $k' \geq k$, then
		$\Ga(\frU,\hsD^{(k,m)}_{\frX,\Q})$ is left and right flat over $\Ga(\frU,\hsD^{(k',m)}_{\frX,\Q})$.
\end{cor}
\begin{proof} Denote $\hD_{k,\Q}= \Ga(\frU,\hsD^{(k,m)}_{\frX,\Q})$ and similarly for $k'$.
		Let $ M \hra N $ be an injection of finite type $\hD_{k',\Q}$-modules and put
		$$\sM=\hsD^{(k',m)}_{\frU,\Q}\ot_{\hD_{k',\Q}} M, \ \textrm{ resp. }
		\sN= \hsD^{(k',m)}_{\frU,\Q}\ot_{\hD_{k',\Q}} N \;.$$
		Using the equivalence of categories and the exactness in~\ref{easy_thmAB}, we have an injection of coherent
		$\hsD^{(k',m)}_{\frU,\Q}$-modules $ \sM \hra \sN$. Using the previous flatness result~\ref{flatFS},
		we have an injection of coherent $\hsD^{(k,m)}_{\frU,\Q}$-modules
		\begin{numequation}\label{injection2}
		\hsD^{(k,m)}_{\frU,\Q} \ot_{\hsD^{(k',m)}_{\frU,\Q}}\sM \hra \hsD^{(k,m)}_{\frU,\Q}
		\ot_{\hsD^{(k',m)}_{\frU,\Q}}\sN \;.
		\end{numequation} Taking global sections and using again~\ref{easy_thmAB},
		we observe that
		\begin{align*}\Ga(\frU,\hsD^{(k,m)}_{\frU,\Q} \ot_{\hsD^{(k',m)}_{\frU,\Q}}\sM)&\simeq
				\Ga(\frU,\hsD^{(k,m)}_{\frU,\Q} \ot_{\hD_{k'}} M)\\
				& \simeq
		\Ga(\frU, \hsD^{(k,m)}_{\frU,\Q}\ot_{\hD_{k}}\hD_{k}\ot_{\hD_{k'}} M ) \\
               & \simeq \hD_{k,\Q}\ot_{\hD_{k',\Q}} M
		\end{align*} and similarly for $\sN$. Finally, with~formula \eqref{injection2}
		we arrive at an injection
		$$ \hD_{k,\Q}\ot_{\hD_{k',\Q}}M \hra \hD_{k,\Q}\ot_{\hD_{k',\Q}} N\;,$$
		which proves the corollary.
\end{proof}

We compare now rings with different levels $k$. Let $e=e(L/\Qp)$ be the ramification index of the extension $\Qp\subseteq L$.

\begin{prop}\label{comp_level} Let $e' \in \Ne$ such that $e' \geq \frac{e}{p-1}$.
\begin{enumerate} \item
 If $k \geq k_{\frX}$, and $k'\geq k+e'$, then we have the following inclusions of sheaves of rings
$$ \sD^{\dagger}_{\frX,k'} \hra \sD^{\dagger}_{\frX,k+e'}  \hra \hsD^{(k,0)}_{\frX,\bbQ}
\hra \hsD^{(k,m)}_{\frX,\bbQ} \hra \sD^{\dagger}_{\frX,k}\;.$$
\item Suppose $e \leq p-1$. If $k \geq k_{\frX}$, and $k' \geq k+1$, then we have the following inclusions of sheaves of rings
$$ \sD^{\dagger}_{\frX,k'} \hra \sD^{\dagger}_{\frX,k+1}  \hra \hsD^{(k,0)}_{\frX,\bbQ}
\hra \hsD^{(k,m)}_{\frX,\bbQ} \hra \sD^{\dagger}_{\frX,k}\;.$$
\end{enumerate}
\end{prop}
\begin{proof} It is enough to prove (i). The only non trivial inclusion is $$ \sD^{\dagger}_{\frX,k+e'}
\hra \hsD^{(k,0)}_{\frX,\bbQ}\;,$$ which we may prove locally over some affine open $\frU \subset \frX$
of the basis of open sets $\sB$ from \ref{finite_tDm}. We use the following notations
$$B=\Ga(\frU,\cO_{\frX}), \hskip5pt \hD_k=\Ga(\frU,\hsD^{(k,0)}_{\frX}), \hskip5pt
D^{\dagger}_k=\Ga(\frU,\sD^{\dagger}_{\frX,k}) \;.$$ As before ~\eqref{ddagkX}, we have then the following descriptions,
 assuming that $\pr^*\sT_{\frX_0}$ is free
        restricted to $\frU$, with basis $\der_1,\ldots,\der_M$,
\begin{gather*}
	 \hD_{k}=\left\{\sum_{\unu}\vpi^{k|\unu|}b_{\unu}\uder^{\unu}\,|\,b_{\unu}\in B, b_{\unu}\ra
0 \right\} ,\\
   D^{\dagger}_k=\left\{\sum_{\unu}\vpi^{k|\unu|}b_{\unu}\uder^{[ \unu ]}\,|\, b_{\unu}\in B_{\Q}
\textrm{ and } \exists C>0, \eta<1 \, | \, \|b_{\unu}\|< C \eta^{|\unu|} \right\},
\end{gather*}
where $\|\cdot \|$ is any Banach algebra norm on $B_{\Q}$. For the rest of the proof we endow $B_{\Q}$ with the gauge norm $\|\cdot\|$ associated with the lattice $B \subset B_{\Q}$.
We need the following

\begin{aux}\label{lem_ineq} Let $\unu=(\nu_1,\ldots,\nu_M)$, then we have
         $$ \frac{|\unu|}{p-1}-M \log_p(|\unu|+1)-M \leq v_p(\unu!) \leq \frac{|\unu|}{p-1} \;.$$
\end{aux}

\Pf \cite[1.1.2]{NootHuyghe07}. \qed

\vskip8pt

Let $P=\sum_{\unu}\vpi^{(k+e')|\unu|}b_{\unu}\uder^{[ \unu ]}\in D^{\dagger}_{k+e'}$, then there exist $R>0$ and $S>0$ such that

$$\log_p \|b_{\unu}\| > R|\unu| - S \;.$$

We can write

$$ P=\sum_{\unu} \vpi^{k|\unu|}c_{\unu} \uder^{\unu}\, \quad \textrm{ with }
c_{\unu}=\frac{\vpi^{e'|\unu|}}{\unu!}b_{\unu} \;.$$

Then the following inequality holds

$$\log_p \|c_{\unu}\| > \left(\frac{e'}{e}-\frac{1}{p-1}+R \right)|\unu|-S \;.$$

Under the conditions of the statement, $\log_p\|c_{\unu}\| \ra +\infty$ if $|\unu|\ra +\infty$, which
proves that $P\in \hD_{k,\Q}$. This ends the proof of the proposition \ref{comp_level}.
\end{proof}

We now complete these results by additional flatness results when $k$ varies.
\begin{prop}\label{flatFSdag} Let $k'\geq k$.
 \begin{enumerate}         \item The morphism
				of sheaves of rings $\sD^{\dagger}_{\frX,k'}\hra \sD^{\dagger}_{\frX,k}$ is left and right flat.
		\item Let $\frU\subset \frX$ be an open affine. Then
		$\Ga(\frU,\sD^{\dagger}_{\frX,k})$ is left and right flat over $\Ga(\frU,\sD^{\dagger}_{\frX,k'})$.
        \end{enumerate}
\end{prop}
\begin{proof} It is enough to prove (ii). Denote $D^{\dagger}_k=\Ga(\frU,\sD^{\dagger}_{\frX,k})$ (resp. for $k'$),
and for any integer $m \geq 0$, $\hD^{(m)}_{k,\Q}=\Ga(\frU,\sD^{(k,m)}_{\frX,\Q})$ (resp. for $k'$).
Since $\frU$ is quasi-compact, we know that
$$ D^{\dagger}_k=\varinjlim_m \hD^{(m)}_{k,\Q}\;.$$
Let $u$ : $M \hra N$ be an injection of coherent $D^{\dagger}_{k'}$-modules, then we have the
\begin{aux} There exist an integer $m \geq 0$ and an injection $u^{(m)}$ of coherent $\hD^{(m)}_{k',\Q}$-modules
$u^{(m)}$ : $M^{(m)} \hra N^{(m)}$ such that following properties are satisfied.
\begin{enumerate}
\item There are canonical isomorphisms
$$ D^{\dagger}_{k'}\ot_{\hD^{(m)}_{k',\Q}}M^{(m)}\sta{\alpha}{\simeq} M \;, \textrm{ and } D^{\dagger}_{k'}\ot_{\hD^{(m)}_{k',\Q}}N^{(m)}\sta{\beta}{\simeq} N  \;.$$
\item There is a commutative diagram of coherent $D^{\dagger}_{k'}$-modules
	$$ \xymatrix@R=2pc { D^{\dagger}_{k'}\ot_{\hD^{(m)}_{k',\Q}}M^{(m)}\ar@{->}[d]^{\alpha} \ar@{^{(}->}[r]^{1 \ot u^{(m)}}&
 D^{\dagger}_{k'}\ot_{\hD^{(m)}_{k',\Q}} \ar@{->}[d]^{\beta} N^{(m)} \\
                        M\ar@{^{(}->}[r]^{u} & N  }$$
		\end{enumerate}
\end{aux}
\begin{proof}
In order to prove the auxiliary result, we first remark that the morphism $u$ can be extended to any finite presentation of $M$ and $N$ as
$D^{\dagger}_{k'}$-modules. Thus, there exists a diagram of presentations of $D^{\dagger}_{k'}$-modules
 $$ \xymatrix {   (D^{\dagger}_{k'})^a \ar@{->}[r] \ar@{->}[d]^A & (D^{\dagger}_{k'})^b \ar@{->}[d]^B\ar@{->}[r]
& M \ar@{->}[r]\ar@{->}[d]^{u} & 0 \\
 (D^{\dagger}_{k'})^c \ar@{->}[r] & (D^{\dagger}_{k'})^d \ar@{->}[r] & N \ar@{->}[r]& 0}$$
Furthermore, there exists $m$ such that the matrices of the maps $A$ and $B$ have coefficients in $\hD^{(m)}_{k',\Q}$
and define maps  $$ A \,\colon\, (\hD^{(m)}_{k',\Q})^a \ra (\hD^{(m)}_{k',\Q})^b \textrm{ (resp. for } B\textrm{)} \;,$$
whose cokernel is a coherent $\hD^{(m)}_{k',\Q}$-module denoted by $M^{(m)}$ (resp. $N^{(m)}$ is the cokernel of $B$).
We finally get from this the following commutative diagram of exact sequences
$$ \xymatrix {   (\hD^{(m)}_{k',\Q})^a \ar@{->}[r] \ar@{->}[d]^A & (\hD^{(m)}_{k',\Q})^b \ar@{->}[d]^B\ar@{->}[r]
& M^{(m)} \ar@{->}[r]\ar@{->}[d]^{u^{(m)}} & 0 \\
 (\hD^{(m)}_{k',\Q})^c \ar@{->}[r] & (\hD^{(m)}_{k',\Q})^d \ar@{->}[r] & N^{(m)} \ar@{->}[r] & 0}$$
where by definition $u^{(m)}$ is the induced map by $u$ between $M^{(m)}$ and $N^{(m)}$. By construction
there are canonical isomorphism $$D^{\dagger}_{k'}\ot_{\hD^{(m)}_{k',\Q}}M^{(m)}\simeq M \textrm{ (resp. for
}N\textrm{)} \;.$$
Define now $ K$ to be the kernel of the map $u^{(m)}$ : $M^{(m)}\ra N^{(m)}$, then as $D^{\dagger}_{k'}$ is flat
over $\hD^{(m)}_{k',\Q}$ by ~\ref{prop-coherence_D-dagger2}, we have an exact sequence of $D^{\dagger}_{k'}$-modules
$$ 0 \ra D^{\dagger}_{k'}\ot_{\hD^{(m)}_{k',\Q}}K \ra M \ra N\;,$$
showing that $$D^{\dagger}_{k'}\ot_{\hD^{(m)}_{k',\Q}}K=0 \;,$$
and again, as  $D^{\dagger}_{k'}$ is flat over $\hD^{(m)}_{k',\Q}$, that
$$ D^{\dagger}_{k'}\ot_{\hD^{(m)}_{k',\Q}}\left(M^{(m)}/K\right)\simeq D^{\dagger}_{k'}\ot_{\hD^{(m)}_{k',\Q}}M^{(m)}\simeq  M\;.$$
Finally, the $\hD^{(m)}_{k',\Q}$-coherent modules
$M^{(m)}/K$ and $N^{(m)}$ satisfy the required properties. This ends the proof of the auxiliary result.
\end{proof}

Take $M^{(m)}$, $N^{(m)}$ and $u^{(m)}$ : $M^{(m)}\hra N^{(m)}$ as given by the auxiliary result. As $\hD^{(m)}_{k,\Q}$ is flat
over $\hD^{(m)}_{k',\Q}$ by ~\ref{flatFSaffine}, we have an injection of coherent $\hD^{(m)}_{k,\Q}$-modules
$$ \hD^{(m)}_{k,\Q}\ot_{\hD^{(m)}_{k',\Q}}M^{(m)} \hra \hD^{(m)}_{k,\Q}\ot_{\hD^{(m)}_{k',\Q}}N^{(m)}\;.$$
We can tensor this map by $D^{\dagger}_k$ which is a flat $\hD^{(m)}_{k,\Q}$-module~\ref{prop-coherence_D-dagger2},
 and use the properties
of $M^{(m)}$ and $N^{(m)}$ to get an injection
$$ D^{\dagger}_k\ot_{D^{\dagger}_{k'}} M \hra  D^{\dagger}_k\ot_{D^{\dagger}_{k'}} N\;,$$
and this proves the proposition.
\end{proof}

\vskip5pt

We end this part by some global properties of coherent sheaves over $\frX$ (resp. coherent $D^{\dagger}_k$-modules) when the base $\frX_0$ is affine.
In general, the formal scheme $\frX$ is projective over $\frX_0$ and we can consider the Serre twist $\cO_{\frX}(1)$. It is
a locally free sheaf of rank $1$ over $\frX$. If $\frU_{0}$ is an open affine formal subscheme of $\frX_{0}$,
and if $t\in \cI(\frU_0)$, then the restriction of the sheaf $\cO_{\frX}(1)$ to the open subset $D_+(t)$ is generated by $t$. As usual, $\cI$ denotes the ideal sheaf on
$\frX_0$ which gives rise to the blow-up $\frX$.

If $\cM$ is a sheaf of $\cO_{\frX}$-modules over $\frX$ and if $r\in \Z$, then we let $\cM(r)$ denote the twisted sheaf

$$\cM(r)=\cM\ot_{\frX}\frX(r) \;.$$

\begin{lemma}\label{twist} Let $\frX_0$ be a noetherian affine formal scheme. For all integers $r$, there is an
		isomorphism $\cO_{\frX,\Q}(-r)\simeq \cO_{\frX,\Q}(-1) $.
\end{lemma}
\begin{proof} We can assume as in~\ref{subsec-const}, that $\frX_0=\Spf(A)$, and that the ideal $\cI$ is generated by the elements $(\varpi^{k_{\frX}},f_1,\ldots,f_r)$. The formal scheme $\frX$ is then covered by
		the open formal subschemes $$D_+(f_i)=\Spf {A\left\{\frac{f_j}{f_i}\right\}} \;,$$ and
			  $$\cO_{\frX}(-1)(D_+(f_i))\simeq {A\left\{\frac{f_j}{f_i}\right\}}\cdot \frac{1}{u_i} \;,$$
         where $u_i$ is a generator of $\cO_{\frX}(1)_{|D_+(f_i)}$.
			  By definition, we have a section
			  $$\alpha=\frac{\varpi^{k_{\frX}}}{u_0}\in \cO_{\frX}(-1)(D_+(\varpi^{k_{\frX}})) \;.$$
			  Moreover the following equations hold
			  $$\frac{\varpi^{k_{\frX}}}{u_0} =\frac{f_i}{u_i}\, \in \cO_{\frX}(-1)(D_+(f_i)) \;,$$
			  $$ \frac{\varpi^{k_{\frX}}}{u_i}=\frac{\varpi^{k_{\frX}}}{f_i}
			  \cdot \frac{\varpi^{k_{\frX}}}{u_0}\, \in \cO_{\frX}(D_+(f_i))\cdot \frac{\varpi^{k_{\frX}}}{u_0} \;.$$
			  This proves that $\cO_{\frX,\Q}(-1)$ is free with basis $\alpha=\varpi^{k_{\frX}}/u_0$.
               If $r\geq 0$,  $\cO_{\frX,\Q}(-r)\simeq  \cO_{\frX,\Q}(-1)^{\otimes r}$ is thus a free
            $\cO_{\frX,\Q}$-module of rank one as well. Since $\cO_{\frX,\Q}(r)$ is the dual of $\cO_{\frX,\Q}(-r)$,
              it is then also a free $\cO_{\frX,\Q}$-module of rank one.
             \end{proof}
We assume from now on until the rest of this subsection 2.2 that the formal scheme $\frX_0$ is affine.
Let $k\geq k_{\frX}$. To simplify the notation, we write $\sD:= \sD^{(k,m)}_{\frX} $ and we denote by
$\hsD$ the $p$-adic completion of this sheaf. We also let $\sD^{\dagger}=\varinjlim_m \hsD^{(k,m)}_{\frX,\Q}$.

\vskip5pt

Let us first consider the reduction $\sD_i=\hsD/\varpi^{i+1}\hsD$. This is a coherent sheaf thanks to~\ref{aux1} and a
		quasi-coherent sheaf of $\cO_{X_i}$-modules. Let $\cM$ be a coherent $\sD_i$-module.
		
\begin{lemma} There exist $a,r\in \Ne$
						such that there is a surjection of coherent $\sD_i$-modules
						$$ \left(\sD_i(-a)\right)^r \tra \cM \;.$$
\end{lemma}
\begin{proof} As the sheaf $\cM$ is quasi-coherent over the noetherian scheme $X_i$, it is an
inductive limit of its $\cO_{X_i}$-coherent subsheaves. Moreover $\sD_i$ is an inductive limit
of the coherent sheaves $\sD_{i,n}$ of differential operators of order less than $n$.
We can thus write
$$ \cM=\varinjlim_{n\in \Ne} \cM^{(n)} \;,$$
where $\cM^{(n)}$ is a coherent $\cO_{X_i}$-module.
				   Take an open affine subscheme $U \subset X_i$. Then
				 $\sD_i(U)$ is noetherian and $\cM(U)$ is a $\sD_i(U)$-module of finite type. Hence,
				   there exists $N>0$ such that $$ \varinjlim_n \sD_i(U)\cdot \cM^{(n)}(U)
				   =\sD_i(U)\cdot \cM^{(N)}(U)\,\subset \cM(U) \;.$$ Since
				   $$ \cM(U)=\varinjlim_n \sD_i(U)\cdot \cM^{(n)}(U) \;,$$
				   we see that $\sD_i(U)\cdot \cM^{(N)}(U)=\cM(U)$ and using (iii) of
				   \ref{aux1}, we find a surjection
				   $ \sD_{U}\ot \cM^{(N)}_{|U} \tra \cM_{|U}.$
				   As $X_i$ is quasi-compact, $X_i$ can be covered by a finite number of affine open
				   subsets. There exists therefore $N'$ and a surjection
                   $ \sD_i\ot \cM^{(N')} \tra \cM.$ As $\cM^{(N')}$ is a coherent
				   $\cO_{X_i}$-module, there exist $r\in \Ne$, $a\in \Ne$ and a surjection
					 $ \left(\cO_{X_i}(-a)\right)^r \tra \cM^{(N')}$ and this proves the lemma.
            \end{proof}
			\begin{lemma} \begin{enumerate} \item There exists $a \geq 0$ such that
                       $$ \forall b\geq a,\forall l>0,\, H^l(X_i,\sD_i(b))=0 \;.$$
                        \item Let $\cM$ be a coherent $\sD_i$-module, then there exists $a\geq 0$ such that
                        $$ \forall b\geq a,\forall l>0,\, H^l(X_i,\cM(b))=0 \;.$$
                         \end{enumerate}
             	\end{lemma}
					\begin{proof} We have $R\Ga(X_i,\sD_i(b))=R\Ga(X_{0,i},\cdot)\circ R\pr_*\sD_i(b)$. As $X_{0,i}$
							is affine, it is enough to prove that $R^l\pr_*\sD_i(b)=0$ for $l \geq 1$ and
							for $b\geq a$. Take $\frU_0\subset \frX_0$ affine, endowed with coordinates
							$x_1,\ldots,x_M$ and let $\frU=\pr^{-1}(\frU_0)$ and
							$U=\frU \times \Spec (\fro/\varpi^{i+1}\fro)$. Denote by $\cO_{\frU}(1)$ the restriction to
							$\frU$ of the Serre twist over $\frX$. Then $\sD_{i|U}$ is a free
							$\cO_{U}$-module, so that there exists $c$ such that
                           $$ \forall b\geq c,\forall l>0,\, H^l(U,\sD_i(b))=0 \;.$$
							By taking the maximum $a$ of the constants $c$ for each affine open $\frU_0$ of a finite
							cover of $\frX_0$ by such affine subschemes $\frU_0$, we get that
                            $R^l\pr_*\sD_i(b)=0$ for $l>0$ and $b\geq a$ and this proves part (i) of the lemma. For the
                            second assertion, we will prove the following statement by decreasing induction on $K$ :

                            \vskip5pt

                        For all coherent   $\sD_i$-modules  $\cN$,  there exists $a\geq0$,
                         such that  $$\forall L \geq K, \forall b\geq a, H^L(X_i,\cN(b))=0 \;.$$

                         \vskip5pt

                        If $K\geq M+1$, the result is clear, since $\cN$ is a quasi-coherent sheaf on a scheme
                    of dimension $M$. Suppose now that the result is true for $K\geq 2$, and consider
                a coherent $\sD_i$-module $\cN$. By the previous lemma, there exist $a$ and $r$ and a exact sequence of coherent $\sD_i$-modules
                       $$ 0 \ra \cM \ra \sD_i^r \ra \cN(a) \ra 0 \;.$$
                       Tensoring this sequence by $\cO_{X_i}(c_1+c_2)$ where $c_1$, $c_2$ are non negative integers and
                 looking at the cohomology long exact sequence, we get exact sequences for all $L$
                         $$ H^L(X_i,\sD_i^r(c_1+c_2))\ra  H^L(X_i,\cN(a+c_1+c_2))\ra H^{L+1}(X_i,\cM(c_1+c_2)) \;.$$
                    By the induction hypothesis, there exists
                    $c_1$ such that $\forall L \geq K, \forall d \geq c_1, H^{L}(X_i,\cM(d))=0$, and
                      by (i), there exists
                    $c_2$ such that $\forall L \geq 1, \forall d \geq c_2, H^{L}(X_i,\sD_i(d))=0$.
                     Finally, if $d\geq a+c_1+c_2$, for every $L\geq K-1\geq 1$, we get that
                            $H^L(X_i,\cN(d))=0$. This proves (ii) by induction as claimed.
					\end{proof}

\begin{prop}\label{Xaf}\begin{enumerate} \item For any coherent $\hsD$-module $\cM$, there exist $a,r\in \Ne$
						and a surjection of coherent $\hsD$-modules
						$$ \left(\hsD(-a)\right)^r \tra \cM \;.$$
				\item For any coherent $\hsD_{\Q}$-module $\cM$, there exist $r\in \Ne$
						and a surjection of coherent $\hsD_\Q$-modules
						$$ \left(\hsD_{\Q}\right)^r \tra \cM \;.$$
                 \item For any coherent $\sD^{\dagger}$-module $\cM$, there exist $r\in \Ne$
						and a surjection of coherent $\sD^\dagger$-modules
						$$ \left(\sD^{\dagger}\right)^r \tra \cM \;.$$
			\end{enumerate}
\end{prop}
\begin{proof} Let $\cM$ be a coherent $\hsD$-module. For part (i), we need to show that there
exists a non negative integer $a$ such that the twist $\cM(a)$ is generated by a finite number of global sections. Let $\cM_t$
be the torsion part of $\cM$, which is a coherent submodule of $\cM$, since $\hsD(\frU)$ is
		noetherian for every affine open $\frU$. Over an affine open $\frU$ of $\frX$, the module $\cM_t(\frU)$ is a finite
		type module over $\hsD(\frU)$ and there exists a constant $c$ such that $\varpi^{c}\cM_t(\frU)=0$. Since $\frX$ can be
		covered by a finite number of open affine formal subschemes, there exists $L$ such that $\varpi^L \cM_t=0$. Then
       for $i \geq L$, and denoting $\cG_0=\cM/\left(\varpi \cM + \cM_t \right)$ we have exact sequences
        $$ 0 \ra \cG_0 \sta{\varpi^i}{\ra} \cM/\varpi^{i+1}\cM \ra \cM/\varpi^{i}\cM \ra 0 \;.$$
        Since $\cG_0$ is a coherent $\sD_0$-module, there exists $a_1\geq 0 $ so that $H^1(X_0,\cG_0(b))=0$
          for every $b \geq a_1$. As a consequence, if $b \geq a_1$, for all $i \geq L$ we have surjections
		  $$ \Ga(\frX,\cM/\varpi^{i+1}\cM(b))\tra \Ga(\frX,\cM/\varpi^{i}\cM(b)) \;.$$
		  Since $\cM/\varpi^{L}\cM$ is a coherent $\sD_{L-1}$-module, there exists
   $a \geq a_1$ and a surjection $s$ : $\sD_{L-1}^r \tra \cM/\varpi^{L}\cM(a)$ defined
   by global sections $ e_1,\ldots,e_r\in \Ga(X_{L-1},\cM/\varpi^{L}\cM(a))$. Finally we see by induction
   on $i$ that these sections $e_1,\ldots,e_r$ can be lifted to global sections of
   $\Ga(\frX,\cM/\varpi^{i}\cM(a))$ for every $i$, and thus to global sections of
   $\Ga(\frX,\cM(a))$. These sections define a map $\hsD^r \ra \cM(a)$ that is surjective since it
   is surjective mod $\varpi$. This proves the part (i).

			Assertion (ii) follows from (i) and lemma ~\ref{twist}. For (iii), we remark that, if $\cM$ is a coherent
			$\sD^{\dagger}$-module, there exists a coherent $\hsD$-module $\cN$ such that
			$$\cM\simeq \sD^{\dagger}\ot_{\hsD}\cN \;.$$
			We can then apply (ii) to $\cN$ and this proves (iii).
\end{proof}

\subsection{An invariance theorem for admissible blow-ups}
We keep here the hypothesis from the previous section. In particular, $\frX_0$ denotes a smooth formal $\frS$-scheme and
$$\pr: \frX \ra \frX_0$$ denotes an admissible blow-up. Let $\pr':  \frX' \rightarrow \frX_0$ be another admissible formal
blow-up and let $\pi: \frX' \rightarrow \frX$ be a morphism over $\frX_0$, inducing an isomorphism between the
associated rigid-analytic spaces $\frX_{\Q}$ and $\frX'_\Q$ (which are both canonically isomorphic to the rigid-analytic
space $\frX_{0,\Q}$ associated to $\frX_0$). Then we have the following invariance property. This does not make use of
the smoothness assumption for $\frX_0$.
\begin{prop}\label{eq_catO}
The functors $\pi_*$ (resp. $\pi^*$) are exact on the category of coherent $\cO_{{\frX',\Q}}$-modules
(resp. coherent $\cO_{{\frX,\Q}}$-modules) and induce an equivalence of categories
between coherent $\cO_{{\frX',\Q}}$-modules
and coherent $\cO_{{\frX,\Q}}$-modules.
\end{prop}

\begin{proof} Let $\sp$ (resp. $\sp'$) be the specialization map $\frX_\Q\rightarrow \frX$ (resp. $\frX'_\Q \ra \frX'$). Then by
Tate's acyclicity theorem
one knows that $\sp_*$ is exact over the category of coherent $\cO_{{\frX}_\Q}$-modules. Moreover,
via specialization,
the category of coherent $\cO_{\frX_\Q}$-modules over the rigid space $\frX_\Q$
is equivalent to the category of  coherent
$\cO_{\frX,\Q}$-modules over the formal scheme $\frX$ \cite[discussion after (4.1.3.1)]{BerthelotDI} and similarly for $\frX'$.
Let $\tpi$ be the induced map by $\pi$
between the analytic spaces $\frX'_\Q$ and $\frX_\Q$, which is an isomorphism by assumption. One has the following
commutative
diagram
$$\xymatrix {\frX'_\Q \ar@{->}[r]^{\sim}_{\tpi}\ar@{->}[d]^{\sp'}  & \frX_\Q\ar@{->}[d]^{\sp}\\
          \frX' \ar@{->}[r]_{\pi}& \frX, }$$

from which we can deduce the following
		\begin{lemma}\label{lemma_int} With the previous notations,
there is an isomorphism $$\cO_{\frX,\Q}{\simeq}\pi_*\cO_{\frX',\Q} \;.$$
        \end{lemma}
\begin{proof} Let us first note that $R\tpi_*\cO_{{\frX'}_\Q}\simeq \cO_{{\frX}_\Q}$ because $\tpi$ is an isomorphism of
analytic spaces.
            Since $R\sp'_*\cO_{{\frX'}_{\Q}}=\cO_{{\frX',\Q}}$ (and the same with $\frX$), we can compute using the
previous diagram
           $R\pi_*\cO_{{\frX',\Q}} \simeq R\sp_*R\tpi_* \cO_{{\frX'}_{\Q}} \simeq \cO_{{\frX,\Q}}.$ This proves the
lemma.
\end{proof}
Let us now prove the proposition. Let $\cF$ be a coherent $\cO_{{\frX',\Q}}$-module, then by \cite[4.1.3]{BerthelotDI}
there exists a coherent
$\cO_{{\frX'}_\Q}$-module $\tcF$ such that $\cF=\sp'_*\tcF$. Considering again the previous diagram, we compute
$$ R\pi_*\cF \simeq R\sp_* R\tpi_* \tcF\;.$$ As $\tpi$ is an isomorphism,  $R^i\tpi_* \tcF=0$ if $i \geq 1$ and
$\tpi_* \tcF$ is a coherent $\cO_{{\frX}_\Q}$-module. Finally, $R\sp_*$ is reduced to $\sp_*$ and the spectral sequence
of the composite functors degenerates, giving us that $R^i\pi_* \cF=0$ if $i \geq 1$ and $\pi_* \cF$ is a coherent
$\cO_{{\frX,\Q}}$-module. It is moreover clear that $\pi^*$ preserves coherence. Consider the map of
coherent $\cO_{\frX',\Q}$-modules $\pi^*\pi_*\cF \rightarrow \cF$. To prove that this is an isomorphism is local
on $\frX$, which we can assume to be affine. In this case, it is enough to prove the statement for $\cF=\cO_{\frX',\Q}$,
since
$\pi_*$ is exact. But $\pi^* \pi_*\cO_{\frX',\Q}\simeq \pi^*\cO_{\frX,\Q}$ because of the lemma and thus
$\pi^* \pi_*\cO_{\frX',\Q}\simeq \cO_{\frX',\Q}$.

Let $\cE$ be a coherent $\cO_{{\frX,\Q}}$-module and consider the canonical map $\cE \rightarrow \pi_*\pi^*\cE$. Again,
since $\pi_*$ is exact, we are reduced to the case where $\frX$ is affine and $\cE=\cO_{\frX,\Q}$ to prove that this map
is an
isomorphism. In this case, the isomorphism follows again from the lemma.

Since $\pi_*$ and $\pi^*$ are quasi-inverse to each other, and $\pi_*$ is an exact functor, $\pi^*$ is exact as well.
This finishes the proof of the proposition.
\end{proof}

In the sequel, we will give a version of the invariance property~\ref{eq_catO} for $\sD$-modules. Recall that we have
the sheaves of differential operators
$\sD^{\dagger}_{\frX,k}, \hsD^{(k,m)}_{\frX,\Q}$ etc. for $k \geq k_{\frX}$, cf. previous subsection, at our disposal
(similarly for $\frX'$). In the following we {\it fix} a congruence level $k\geq {\rm max}\{k_{\frX},k_{\frX'}\}$.
\vskip5pt

For a $\sD^{\dagger}_{\frX',k}$-module $\sM$, we let, as usual, $\pi_*\sM$ denote
 the push-forward of $\sM$ in the sense of abelian sheaves (and analogously in the case of
$\hsD^{(k,m)}_{\frX',\Q}$-modules).
Conversely, there is a functor $\pi^{!}$ in the other direction constructed as usual using the formalism of inverse
images of $\sD$-modules:
first of all, by definition of the sheaves
 $\sD^{(k,m)}_{X'_i}$ and $\sD^{(k,m)}_{X_i}$, cf. (\ref{def1D}), and the fact that $(\pr')^*=\pi^*\circ\pr^*$ we have
$$\sD^{(k,m)}_{X'_i}\simeq \pi_i^*\sD^{(k,m)}_{X_i}\;,$$ and the sheaf
$\sD^{(k,m)}_{X'_i}$ can be uniquely endowed with a structure of right $\pi^{-1}\sD^{(k,m)}_{X_i}$-module.
Passing to the $p$-adic completion, we see that the sheaf $\hsD^{(k,m)}_{\frX'}$ is a sheaf of right
$\pi^{-1}\hsD^{(k,m)}_{\frX}$-modules.
Then, passing to the inductive limit over $m$ implies that $\sD^{\dagger}_{\frX',k}$ is a right
$\pi^{-1}\sD^{\dagger}_{\frX,k}$-module. For a $\sD^{\dagger}_{\frX,k}$-module $\sM$, we then define
\begin{numequation}\label{inversefunctor}\pi^!\sM :=\sD^{\dagger}_{\frX',k}\otimes_{\pi^{-1}\sD^{\dagger}_{\frX,k}}\pi^{-1}\sM \;,\end{numequation}
and we make the analogous definition in the case of
$\hsD^{(k,m)}_{\frX,\Q}$-modules.\footnote{Since $\sD^{(k,m)}_{\frX'}=\pi^*\sD^{(k,m)}_{\frX}$, the functor $\pi^{!}$ is a version of the usual $D$-module pullback functor \cite{Hotta}, whence our notation.}

\vskip8pt

Before stating the next theorem, we need the following lemmas. Denote by $\cA$ the abelian category of projective systems $K_\bullet = (K_i)_{i \in \bbN} = (K_0 \leftarrow K_1 \leftarrow \ldots)$ of $\cO_\frX$-modules $K_i$, where $K_i$ is annihilated by multiplication by $\vpi^{i+1}$ for every $i \ge 0$. Put $\cP^b = D^b(\cA)$. Note that for a complex $(K^n)_{n \in \Z}$ in $\cP^b$, where each $K^n = (K^n_i)_{i \in \bbN}$,  there exists $J>0$ so that $\cH^n(K^\bullet)=0$ if $|n|>J$ ($n\in \Z$). The functor $\varprojlim: \cA \ra \Mod(\cO_\frX)$ extends to a derived functor $R\varprojlim$ from the derived category $\cP^b$ to the bounded derived category $D^b(\cO_\frX) := D^b(\Mod(\cO_\frX))$ because $R\varprojlim$ has cohomological dimension $1$, cf. \cite[4.1]{HartshorneO}.

\begin{lemma}\label{lemprojlim} Let $N \in \Ne$, and let
$K_\bullet$ be an object of $\cA$ such that $\varpi^N K_\bullet = 0$, then the complex $\Q \ot_{\Z} R\varprojlim K_\bullet$ is quasi-isomorphic to 0 in $D^b(\cO_{\frX,\Q})$.
\end{lemma}
\begin{proof} In the following we consider $K_\bullet$ as an object in $\cP^b$ concentrated in degree zero. By hypothesis, the map $\varpi^N \, \cdot: K_\bullet \ra K_\bullet$ factorizes through the zero complex
                             $$ K_\bullet \ra 0  \ra K_\bullet \;.$$
              After applying  $R^k \varprojlim = \cH^k \circ R\varprojlim$, for $k \in \Z$,
             we find that multiplication with $\varpi^N$ factorizes
                            $$ R^k\varprojlim K_\bullet \ra 0  \ra R^k\varprojlim K_\bullet \;,$$
               meaning that for every $k\in \Z$, $\varpi^N  R^k\varprojlim K_\bullet=0$ and thus
                   $$ \Q \ot_{\Z} R^k\varprojlim K_\bullet = 0 \;.$$
 This proves the lemma, as $\Q$ is flat over $\Z$ and as  this module is the $k$-th cohomology sheaf
of the complex $$ \Q \ot_{\Z} R\varprojlim K_\bullet \;.$$
\end{proof}

\begin{lemma}\label{projlim} Let $N \in \Ne$, and let $\cE^\bullet$, $\cF^\bullet$ two objects of $\cP^b$, and $h: \cE^\bullet \ra \cF^\bullet$ a morphism in
$\cP^b$ so that the mapping cone $\cC^\bullet$ of $h$ (defined up to a quasi-isomorphism) satisfies
                 $$\forall j \in \Z \,: \; \varpi^N\cdot \cH^j (\cC^\bullet) = 0 \;.$$
           Then the map $h$ induces a quasi-isomorphism
                          $$ \Q \ot_{\Z} R\varprojlim (\cE^\bullet) \simeq  \Q \ot_{\Z} R\varprojlim (\cF^\bullet) \;.$$
\end{lemma}
\begin{proof} First note that since $\cE^\bullet$ and $\cF^\bullet$ have bounded cohomology, this is also the case
of $\cC^\bullet$, so that the condition of the lemma involves only a finite number of $j\in\Z$. More precisely, there
exists $J \in \Ne$, such that the sheaves $\cH^j(\cC^\bullet)$ are zero for all $j$ satisfying the
condition $|j|> J$. Using the
cohomological truncations functors $\sigma_{\geq n}$ as defined in ~\cite[I, 7]{HartshorneD},
and denoting $\sigma_{> n}$ of loc.
cit. by  $\sigma_{\geq n+1}$, we have for each $n$ a triangle in $\cP^b$ (\cite[I, 7.2]{HartshorneD})
$$ \cH^n (\cC^\bullet) \ra  \sigma_{\geq n}(\cC^\bullet)  \ra \sigma_{\geq n+1}(\cC^\bullet) \stackrel{+1}{\lra} \;.$$
We will prove by decreasing induction on $n$ that $\Q \ot R\varprojlim \sigma_{\geq n}(\cC^\bullet)$ is quasi-isomorphic to
$0$. This is true if $n=J+1$. Assume that this is true for $n+1$, then after applying $\Q \ot R\varprojlim$
to the previous triangle, we get a triangle
 $$  \Q \ot R\varprojlim \cH^n (\cC^\bullet) \ra \Q \ot R\varprojlim \sigma_{\geq n}(\cC^\bullet) \ra
 \Q \ot R\varprojlim \sigma_{\geq n+1}(\cC^\bullet) \sta{+1}{\lra} \;.$$
As by hypothesis $\varpi^N \cH^n (\cC^\bullet)=0$, we see by applying the previous lemma ~\ref{lemprojlim} to the projective
system $\cH^n (\cC^\bullet)$ that
$\Q \ot R\varprojlim \cH^n (\cC^\bullet)$ is quasi-isomorphic to $0$. Therefore, we have a quasi-isomorphism
$$ \Q \ot R\varprojlim \sigma_{\geq n}(\cC^\bullet) \simeq  \Q \ot R\varprojlim \sigma_{\geq n+1}(\cC^\bullet) \;,$$
and the complex $\Q \ot R\varprojlim \sigma_{\geq n}(\cC^\bullet)$ is hence quasi-isomorphic to $0$. Thus,
we conclude that for all $n$, $\Q \ot R\varprojlim \sigma_{\geq n}(\cC^\bullet)$ is quasi-isomorphic to $0$. Since  $\cC^\bullet$ has bounded cohomology, it is quasi-isomorphic
 to $\sigma_{\geq n}(\cC^\bullet)$ for $n$
small enough and this finally proves that $\Q \ot R\varprojlim \cC^\bullet$ is quasi-isomorphic to $0$. Now we consider the triangle in $\cP^b$
$$ \cE^\bullet \ra \cF^\bullet \ra \cC^\bullet \sta{+1}{\ra} \;,$$
and apply $\Q \ot R\varprojlim$. In this way we obtain a triangle
$$ \Q \ot R\varprojlim \cE^\bullet \ra \Q \ot R\varprojlim \cF^\bullet \ra \Q \ot R\varprojlim \cC^\bullet
 \sta{+1}{\ra} \;,$$
and since $\Q \ot R\varprojlim \cC^\bullet$ is quasi-isomorphic to $0$, we see that
the first map of the latter triangle is a quasi-isomorphism as claimed.
\end{proof}

As before, we denote the associated rigid analytic space of $\frX$ by $\frX_\Q$. From ~\ref{eq_catO} and the
lemma ~\ref{lemma_int} $R^j\pi_*\cO_{\frX',\Q}=0$ for $j>0$ and $\pi_*\cO_{\frX',\Q}=\cO_{\frX,\Q}$. As the map
$\pi$ is proper, the sheaves $R^j\pi_*\cO_{\frX'}$ are coherent $\cO_{\frX}$-modules and
there is $N\geq 0$ such that
 $\vpi^NR^j\pi_*\cO_{\frX'}=0$ for all $j>0$ and such that the kernel and cokernel of the natural map
 $\cO_{\frX}\rightarrow\pi_*\cO_{\frX'}$ are killed by $\vpi^N$ as well. For any $i\geq 0$, let as usual $X_i$
be the reduction of $\frX$
 mod $\vpi^{i+1}$ and similarly for $\frX'$ and denote by $\pi_i: X'_i\rightarrow X_i$ the morphism induced by $\pi$.
We will need the following
\begin{aux}\label{lemma_int2} \hskip0pt Let  $  i \geq 0$.
\begin{enumerate} \item  Kernel and cokernel
of the canonical map
$ \cO_{X_i}\ra \pi_{i*} \cO_{X'_i}$ are annihilated by $\varpi^{2N}$.
                            \item For all $j \geq 1$, one has $\varpi^{2N}R^j\pi_{i*}\cO_{X'_i}=0.$
              \end{enumerate}
\end{aux}
\begin{proof} As the formal scheme $\frX'$ is flat, there are exact sequences
$$ 0 \lra \cO_{\frX'}\sta{\varpi^{i+1}}{\lra} \cO_{\frX'} \lra \cO_{X'_i} \lra 0 \;.$$ Applying $R^j\pi_{*}$, we get
exact sequences for any $i$ and $j \geq 1$,
		$$R^j\pi_{*}\cO_{\frX'} \ra R^j\pi_{i*}\cO_{X'_i} \ra R^{j+1}\pi_{*}\cO_{\frX'} \;,$$
that prove that $\varpi^{2N}R^j\pi_{i*}\cO_{X'_i}=0$. Moreover we can consider the following commutative diagram
of exact sequences
		$$\xymatrix {  0 \ar@{->}[r] & \cO_{\frX}\ar@{->}[d]\ar@{->}[r]^{\varpi^{i+1}} &
 \cO_{\frX}\ar@{->}[r] \ar@{->}[d]& \cO_{X_i}\ar@{->}[d] \ar@{->}[r] & 0 \\
		0 \ar@{->}[r] & \pi_*\cO_{\frX'} \ar@{->}[r]^{\varpi^{i+1}} &  \pi_*\cO_{\frX'} \ar@{->}[r] &
 \pi_{i*} \cO_{X'_i}\ar@{->}[r] & R^1\pi_*\cO_{\frX'} .
}$$
By the snake lemma the kernel of the canonical map $ \cO_{X_i}\ra \pi_{i*} \cO_{X'_i}$ is killed by
$\varpi^{2N}$. By chasing the diagram we also see that the cokernel of this map is also killed by
$\varpi^{2N}$ for all $i$. This proves the auxiliary result.
\end{proof}

\begin{lemma} \label{prop-exactdirectimage-forD}
Let $\pi:\frX'\rightarrow\frX$ be a morphism over $\frX_0$ between admissible formal blow-ups of the smooth formal
scheme $\frX_0$. Let $k \geq \max\{k_{\frX},k_{\frX'}\}$.

\begin{enumerate}
\item Then we have: $R^j\pi_*\sD^{\dagger}_{\frX',k}=0$ for $j>0$.
				Moreover, $\pi_*\sD^{\dagger}_{\frX',k} = \sD^{\dagger}_{\frX,k}$.
\item There is a canonical isomorphism $\sD^{\dagger}_{\frX',k}\simeq \pi^!\sD^{\dagger}_{\frX,k} $.
\end{enumerate}

\begin{proof}
 Since $R^j\pi_*$ commutes with inductive limits,
 it suffices to prove the claim
 for $\hsD^{(k,m)}_{\frX',\Q}$. Abbreviate $\hsD_{\frX'}=\hsD^{(k,m)}_{\frX'}$ (and similarly for $\frX$), and
$\sD_{X'_i}=\hsD^{(m)}_{\frX',k}/\vpi^{i+1}\hsD^{(m)}_{\frX',k}$ (and similarly for $X_i$).
We need to compute $R\pi_* \hsD_{\frX'}$. Note that by ~\cite[Lemma 20.32.4]{stacks-project}
$R\varprojlim_i \sD_{X'_i}\simeq \varprojlim_i \sD_{X'_i}$, so that
$$R\pi_* \hsD_{\frX'} \simeq R\pi_*R\varprojlim_i \sD_{X'_i}
                                     \simeq R\varprojlim_i  R\pi_{i*}\sD_{X'_i} \;,$$

by ~\cite[Lemma 20.32.2]{stacks-project}. As the sheaf $\sD_{X_i}$ is a flat $\cO_{X_i}$-module, the projection formula
gives a canonical isomorphism
$$R\pi_{i*}\sD_{X'_i}\simeq R\pi_{i*}\cO_{X'_i} \otimes_{\cO_{X_i}}\sD_{X_i}\;,$$
so that the canonical map $\cO_{X_i}\ra  R\pi_{i*}\cO_{X'_i}$ induces a map of projective systems of
complexes $h: (\sD_{X_i})_i \ra (R\pi_{i*}\sD_{X'_i})_i$. We consider these projective systems as objects of $\cP^b$. By applying $R\varprojlim_i$ to $h$, we get the canonical map $\widehat{h}: \hsD_{\frX} \ra R\pi_* \hsD_{\frX'}$.
Moreover, we have
$$ \forall j \geq 0 \; \forall i \;\; R^j\pi_{i*}\sD_{X'_i}\simeq R^j\pi_{i*}\cO_{X'_i}\ot_{\cO_{X_i}}\sD_{X_i} \;.$$
By flat base change from $\cO_{X_i}$ to $\sD_{X_i}$, the previous auxiliary result~\ref{lemma_int2} (i) implies that the kernel and the cokernel of the map
$(\sD_{X_i})_i \ra (\pi_{i*}\sD_{X'_i})_i$ of projective systems are annihilated by $\varpi^{2N}$. Similarly, by  \ref{lemma_int2} (ii) the projective systems $(R^j\pi_{i *}\sD_{X'_i})_i$ for $j\geq 1$ are annihilated by $\varpi^{2N}$. Let $\cC^\bullet$ be the cone of
$h$, then, as the functor $\cH^0$ is a cohomological functor \cite[definition, p.27]{HartshorneD} we have the
following exact cohomology sequence of projective systems of sheaves
$$ 0 \ra (\cH^{-1}(\cC^\bullet)) \ra (\sD_{X_i})_i \ra (\pi_{i*}\sD_{X'_i})_i \ra (\cH^{0}(\cC^\bullet)) \ra 0, $$
and $\forall j \geq 1$
$$ (R^j\pi_{i*}(\sD_{X_i}))_i \simeq \cH^j(\cC^\bullet) \;.$$
We thus see that the cohomology of $\cC^\bullet$ is annihilated by $\varpi^{2N}$, so that we can apply
lemma~\ref{projlim} and obtain a quasi-isomorphism $\widehat{h} \otimes \Q: \hsD_{\frX,\Q} \car R\pi_* \hsD_{\frX',\Q}$.
By passing to the cohomology sheaves (and to the inductive limit over all $m$), this proves (i). The part (ii) follows from the definition of the functor $\pi^{!}$, cf. \ref{inversefunctor}.
\end{proof}

\end{lemma}

We can now state the
\begin{thm}\label{prop-exactdirectimage}
Let $\pi:\frX'\rightarrow\frX$ be a morphism over $\frX_0$ between admissible formal blow-ups of the smooth formal
scheme $\frX_0$. Let $k \geq \max\{k_{\frX},k_{\frX'}\}$.

\begin{enumerate}
\item
If $\sM$ is a coherent $\sD^{\dagger}_{\frX',k}$-module, then $R^j\pi_*\sM=0$ for $j>0$.
				Moreover, $\pi_*\sD^{\dagger}_{\frX',k} = \sD^{\dagger}_{\frX,k}$,
				so that $\pi_*$ induces an exact functor between coherent modules over
$\sD^{\dagger}_{\frX',k}$ and $\sD^{\dagger}_{\frX,k}$ respectively.

\item The formation $\pi^!$ is an exact functor from the category of coherent $\sD^{\dagger}_{\frX,k}$-modules
to the category of coherent $\sD^{\dagger}_{\frX',k}$-modules, and $\pi^!$ and $\pi_*$ are quasi-inverse
 equivalences between these categories.
\end{enumerate}

The same statement holds for coherent modules over $\hsD^{(k,m)}_{\frX,\Q}$ and $\hsD^{(k,m)}_{\frX',\Q}$ respectively.
\end{thm}

\begin{proof}
 The first assertion of part (i) is true for  $\sD^{\dagger}_{\frX',k}$ by the previous
lemma ~\ref{prop-exactdirectimage-forD}. Now there is a basis of the topology
of $\frX$ consisting of affine opens $\frV$ such that $\pr(\frV)$ is contained in some affine open of
$\frX_0$. For this reason, if some statement is local on $\frX$, then we can
assume that $\frX_0$ is affine.
To prove (i) of the theorem for coherent $\sD^{\dagger}_{\frX',k}$-modules,
 we can thus assume (and we do assume) that $\frX_0$ is affine.
Then $\frX$ and $\frX'$ are admissible blow-ups of a smooth affine formal scheme $\frX_0$ and we still
call $\pi$ the map $\frX' \ra \frX$.
We consider now the following assertion depending on $j$:
$$\textrm{ For any coherent }\sD^{\dagger}_{\frX',k}\textrm{-module }\sM,\textrm{ and for any }l\geq j,
            \textrm{ one has }   R^l\pi_*\sM=0 \;.$$
We will prove this assertion for any $j\geq 1$ by decreasing induction on $j$. The statement for $j=1$ establishes then the first assertion of part (i) in general.
Since $\frX'$ has dimension $\leq M+1$, the assertion is true for $j=M+2$. Assume that the assertion is true
for $j+1$ and take a coherent $\sD^{\dagger}_{\frX',k}$-module $\sM$. Since $\frX_0$ is affine, we can apply ~\ref{Xaf} to find a
non negative integer $r$ and an exact sequence of coherent $\sD^{\dagger}_{\frX',k}$-modules
$$ 0 \ra \sN \ra (\sD^{\dagger}_{\frX',k})^r \ra \sM \ra 0 \;.$$
Since $j \geq 1$ and since $R^{j}\pi_* \sD^{\dagger}_{\frX',k}=0$ by
lemma ~\ref{prop-exactdirectimage-forD}, the long exact sequence for $\pi_*$ gives us an isomorphism
$$ R^{j}\pi_* \sM \simeq R^{j+1}\pi_* \sN \;.$$
But the right-hand side is zero by the induction hypothesis applied to $\sN$. This establishes the assertion for $j$ and completes the induction step. This ends the proof of the first assertion of part (i).

What remains to prove for part (i) is that $\pi_*\cM$ is coherent over $\sD^{\dagger}_{\frX,k}$ if $\cM$ is coherent over $\sD^{\dagger}_{\frX',k}$. To show this, we continue to assume that $\frX_0$ is affine. By \ref{Xaf}, there is a finite presentation
$$(\sD^{\dagger}_{\frX',k})^s \ra (\sD^{\dagger}_{\frX',k})^r \ra \sM \ra 0 \;.$$
Applying $\pi_*$ and using that $\pi_*$ is exact, we obtain a finite presentation for $\pi_* \cM$, which implies that the latter is coherent.

\vskip8pt

Let us prove part (ii) in the case of
coherent $\sD^{\dagger}_{\frX,k}$-modules, the case of coherent $\hsD^{(k,m)}_{\frX,\Q}$-modules can be treated analogously. By definition of the functor $\pi^{!}$, cf. \ref{inversefunctor}, we have $\pi^! \sD^{\dagger}_{\frX,k}=\sD^{\dagger}_{\frX',k}$.
 To prove that $\pi^!$ preserves coherence is local on $\frX$, so that we can (and will) again assume
that $\frX_0$ is affine. Let $\sM$ be a coherent $\sD^{\dagger}_{\frX,k}$-module. We can apply proposition ~\ref{Xaf} to $\frX$ and obtain a finite presentation
$$  \left(\sD^{\dagger}_{\frX,k}\right)^s \ra \left(\sD^{\dagger}_{\frX,k}\right)^r \ra \sM \ra 0 \;.$$
 Since the tensor product is right exact, we get a finite presentation of $\pi^!\sM$
$$  \left(\sD^{\dagger}_{\frX',k}\right)^s \ra \left(\sD^{\dagger}_{\frX',k}\right)^r \ra \pi^!\sM \ra 0 \;,$$
which implies that $\pi^!\sM$ is a coherent $\sD^{\dagger}_{\frX',k}$-module. In particular, the functor $\pi^!$ preserves coherence. The map $\pi^{-1}\sM\rightarrow \pi^!\sM$ sending $x $ to $1 \otimes x$ induces a morphism $$can_\sM: \sM \rightarrow \pi_*\pi^!\sM \;,$$
which is natural in $\sM$. Whether $can_\sM$ is an isomorphism can be decided locally on $\frX$, and we can again assume that
$\frX_0$ is affine. $\frX$ and $\frX'$ are admissible blow-ups of $\frX_0$, and by \ref{Xaf} there a finite presentation
\begin{numequation}\label{1stpres} \left(\sD^{\dagger}_{\frX,k}\right)^s \ra \left(\sD^{\dagger}_{\frX,k}\right)^r \ra \sM \ra 0 \;,
\end{numequation}
and so $\pi^!\sM$ admits a finite presentation
$$  \left(\sD^{\dagger}_{\frX',k}\right)^s \ra \left(\sD^{\dagger}_{\frX',k}\right)^r \ra \pi^!\sM \ra 0 \;.$$
We apply $\pi_*$ to this latter sequence and use that $\pi_*$ is exact (by (i)), together with \ref{prop-exactdirectimage-forD} to obtain the finite presentation
\begin{numequation}\label{2ndpres} \left(\sD^{\dagger}_{\frX,k}\right)^s \ra \left(\sD^{\dagger}_{\frX,k}\right)^r \ra \pi_* \pi^!\sM \ra 0 \;.
\end{numequation}
The natural transformation $can$ induces a morphism from \ref{1stpres} to \ref{2ndpres}, and because $can_{\sD^{\dagger}_{\frX,k}}$ is an isomorphism, so is $can_\sM$.

In the reverse direction, let $\sM'$ be a coherent $\sD^{\dagger}_{\frX',k}$-module.
 There is a map $can'_{\sM'}: \pi^!\pi_*\sM'\rightarrow \sM'$, sending $P\otimes x$ to $Px$, which is natural in $\sM'$. Whether this map is bijective can be decided locally on $\frX$ and we may assume that $\frX_0$ is affine, and $\frX'$ and $\frX$ are admissible blow-ups
of $\frX_0$. Since
$\pi_*$ is exact, and using ~\ref{Xaf} over $\frX'$, we are reduced
to the case where $\sM'=\sD^{\dagger}_{\frX',k}$. In this case $\pi^!\pi_*\sD^{\dagger}_{\frX',k}\simeq
\sD^{\dagger}_{\frX',k}$ by (i).
From all of this, we can conclude that $\pi^!$ and $\pi_*$ are quasi-inverse functors. As $\pi_*$ is exact on coherent
$\sD^{\dagger}_{\frX',k}$-modules, $\pi^!$ is exact on coherent $\sD^{\dagger}_{\frX,k}$-modules. \end{proof}

\begin{cor} \label{globcalc} In the situation of the preceding theorem, one has

$$\Ga(\frX,\sD^{\dagger}_{\frX,k})=\Ga(\frX_0,\sD^{\dagger}_{\frX_0,k})
		                          =\Ga(\frX',\sD^{\dagger}_{\frX',k}) \;.$$

\end{cor}

\vskip8pt

As an application of the invariance theorem we can extend the local theorems A and B \ref{easy_thmAB} and \ref{easy_thmAB2} to  global statements, provided that the base $\frX_0$ is affine.\footnote{Note that this is not covered by \ref{easy_thmAB} and \ref{easy_thmAB2}, since an admissible blow-up $\frX$ of an affine $\frX_0$ is in general not affine.}.

\begin{thm}\label{thmA}\label{thmB} {\rm (Global theorem A and B over an affine base)} Let $\frX_0$ be affine.

\begin{enumerate}

 \item
    For any coherent $\hsD^{(k,m)}_{\frX,\Q}$-module $\sM$ and for all $q>0$ one has $H^q(\frX,\sM)=0$.

 \item The functor $\Ga(\frX,-)$ is an equivalence between the category of coherent $\hsD^{(k,m)}_{{\frX,\Q}}$-modules and the category of coherent $\Ga(\frX,\hsD^{(k,m)}_{\frX,\Q})$-modules.
     \end{enumerate}

      The same statement holds for coherent modules over $\sD^{\dagger}_{\frX,k}$ and $\Ga(\frX,\sD^{\dagger}_{\frX,k})$.
\end{thm}
\begin{proof}
Denote by $\pi: \frX\rightarrow\frX_0$ the blow-up morphism. The functor $\Ga(\frX,.)$ equals the
composite of the two functors $\pi_*$ and $\Ga(\frX_0,-)$. Hence the theorem follows from \ref{prop-exactdirectimage} and its corollary \ref{globcalc} and \ref{easy_thmAB2}.
\end{proof}

\section{Coadmissible \texorpdfstring{$\sD$}{}-modules on \texorpdfstring{$\frX$}{} and the Zariski-Riemann space}

We continue to denote throughout this section by $\frX_0$ a smooth formal $\frS$-scheme, and we consider an admissible formal blow-up $$\pr:\frX \ra \frX_0 \;.$$
\vskip8pt

The purpose of the first subsection is to study projective systems $(\sM_k)_{k \ge k_\frX}$ of coherent modules $\sM_k$ over $\sD^\dagger_{\frX,k}$, and to pass to their associated projective limits. In the second subsection we will then let $\frX$ vary in the system of all admissible formal blow-ups of $\frX_0$.

\subsection{Coadmissible \texorpdfstring{$\sD$}{}-modules on \texorpdfstring{$\frX$}{}}

We make the general convention that $k$ always denotes an integer which is at least as large as $k_\frX$.
\begin{para}{\it Fr\'echet-Stein algebras.} Let $B$ be a noetherian $K$-Banach algebra. We recall that any finitely generated $B$-module has a canonical structure as $B$-Banach module and any $B$-linear map between two such modules is continuous. The topology can be defined as the quotient topology with respect to any chosen finite presentation of the module \cite[Prop. 2.1]{ST03}.

We recall from \cite[sec. 3]{ST03} that a $K$-Fr\'echet algebra $A$ is called {\it Fr\'echet-Stein } if there is a projective system $(A_k, A_{k+1} \ra A_k)_{k \in \Ne}$ of (left) noetherian $K$-Banach algebras $A_k$ with (right) flat transition maps $A_{k+1} \ra A_k$, and an isomorphism of topological $K$-algebras $A \simeq \varprojlim_k A_k$. For the following definition we fix such an isomorphism. We denote by $\cC_A$ the full abelian subcategory of the category of all (left) $A$-modules consisting of the coadmissible $A$-modules, as introduced in \cite{ST03}. For an $A$-module $M$ to be coadmissible means that there is a projective system $(M_k, M_{k+1} \ra M_k)_{k \in \Ne}$, where each $M_k$ is a finitely generated $A_k$-module, such that

\vskip5pt

\hskip10pt (i) the transition map $M_{k+1} \ra M_k$ is a homomorphism of $A_{k+1}$-modules, and the induced map $A_{k}\otimes_{A_{k+1}} M_{k+1} \ra M_k$ is an isomorphism of $A_k$-modules, and

\vskip5pt

\hskip10pt (ii) $M$ is isomorphic to $\varprojlim_k M_k$ as an $A$-module.

\vskip8pt

(The projective limit is considered as an $A$-module via the fixed isomorphism\footnote{However, the category $\cC_A$ is independent of the isomorphism.} $A \simeq \varprojlim_k A_k$.)
We sometimes call $(M_k)$ an $(A_k)$-sequence for $M$. For $M= \varprojlim_k M_k\in \cC_A$ we have that the image of $M\rightarrow M_k$ is dense with respect to the canonical topology for any $k$, and $\varprojlim^{(1)}_k M_k=0$, cf. \cite[first theorem in sec. 3]{ST03}.

\end{para}

\begin{para}{\it The sheaf $\sD_{\frX,\infty}$.} We denote by

\begin{numequation}\label{D_infty_proj_lim}
\sD_{\frX,\infty}=\varprojlim_k \sD^\dagger_{\frX,k}
\end{numequation}

the projective limit of the system of sheaves $\sD^\dagger_{\frX,k}$. Then $\sD_{\frX,\infty}$ is again a sheaf of rings and
for every open subset $\frV \sub \frX$ we have $\sD_{\frX,\infty}(\frV) = \varprojlim_k \sD^\dagger_{\frX,k}(\frV)$.
\end{para}

\begin{prop}\label{prop_local_description} \hskip0pt \begin{enumerate}

\item The canonical morphism of sheaves $\hsD^{(k,0)}_{\frX,\Q} \ra \sD^\dagger_{\frX,k}$ induces an isomorphism

$$\varprojlim_k \hsD^{(k,0)}_{\frX,\Q} \stackrel{\simeq}{\lra}
\varprojlim_k \sD^\dagger_{\frX,k} = \sD_{\frX,\infty} \;.$$

\vskip8pt

\item For every affine open subset $\frV \sub \frX$ the isomorphism $$\sD_{\frX,\infty}(\frV) = \varprojlim_k \hsD^{(k,0)}_{\frX,\Q}(\frV)$$ induced from (i) gives $\sD_{\frX,\infty}(\frV)$ the structure of a Fr\'echet-Stein algebra.

\vskip8pt

\item Let $\frU \sub \frX_0$ be an open affine subset which can be equipped with a system of \'etale coordinates $x_1, \ldots, x_M$ and $\partial_1, \ldots, \partial_M$ the corresponding derivations. Then, for any affine open $\frV \sub  \pr^{-1}(\frU)$ we have
\begin{numequation}\label{formula_local_description}
\sD_{\frX,\infty}(\frV) = \left\{\sum_{\unu} a_{\unu} \uder^{\unu} \, \Big| \, a_{\unu} \in \cO_{\frX,\Q}(\frV) \, , \; \forall R>0: \lim_{|\unu| \ra \infty} \|a_{\unu}\|R^{|\unu|} = 0 \right\} \;,
\end{numequation}
where $\| \cdot \|$ is any submultiplicative Banach space norm on $\cO_{\frX,\Q}(\frV)$.

\vskip8pt



\item Let $\frX' \ra \frX_0$ be another admissible formal blow-up, and let $\pi: \frX' \ra \frX$ be a morphism over $\frX_0$. Then the canonical
		isomorphisms $\pi_* \sD^\dagger_{\frX',k} = \sD^\dagger_{\frX,k}$, for $k \ge \max\{k_{\frX'},k_\frX\}$, cf. \ref{prop-exactdirectimage},
				give rise to a canonical isomorphism

\begin{numequation}\label{can_iso_D_infty}
            \pi_* \sD_{\frX',\infty} = \sD_{\frX,\infty} \;.
\end{numequation}
\end{enumerate}
\end{prop}

\begin{proof} 
Take an affine open $\frV \subset \frX$.
		We deduce from the proposition ~\ref{comp_level} that the projective systems of $K$-algebras
		 $\sD^{\dagger}_{\frX,k}(\frV)$ and $\hsD^{(k,0)}_{\frX,\bbQ}(\frV)$ are equivalent. This proves (i). For (ii), note that the transition map $ \hsD^{(k+1,0)}_{\frX,\Q}(\frV)\rightarrow \hsD^{(k,0)}_{\frX,\Q}(\frV)$ is a flat homomorphism between noetherian Banach algebras, according to propositions \ref{flatFS} and \ref{finite_tDm}. To prove (iii), we assume additionally $\frV \subset \pr^{-1}\frU$. By part (i), is then enough to show
              $$  \varprojlim_k  \hsD^{(k,0)}_{\frX,\bbQ}(\frV)  = \left\{\sum_{\unu} a_{\unu} \uder^{\unu} \, \Big| \, a_{\unu} \in \cO_{\frX,\Q}(\frV)
			  \, , \; \forall R>0: \lim_{|\unu| \ra \infty} \|a_{\unu}\|R^{|\unu|} = 0 \right\} \;.$$
			  Denote by $E$ the right-hand side of the preceding equality.
    Recall that
		$$ \hsD^{(k,0)}_{\frX,\bbQ}(\frV)= \left\{\sum_{\unu}\vpi^{k|\unu|}b_{\unu}\uder^{\unu}\,|\,b_{\unu}\in
		\cO_{\frX,\Q}(\frV), \|b_{\unu}\|\ra 0\right\}\;.$$
		Let $P=\sum_{\unu}a_{\unu} \uder^{\unu}\,\in E$. Since all Banach algebra norms over $\cO_{\frX,\Q}(\frV)$ are equivalent, we can use
		the $p$-adic norm $|\cdot |_p$ of $\cO_{\frX,\Q}(\frV)$ relatively to the lattice $\cO_{\frX}(\frV)$. Fix $k\in\Ne$ and define
		$$b_{\unu}=\vpi^{-k|\unu|}a_{\unu}\;,$$ then, using the ramification
		index $e$ of the extension $L/\Qp$, we get that
		$$ |b_{\unu}|_p=|a_{\unu}|_p p^{\frac{k|\unu|}{e}} \ra 0 \; \textrm{ if } |\unu|\ra +\infty. $$
		Thus $P=\sum_{\unu}\vpi^{k|\unu|}b_{\unu}\uder^{\unu}\in \hsD^{(k,0)}_{\frX,\bbQ}(\frV)$. Conversely, let
		$P=\sum_{\unu} a_{\unu} \uder^{\unu}\,\in \varprojlim_k  \hsD^{(k,0)}_{\frX,\bbQ}(\frV)$ and $R>0$. Choose $k>0$
		such that $$p^{\frac{k}{e}}>R \textrm{ and define }  b_{\unu}=\vpi^{-k|\unu|}a_{\unu}\;.$$ Since
		$P=\sum_{\unu}\vpi^{k|\unu|}b_{\unu}\uder^{\unu}\,\in
		\hsD^{(k,0)}_{\frX,\bbQ}(\frV)$, $|b_{\unu}|_p\ra 0$, thus
		$$ |a_{\unu}|_p p^{\frac{k|\unu|}{e}} \ra 0, \textrm{ and }|a_{\unu}|_p R^{|\unu|}\ra 0 \;,$$
	proving that $P \in E$, as required for (iii).

		Let us prove (iv). Let $\frV\subset \frX$, then from part (i) of proposition \ref{prop-exactdirectimage}, we know that
		$\sD^{\dagger}_{\frX,k}(\frV)=\sD^{\dagger}_{\frX',k}(\pi^{-1}(\frV))$. We deduce from this the
		equations
$$\pi_*\sD_{\frX',\infty}(\frV) = \sD_{\frX',\infty}(\pi^{-1}(\frV))  = \varprojlim_k \sD^{\dagger}_{\frX',k}(\pi^{-1}(\frV)) = \varprojlim_k \sD^{\dagger}_{\frX,k}(\frV) = \sD_{\frX,\infty}(\frV) \;.$$ \end{proof}

For every affine open subset $\frV \sub \frX$ we have, according to the preceding proposition, the abelian category of coadmissible
 $\sD_{\frX,\infty}(\frV)$-modules $\cC_{\sD_{\frX,\infty}(\frV)}$. We give an alternative description of these modules using the projective system of algebras
$\sD^{\dagger}_{\frX,k}(\frV)$.

\begin{prop}\label{prop_coadmissible} Let $\frV \sub \frX$ be an affine open. A $\sD_{\frX,\infty}(\frV)$-module $M$ is coadmissible if and only if there is a projective system $(M_k, M_{k+1} \ra M_k)$, where each $M_k$ is a finitely presented $\sD^{\dagger}_{\frX,k}(\frV)$-module, such that

\vskip5pt

\hskip10pt (i) the transition map $M_{k+1} \ra M_k$ is $\sD^{\dagger}_{\frX,k+1}(\frV)$-linear and induces an isomorphism $$\sD^{\dagger}_{\frX,k}(\frV)\otimes_{\sD^{\dagger}_{\frX,k+1}(\frV)} M_{k+1} \simeq M_k$$ of $\sD^{\dagger}_{\frX,k}(\frV)$-modules, and

\vskip5pt

\hskip10pt (ii) $M$ is isomorphic to $\varprojlim_k M_k$ as an $\sD_{\frX,\infty}(\frV)$-module.

\end{prop}

\begin{proof}
This follows from the discussion given in the proof of \cite[Prop. 1.2.9]{EmertonA} and we explain the main points. We write
$A_k=\hsD^{(k,0)}_{\frX,\bbQ}(\frV)$ and $B_k= \sD^{\dagger}_{\frX,k}(\frV)$. We have strictly increasing functions $\phi$ and $\psi$ mapping $\bbN$ to itself such that
the map $A\rightarrow A_k$ (resp. $A\rightarrow B_k$) factors through the map $A\rightarrow B_{\phi(k)}$ (resp. $A\rightarrow A_{\psi(k)}$). Indeed, we may take
$\phi(k)=k+e'$ and $\psi(k)=k$ with $e'\geq \frac{e}{p-1}$ a fixed number, according to the proposition \ref{comp_level}. In particular, the systems of $K$-algebras $A_k$ and $B_k$ are equivalent. Now suppose $N$ is coadmissible with system of Banach modules $N_k$. For $k\geq 1$ define the finitely presented $B_k$-module
$$M_k = B_k \otimes_{A_{\psi(k)}} N_{\psi(k)} \;.$$ Since $\psi(k+1)\geq \psi(k)$, the map $N_{\psi(k+1)}\rightarrow N_{\psi(k)}$ induces a map $M_{k+1}\rightarrow M_k$ and then a map
$B_k\otimes_{B_{k+1}} M_{k+1} \rightarrow M_k$. This map is bijective as follows from the diagram presented in the proof of \cite[Prop. 1.2.9]{EmertonA}. Moreover, the projection
map $M\rightarrow N_{\psi(k)}$ induces a $A$-linear map
$$ M \rightarrow B_k \otimes_A M \rightarrow B_k  \otimes_{A_{\psi(k)}} N_{\psi(k)}=M_k $$
compatible with $M_{k+1}\rightarrow M_k$. This gives an $A$-linear map $M\rightarrow \varprojlim_k M_k$ and it remains to see that it is bijective.
We have the natural map $ N_{\psi(k)}\rightarrow M_k$. On the other hand, $\psi(\phi (k))\geq k$ such that there is a map
$$ M_{\phi(k)}= B_{\phi(k)} \otimes_{A_{\psi(\phi(k))}} N_{\psi(\phi(k))}\rightarrow A_k \otimes_{A_{\psi(\phi(k))}} N_{\psi(\phi(k))} \rightarrow N_k$$
using the map $B_{\phi(k)} \rightarrow A_k$. Hence the systems $(M_k)$ and $(N_k)$ are equivalent and $$M\simeq \varprojlim_k N_k \simeq \varprojlim_k M_k \;.$$
This shows that the system $(M_k)$ is as required for $M$. Conversely, starting with a module $M$ and such a system $(M_k)$ the coadmissibility of $M$ follows with the same argument.
\end{proof}




\begin{dfn}\label{dfn_coadm_frX} A $\sD_{\frX,\infty}$-module $\sM$ is called {\it coadmissible} if there  a projective system $(\sM_k, \sM_{k+1} \ra \sM_k)_{k \ge k_\frX}$, where $\sM_k$ is a coherent $\hsD^{(k,0)}_{\frX,\Q}$-module, such that

 \vskip5pt

\hskip10pt (i) the transition map $\sM_{k+1} \ra \sM_k$ is $\hsD^{(k+1,0)}_{\frX,\Q}$-linear and the induced map
\begin{numequation}\label{transition_2}
\hsD^{(k,0)}_{\frX,\Q}\otimes_{\hsD^{(k+1,0)}_{\frX,\Q}} \sM_{k+1} \lra \sM_k
\end{numequation}
is an isomorphism as $\hsD^{(k,0)}_{\frX,\Q}$-modules, and

\vskip5pt

\hskip10pt (ii) $\sM$ is isomorphic to $\varprojlim_k \sM_k$ as $\sD_{\frX,\infty}$-module.
\end{dfn}

We denote by
$$\cC_\frX \subseteq {\rm Mod}(\sD_{\frX,\infty})$$
the full subcategory of coadmissible $\sD_{\frX,\infty}$-modules in the category of all $\sD_{\frX,\infty}$-modules.

\begin{prop}\label{lemma-weak-to-strong_FS} A $\sD_{\frX,\infty}$-module $\sM$ is coadmissible if and only if a projective system
		$(\sM_k, \sM_{k+1} \ra \sM_k)_{k \ge k_\frX}$, where $\sM_k$ is a coherent $\sD^\dagger_{\frX,k}$-module, such that

 \vskip5pt

\hskip10pt (i) the transition map $\sM_{k+1} \ra \sM_k$ is $\sD^\dagger_{\frX,k+1}$-linear and the induced morphism of sheaves
		
\begin{numequation}\label{transition_1}
\sD^\dagger_{\frX,k} \otimes_{\sD^\dagger_{\frX,k+1}} \sM_{k+1} \lra \sM_k
\end{numequation}

is an isomorphism of $\sD^\dagger_{\frX,k}$-modules, and

 \vskip5pt

\hskip10pt (ii)  $\sM$ is isomorphic to $\varprojlim_k \sM_k$ as $\sD_{\frX,\infty}$-module.
\end{prop}
\begin{proof} This follows literally as for the modules of local sections, cf. proposition \ref{prop_coadmissible},
 taking into account that proposition \ref{comp_level} holds on the level of sheaves.
\end{proof}

\vskip8pt

\begin{thm}\label{thm_A_for_frX_infty} (Theorem A for coadmissible modules on $\frX$) Let $\frV\subset\frX$ be an affine open subset.
Then the global sections functor $\Ga(\frV, -)$ induces an equivalence of categories $$\Ga(\frV, -): \cC_\frV\stackrel{\simeq}{\longrightarrow} \cC_{\sD_{\frX,\infty}(\frV)} \;.$$
\end{thm}

\begin{proof} This is an application of proposition \ref{easy_thmAB2}.
Abbreviate $D^{\dagger}_{k}=\Ga(\frV,\sD^{\dagger}_{\frV,k})$.
Let $\sM\in  \cC_\frV$ with coherent $\sD^\dagger_{\frX,k}$-modules $\sM_k$.
Then $M_k=\Ga(\frV,\sM_k)$ is a coherent $D^{\dagger}_{k}$-module and $M=\Ga(\frV,\sM)=  \varprojlim_{k} M_k$.
Taking global sections in the isomorphism
\begin{numequation}\label{star1}\sD^\dagger_{\frX,k} \otimes_{\sD^\dagger_{\frX,k+1}} \sM_{k+1} \stackrel{\simeq}{\lra} \sM_k
\end{numequation} shows that the canonical map
 $$\sD^{\dagger}_{k}\otimes_{\sD^{\dagger}_{k+1}} M_{k+1} \stackrel{\simeq}{\lra} M_k$$ is an isomorphism too. Indeed, this is clear in the case
$\sM_{k+1}=\sD^{\dagger}_{\frV,k+1}$ and the general case follows from taking a finite presentation of $\sM_{k+1}$ as $\sD^{\dagger}_{\frV,k+1}$-module. We conclude with proposition \ref{prop_coadmissible} that $M\in \cC_{\sD_{\frX,\infty}(\frV)}$.
Conversely, given $M\in \cC_{\sD_{\frX,\infty}(\frV)}$ with coherent $D^{\dagger}_{k}$-modules $M_k$ the $\sD^{\dagger}_{\frV,k}$-module
$\sM_k= \sD^{\dagger}_{\frV,k} \otimes_{D^{\dagger}_{k}} M_k$ is coherent and these modules satisfy \ref{star1}. Indeed, that the canonical map \ref{star1} is an isomorphism can be checked on global sections and follows then
from the compatibility with the tensor product. This shows that $\sM= \varprojlim_{k} \sM_k$ lies in $\cC_\frV$. This provides a quasi-inverse to $\Ga(\frV,-)$. \end{proof}

\begin{lemma}\label{lemma-top_ML} Let $\sM\in \cC_\frX$ with sequence $(\sM_k)$ of coherent $\hsD^{(k,0)}_{\frX,\Q}$-modules. Let $\frV \sub \frX$ be an open affine subset.
Then the projective system $(\sM_k(\frV))_{k \ge k_\frX}$ has the following properties:

\vskip8pt

\begin{enumerate}
\item For $k' \ge k$ the transition map $\sM_{k'}(\frV) \ra \sM_k(\frV)$ is uniformly continuous.

\vskip8pt

\item For all $k \ge  k_\frX$ there exists $k' \ge k$ such that for all $k'' \ge k'$ the image of $\sM_{k''}(\frV) \ra \sM_k(\frU)$ is dense in $\im(\sM_{k'}(\frV) \ra \sM_k(\frV))$.

\vskip8pt

\item $\varprojlim^{(1)}_k \sM_k(\frV) = 0$.
\end{enumerate}
\end{lemma}

\begin{proof} Since a continuous map between Banach spaces is uniformly continuous, (i) is clear.
Abbreviate $D_{k}=\Ga(\frV,\hsD^{(k,0)}_{\frX,\Q})$ and $M_k=\Ga(\frV,\sM_k)$.
By the above theorem $M\in\cC_{\sD_{\frX,\infty}(\frV)}$ and, by the general properties of Fr\'echet-Stein algebras which we have recalled above, it remains to see that the modules $M_k=\sM_k(\frV)$ form a $(D_k)$-sequence for $M$. This is an application of proposition \ref{easy_thmAB}:
First, $M_k$ is a coherent $D_{k}$-module. Applying $\Ga(\frV,-)$ to the isomorphism
$$\hsD^{(k,0)}_{\frX,\Q} \otimes_{\hsD^{(k+1,0)}_{\frX,\Q}} \sM_{k+1} \stackrel{\simeq}{\lra} \sM_k$$ shows that the canonical map
$$\sD_{k}\otimes_{\sD_{k+1}} M_{k+1} \stackrel{\simeq}{\lra} M_k$$ is an isomorphism.
 Indeed, this is clear in the case where the restriction of
$\sM_{k+1}$ to $\frV$ equals $\hsD^{(k,0)}_{\frV,\Q}$ and the general case follows from taking a finite presentation of it as $\hsD^{(k,0)}_{\frV,\Q}$-module.
\end{proof}


\begin{thm}\label{thm_B_for_frX_infty} (Theorem B for coadmissible modules on $\frX$) For every coadmissible $\sD_{\frX,\infty}$-module $\sM$ and for all $q>0$ one has $H^q(\frX,\sM) = 0$. Moreover, if $\sM = \varprojlim_k \sM_k$ with coherent $\hsD^{(k,0)}_{\frX,k}$-modules $\sM_k$, then $R\varprojlim_k \sM_k  = \sM$.
\end{thm}

\begin{proof} Our aim is to apply \cite[20.32.4]{stacks-project}. To this end, we let $\cB$ be the set of open affine subsets $\frV \sub \frX$ which are contained in a subset of the form $\pr^{-1}(\frU)$ for some open affine $\frU\sub \frX_0$. We are going to show that the three hypotheses of loc.cit. are fulfilled, namely

\vskip8pt

\begin{enumerate}
\item Every open subset of $\frX$ has a covering whose members are elements of $\cB$.

\item For every $\frV \in \cB$, all $k \ge 0$, and all $q>0$ one has $H^q(\frV,\sM_k) = 0$.

\item For every $\frV \in \cB$ one has $\varprojlim^{(1)}_k \sM_k(\frV) = 0$.
\end{enumerate}

{\it Proof of these conditions.} (i) This is true because $\cB$ is a basis of the topology of $\frX$.

\vskip8pt

(ii) This is true by \ref{easy_thmAB} (ii).

\vskip8pt

(iii) This is true by \ref{lemma-top_ML}.
\qed

\vskip8pt

This shows that the conclusions of \cite[20.32.4]{stacks-project} hold, namely that $R\varprojlim_k \sM_k  = \sM$, and that $H^q(\frX,\sM) = 0$ for all $q>0$.
\end{proof}

\begin{cor}\label{abelian}
The category $\cC_\frX$ is abelian.
\end{cor}
\begin{proof}
Let $f_k:\sM\rightarrow \sN$ be a morphism in  $\cC_\frX$ and denote by $\sK=\ker(f)$ and $\sI= \im(f)$ its kernel and image in ${\rm Mod}(\sD_{\frX,\infty})$. By left-exactness of the projective limit, we have $\sK=  \varprojlim_{k} \sK_k$ where $\sK_k=\ker(f_k)$ and a linear map
\begin{numequation}\label{star2}\sD^\dagger_{\frX,k} \otimes_{\sD^\dagger_{\frX,k+1}} \sK_{k+1}\lra \sK_k \;.
\end{numequation}
Let $\frV\subseteq \frX$ be an open affine. By the preceding theorem, we know that $\sK(\frV)\in  \cC_{\sD_{\frX,\infty}(\frV)}$.
Since the morphism \ref{star2}, restricted to $\frV$, is a morphism of coherent $\sD^\dagger_{\frV,k}$-modules, we may test its bijectivity
on global sections according to proposition \ref{easy_thmAB}. Applying $\Ga(\frV,-)$ to \ref{star2} yields
\begin{numequation}\label{star3} \sD^{\dagger}_{k}\otimes_{\sD^{\dagger}_{k+1}} \sK_{k+1}(\frV) \lra \sK_k(\frV) \;,
\end{numequation}
where $\sD^{\dagger}_{k}=\sD^\dagger_{\frX,k} (\frV)$.
Indeed, this is clear in the case where the restriction of
$\sK_{k+1}$ to $\frV$ equals $\sD^{\dagger}_{\frV,k+1}$ and the general case follows from taking a finite presentation of it as $\sD^{\dagger}_{\frV,k+1}$-module.
However, $(\sK_k(\frV))$ forms a  $(\sD^{\dagger}_{k})$-sequence for the coadmissible module $\sK(\frV)$ and hence \ref{star3} is bijective. So \ref{star2} is bijective and $\sK\in\cC_\frX$.
This implies  $\varprojlim^{(1)}_k \sK_k= 0$ by the preceding theorem and hence  $\sI=  \varprojlim_{k} \sI_k$ where $\sI_k=\im(f_k)$. Again, we have a linear map
$$\sD^\dagger_{\frX,k} \otimes_{\sD^\dagger_{\frX,k+1}} \sI_{k+1}\lra \sI_k$$ whose bijectivity follows with the same argument as above. Hence $\sI\in\cC_\frX$ and this suffices to conclude that the category $\cC_\frX$ is abelian.
\end{proof}

We conclude with a result on how the category $\cC_{\frX}$ behaves under morphisms over $\frX_0$.

\begin{prop}\label{prop_invariance_for_D_infty} Let $\frX' \ra \frX_0$ be another admissible formal blow-up, and let $\pi: \frX' \ra \frX$ be a morphism over $\frX_0$. Then, for every coadmissible $\sD_{\frX',\infty}$-module $\sM$ the sheaf $\pi_* \sM$ is a coadmissible $\sD_{\frX,\infty}$-module via the isomorphism $\pi_* \sD_{\frX',\infty} = \sD_{\frX,\infty}$. Moreover, one has $R^q\pi_* \sM = 0$ for all $q>0$ and an equivalence of categories
 $$\pi_*: \cC_{\frX'}\stackrel{\simeq}{\longrightarrow} \cC_{\frX} \;.$$
\end{prop}

\begin{proof} Write $\sM = \varprojlim_k \sM_k$ with coherent $\sD^\dagger_{\frX',k}$-modules $\sM_k$. By \ref{prop-exactdirectimage} we know that each sheaf $\pi_*(\sM_k)$ is a coherent $\sD^\dagger_{\frX,k}$-module. Moreover $\pi_*(\sM) = \varprojlim_k \pi_*(\sM_k)$, and

$$\sD^\dagger_{\frX,k} \otimes_{\sD^\dagger_{\frX,k+1}} \pi_* \sM_{k+1} = \pi_* \sD^\dagger_{\frX',k} \otimes_{\pi_* \sD^\dagger_{\frX',k+1}} \pi_* \sM_{k+1} \simeq \pi_* \Big(\sD^\dagger_{\frX',k} \otimes_{\sD^\dagger_{\frX',k+1}} \sM_{k+1}\Big) \simeq \pi_* \sM_k \;.$$

\vskip8pt

Note here that the exact functor $\pi_*$, cf. \ref{prop-exactdirectimage}, indeed commutes with the tensor product: this may be checked over
an open affine $\frV\subseteq \frX'$ where we may take a finite presentation of the restriction of the module $\sM_{k+1}$ to $\frV$ as  $\sD^{\dagger}_{\frV,k+1}$-module to reduce to the case of the sheaf $\sD^{\dagger}_{\frV,k+1}$, as in the proof of the preceding theorem. This shows $\pi_* \sM\in \cC_{\frX}.$ We now write $\sM = \varprojlim_k \sM'_k$ with coherent $\hsD^{(k,0)}_{\frX,k}$-modules $\sM'_k$. Because of \cite[20.32.2]{stacks-project} the functors $R\pi_*$ and $R\varprojlim_k$ commute. By \ref{thm_B_for_frX_infty} we have $\sM = R\varprojlim_k \sM'_k$. Hence we obtain $R\pi_*(\sM) = R \pi_*(R\varprojlim_k \sM'_k) \simeq R\varprojlim_k R\pi_*(\sM'_k)$. By \ref{prop-exactdirectimage} we have $R \pi_*(\sM'_k) = \pi_*(\sM'_k)$, and so we obtain  a canonical isomorphism $R\pi_*(\sM) \simeq R\varprojlim_k \pi_*(\sM'_k)$. In the beginning of the proof we have seen that $\pi_* \sM = \varprojlim_k \pi_*(\sM'_k)$ is a coadmissible $\sD_{\frX,\infty}$-module, and so $R\varprojlim_k \pi_*(\sM'_k) = \pi_*(\sM)$, by \ref{thm_B_for_frX_infty}. From $R\pi_*(\sM) = \pi_*(\sM)$ we conclude that $R^q \pi_* \sM =0$ for all $q>0$.
\end{proof}

\subsection{Coadmissible \texorpdfstring{$\sD$}{}-modules on the Zariski-Riemann space \texorpdfstring{$\langle \frX_0 \rangle$}{}}

We finally explain how to pass the previous construction and results
to the projective limit in $\frX$, that is to say, to the Zariski-Riemann space of $\frX_0$.

\begin{para} Let $\cF_{\frX_0}$ be the set of all admissible formal blow-ups $\frX \ra \frX_0$.\footnote{We emphasize that the blow-up morphism $\frX \ra \frX_0$ is part of the datum of $\frX$.} This set is partially ordered by setting $\frX' \succeq \frX$ if the blow-up morphism $\frX' \ra \frX_0$ factors as $\frX' \stackrel{\pi}{\lra} \frX \ra \frX_0$, where  $\frX \ra \frX_0$ is the blow-up morphism. The morphism $\pi: \frX' \ra \frX$ is then uniquely determined by the universal property of blowing up, and is itself a blow-up morphism
\cite[ch. 8, 1.24]{LiuBook}, and we will denote it henceforth by $\pi_{\frX',\frX}$. By \cite[Remark 10 in sec. 8.2]{BoschLectures} the set $\cF_{\frX_0}$ is directed in the sense that any two elements have a common upper bound, and we can consider the topological space equal to the projective limit \footnote{In the paper \cite{HPSS} this space is denoted by $\frX_\infty$.}

$$\langle \frX_0 \rangle = \varprojlim_{\frX \in \cF_{\frX_0}} \frX. $$

\vskip8pt

This is the Zariski-Riemann space associated with $\frX_0$. For its basic properties we refer to \cite[9.3]{BoschLectures}.
\end{para}

\begin{para}{\it Sheaves on the space $\langle \frX_0 \rangle$.}  For $\frX \in \cF_{\frX_0}$ we denote the canonical projection map $\langle \frX_0 \rangle \ra \frX$ by $\sp_\frX$. If $\frX' \succeq \frX$ in $\cF_{\frX_0}$, we have $\sp_\frX = \pi_{\frX',\frX} \circ \sp_{\frX'}$. The isomorphism $(\pi_{\frX',\frX})_* \sD_{\frX',\infty} = \sD_{\frX,\infty}$ from \ref{can_iso_D_infty}, together with the adjunction map $\pi_{\frX',\frX}^{-1} \circ (\pi_{\frX',\frX})_* \ra {\rm id}$ gives rise to a canonical map

\begin{numequation}\label{transition_D_infty} \varphi_{\frX,\frX'}: \pi_{\frX',\frX}^{-1}\sD_{\frX,\infty} = \pi_{\frX',\frX}^{-1} (\pi_{\frX',\frX})_* \sD_{\frX',\infty} \lra \sD_{\frX',\infty} \;.
\end{numequation}

These morphisms of sheaves satisfy

$$\varphi_{\frX,\frX''} = \varphi_{\frX',\frX''} \circ \pi_{\frX'',\frX'}^{-1} \varphi_{\frX,\frX'}$$

whenever $\frX'' \succeq \frX' \succeq \frX$. We then obtain an inductive system $(\sp_\frX^{-1} \sD_{\frX,\infty})_{\frX \in \cF_{\frX_0}}$ of sheaves of rings on $\langle \frX_0 \rangle$, and we put

$$\sD_{\langle \frX_0 \rangle} = \varinjlim_\frX \sp_\frX^{-1} \sD_{\frX,\infty} \;.$$

\vskip8pt
\end{para}

\begin{dfn}\label{dfn_coadm_system} A $\sD_{\langle \frX_0 \rangle}$-module $\sM$ is called {\it coadmissible} if there is a family $(\sM_\frX,\psi^\sM_{\frX,\frX'})$ of coadmissible $\sD_{\frX,\infty}$-modules $\sM_\frX$, for all $\frX \in \cF_{\frX_0}$, together with an isomorphism

$$\psi^\sM_{\frX',\frX}: (\pi_{\frX',\frX})_*\sM_{\frX'} \stackrel{\simeq}{\lra} \sM_\frX \;,$$

\vskip8pt

of $\sD_{\frX,\infty}$-modules, whenever we have $\frX' \succeq \frX$ in $\cF_{\frX_0}$. This system of modules and isomorphisms is required to satisfy the following conditions:

\vskip8pt

\begin{enumerate}
\item  Whenever $\frX'' \succeq \frX' \succeq \frX$ in $\cF_{\frX_0}$ the following transitivity condition holds :$$\psi^\sM_{\frX',\frX} \circ (\pi_{\frX',\frX})_*(\psi^\sM_{\frX'',\frX'}) = \psi^\sM_{\frX'',\frX} \;.$$

\vskip8pt

\item $\sM$ is isomorphic to the inductive limit $\varinjlim_\frX  \sp_\frX^{-1} \sM_\frX$ as $\sD_{\langle \frX_0 \rangle} $-module.  \end{enumerate}
\end{dfn}

Note that the transition morphism $\sp_\frX^{-1} \sM_\frX \ra \sp_{\frX'}^{-1} \sM_{\frX'}$ in the inductive limit in (ii) is defined by applying the functor  $\sp_{\frX'}^{-1} $ to the morphism $$\pi_{\frX',\frX}^{-1}\sM_\frX \simeq \pi_{\frX',\frX}^{-1} (\pi_{\frX',\frX})_*\sM_{\frX'} \ra \sM_{\frX'} \;.$$ The latter is obtained
from $(\psi^\sM_{\frX',\frX})^{-1}: \sM_\frX \stackrel{\simeq}{\lra} (\pi_{\frX',\frX})_*\sM_{\frX'}$ and the adjunction map $\pi_{\frX',\frX}^{-1} \circ (\pi_{\frX',\frX})_* \ra {\rm id}$.

\vskip8pt

We denote by $$\cC_{\langle \frX_0 \rangle} \subseteq {\rm Mod}(\sD_{\langle \frX_0 \rangle})$$ the full subcategory of coadmissible
$\sD_{\langle \frX_0 \rangle}$-modules in the category of all $\sD_{\langle \frX_0 \rangle}$-modules.

\vskip8pt

\begin{prop}\label{prop-equivZar} Let $\frX\in\cF_{\frX_0}$. One has an equivalence of categories

 $$(\sp_\frX)_*: \cC_{\langle \frX_0 \rangle} \stackrel{\simeq}{\longrightarrow} \cC_{\frX} \;.$$

 \vskip5pt

 In particular, the category $\cC_{\langle \frX_0 \rangle} $ is abelian.
 \end{prop}
 \begin{proof} Let  $\sM = \varinjlim_\frX  \sp_\frX^{-1} \sM_\frX$ and let $\frX' \succeq \frX$. There is a canonical isomorphism

$$  (\sp_\frX)_* \sp_{\frX'}^{-1} \sM_{\frX'} \stackrel{\simeq}{\longrightarrow} \sM_{\frX} \;.$$
\vskip5pt
Indeed, let $\frU\subseteq \frX$ be an open and let $\frV= \pi_{\frX',\frX}^{-1} (\frU)$. Then

$$  (\sp_\frX)_* \sp_{\frX'}^{-1} \sM_{\frX'}(\frU) = \sp_{\frX'}^{-1} \sM_{\frX'} (\sp^{-1}_\frX(\frU))=  \sM_{\frX'} (\frV) $$

\vskip5pt
using that  $\sp_{\frX'}(\sp^{-1}_\frX(\frU))=\frV$. But  $\sM_{\frX'} (\frV) \simeq  \sM_{\frX} (\frU)$ via the map $\psi^\sM_{\frX',\frX}$.

In particular, we get an isomorphism

$$ (\sp_\frX)_*(\sM) =  \varinjlim_{\frX'} (\sp_\frX)_* \sp_{\frX'}^{-1} \sM_{\frX'} \stackrel{\simeq}{\longrightarrow} \sM_{\frX} \;.$$

\vskip5pt

This shows that the functor $(\sp_\frX)_*$ appearing in the proposition is well-defined. In the other direction, let $\sM\in  \cC_{\frX}$ and
define for  $\frX' \succeq \frX$ the module $\sM_{\frX'}$ by the requirement  $(\pi_{\frX',\frX})_*(\sM_{\frX'})\simeq \sM_{\frX}$ via \ref{prop_invariance_for_D_infty}. The family $(\sM_{\frX'})$ then satisfies the conditions (i) and (ii) in the above definition and its inductive limit $\sM$ lies therefore in  $\cC_{\langle \frX_0 \rangle}$.
This gives a quasi-inverse to the functor  $(\sp_\frX)_*$. By corollary \ref{abelian}, the category $\cC_{\langle \frX_0 \rangle} $ is then abelian.

 \end{proof}

Remark: It follows from the above proof that there is a canonical isomorphism of sheaves of rings
$$\Ga(\langle \frX \rangle, \sD_{\langle \frX_0 \rangle})\simeq  \Ga(\frX, \sD_{\frX,\infty})$$
for any $\frX\in\cF_{\frX_0}.$

\vskip5pt

\begin{thm}\label{thm_AB_for_frX_infty} {\rm (Theorem A and B for coadmissible $\sD_{\langle \frX_0 \rangle}$-modules) }

\begin{enumerate} \item Let $\frV\subseteq \frX$ be an open affine and let $\langle \frV \rangle=\sp^{-1}_\frX(\frV)$ be its Zariski-Riemann space. One has an equivalence of categories
$$ \Ga(\langle \frV \rangle,-): \cC_{\langle \frV \rangle}\stackrel{\simeq}{\lra} \cC_{\sD_{\frV,\infty}(\frV)} \;.$$

\vskip8pt

\item  For every $\sM\in  \cC_{\langle \frX_0 \rangle}$ and every $q>0$ one has $$H^q(\langle \frX_0 \rangle,\sM) = 0 \;.$$
\end{enumerate}
\end{thm}

\begin{proof} Part (i) follows from the preceding proposition together with theorem \ref{thm_A_for_frX_infty}. By \cite[0.3.1.19]{FujiwaraKato_I} the canonical map

$$\varinjlim_\frX H^q(\frX,\sM_\frX)\stackrel{\simeq}{\lra}  H^q(\langle \frX_0 \rangle,\sM)$$

is an isomorphism. Thus, part (ii) follows from \ref{thm_B_for_frX_infty}.
\end{proof}

\subsection{Examples}

The first example is given by the structure sheaf of the Zariski-Riemann space tensored with $\Q$. Let us denote
$$     \cO_{\langle \frX_0 \rangle,\Q}= \varinjlim_{\frX} \sp_{\frX}^{-1}\cO_{\frX,\Q} \;.$$
\begin{prop} The sheaf $\cO_{\langle \frX_0 \rangle,\Q}$ is a coadmissible $\sD_{\langle \frX_0 \rangle}$-module.
\end{prop}
\begin{proof} If $\frX,\frX'\in\cF_{\frX_0}$ and $\pi$ : $\frX'\ra \frX$ is a morphism over $\frX_0$, then $\pi_*\cO_{\frX',\Q}=\cO_{\frX,\Q}$ by ~\ref{eq_catO}. By
\ref{prop-equivZar} we have an equivalence of categories
$\cC_{\langle \frX_0 \rangle}\rightarrow \cC_{\frX_0}$ with an explicit quasi-inverse. From these considerations we see that the claim will follow from the fact that $\cO_{\frX_0,\Q}\in \cC_{\frX_0}$.

For any integer $k$, the sheaf $\cO_{\frX_0,\Q}$ is a $\sD^\dagger_{\frX_0,k}$-module. Let $\frU_0\subset \frX_0$ an affine
open of $\frX_0$ with coordinates $x_1, \ldots, x_M$, and corresponding derivations $\der_1,\ldots,\der_M$. Following
~\ref{ddagkX}, we write
$$D^{\dagger}_k=\Ga(\frU_0,\sD^\dagger_{\frX_0,k})=\left\{\sum_{\unu}\vpi^{k|\unu|}a_{\unu}\uder^{[ \unu ]}\,|\, a_{\unu}\in
\cO_{\frX_0,\Q}(\frU_0),
\textrm{ and } \exists C>0, \eta<1 \, | \, |a_{\unu}|< C \eta^{|\unu|} \right\} \;.$$
We have the following lemma, using the notation $\uzero=(0,\ldots,0).$
\begin{lemma} \label{div1} Let $P\in D^{\dagger}_k$, there exist $P_1,\ldots,P_M\in D^{\dagger}_k$ and $a_{\uzero}\in \cO_{\frX_0,\Q}(\frU_0)$ such that
              $$P=a_{\uzero}+\sum_{i=1}^M P_i \cdot \der_i \;.$$
\end{lemma}
        \begin{proof} The proof of this lemma
          is essentially the proof of the Spencer lemma by Berthelot ~\cite[3.2.1]{Berthelot_Trento}
          for the case $k=0$, meaning for the sheaf of arithmetic differential operators. Let us denote $\uone=(1,0,\ldots,0)\in
		\Ne^M$ and by $OP_1$ the set of operators in $D^{\dagger}_k$ such that $a_{\unu}=0$ if $\nu_1 \neq 0$.
		Let $P=\sum_{\unu}\vpi^{k|\unu|}a_{\unu}\uder^{[ \unu ]}\in D^{\dagger}_k$, and consider
		$$P_1=\sum_{\unu\, | \, \nu_1\neq 0}\vpi^{k|\unu|}\frac{a_{\unu}}{\nu_1}\uder^{[ \unu-\uone]} \;,$$ then,
		as $|1/\nu_1|_p=O(|\nu_1|_p)=O(|\unu|_p)$ when $|\unu|_p \ra +\infty$, this operator $P_1$
           belongs to $D^{\dagger}_k$ (here $|.|_p$ is the usual $p$-adic norm over the field $\Q$). Moreover since we have the
		identity
		$$ \nu_1 \uder^{[ \unu]} = \der_1 \cdot \uder^{[ \unu-\uone]} \;,$$ we get that
		$P = P_1 \der_1 + Q_1$ where $Q_1 \in OP_1$. We can now apply the same procedure to $Q_1$ relatively to $\der_2$.
		Doing this, we see that there exist
		 $P_2 \in D^{\dagger}_k$ and $Q_2$ with no terms containing
		neither $\der_1$, nor $\der_2$, such that $P=P_1 \der_1 + P_2 \der_2 + Q_2$. We finally find the lemma iterating $M$ times. Note
		that the $a_{\uzero}$ term given by the lemma is necessarily the same as the initial $a_{\uzero}$ term of $P$.
		 \end{proof}

This allows us to prove the
\begin{lemma} There is a presentation
 $$ \xymatrix @R-25pt {   {\sD^\dagger_{\frU_0,k}}^M \ar@{->}[r]^{\psi} & \sD^\dagger_{\frU_0,k}\ar@{->}[r]
 & \cO_{\frU_0,\Q}\ar@{->}[r] & 0 \\
                   (P_1, \ldots,P_M)\ar@{->}[r] & \sum_{i=1}^M P_i \der_i. &  & \\
      &    P \ar@{->}[r] & P \cdot 1 & .}$$
\end{lemma}
   \begin{proof} Let $\frV_0 \subset \frU_0$ be affine, and denote by $D^{\dagger}_k=\Ga(\frV_0,\sD^\dagger_{\frX_0,k})$,
         we have to prove that we have a presentation, with the same maps as in the statement
		 $$ \xymatrix @R-25pt {   {D^{\dagger}_k}^M \ar@{->}[r] & D^{\dagger}_k\ar@{->}[r]
		 & \cO_{\frX_0,\Q}(\frV_0)\ar@{->}[r] & 0 .}$$
		Let $P\in D^{\dagger}_k$, $P=\sum_{\unu}\vpi^{k|\unu|}a_{\unu}\uder^{[ \unu ]}$, such that $P(1)=a_{\uzero}=0$.
         By the previous lemma, there exist $P_1,\ldots, P_M$ such that $P=\sum_{i=1}^M P_i \der_i$ so that $P\in
		\im(\psi)$.
 \end{proof}

Let us come back now to the proof of the proposition.
Using this presentation, we see that $\cO_{\frX_0,\Q}$ is a coherent $\sD^\dagger_{\frX_0,k}$-module,
 and that we have canonical compatibility relations
$$   \sD^\dagger_{\frX_0,k}\ot_{\sD^\dagger_{\frX_0,k+1}} \cO_{\frX_0,\Q}\simeq \cO_{\frX_0,\Q} \;.$$
Finally the $\sD_{\frX_0,\infty}$-module $\cO_{\frX_0,\Q}$ is isomorphic to the constant projective system of
coherent $\sD^\dagger_{\frX_0,k}$-modules $(\cO_{\frX_0,\Q})$ and is an element of $\cC_{\frX_0}$. As explained at the beginning
of the proof, this implies that $\cO_{\langle \frX_0 \rangle,\Q}\in \cC_{\langle \frX_0 \rangle}$. This ends the proof
of the proposition.
\end{proof}

\vskip8pt

For the second example, we consider a Cartier divisor $\frZ$, which is assumed to be smooth over $\fro$, of the formal scheme $\frX_0$. As above, we denote by
$\frX_{0,\Q}$ and $\frZ_\Q$ the rigid analytic spaces associated with
$\frX_0$ and $\frZ$, respectively. Let $U =\frX_{0,\Q} \backslash \frZ_\Q$
be the open complement, and $j: U \ra \frX_{0,\Q}$ the inclusion of rigid spaces. We have the specialization map
 $\sp: \frX_{0,\Q} \ra \frX_0$.

\begin{prop}\label{coad2} The sheaf $\sp_{*}j_*\cO_U$ is a coadmissible $\sD_{\frX_0,\infty}$-module.
\end{prop}

\begin{proof} We freely use the notation and terminology of \cite[4.0.1]{Berthelot_Trento}.
Let us consider $$V_k=\frX_{0,\Q}\,\backslash \,]\frZ[_{|\varpi|^k}$$
and $\cV_s (V_k)$ the set of strict neighborhoods of $V_k$. Note that $V_k$ is well defined since $\frZ$ is a Cartier divisor of $\frX_0$.
If $\frV=\Spf A$ is an open affine subset of $ \frX_0$, such
that $\frZ\bigcap \frV=V(t_1)$, then $$ V_k\bigcap \frV_\Q=\left\{ x\in \frV_\Q \, | \, |t_1(x)| \geq |\varpi|^k\right\} \;.$$
Note that $$ U=\bigcup_k V_k \textrm{ and } U=\bigcup_k \bigcup_{W\in \cV_s (V_k)} W \;.$$
We introduce also
        $$\cE^{\dagger}_k =\varinjlim_{W \in \cV_s (V_k)}\sp_{*}j_{W*}\cO_W \;,$$
where $j_W$ : $W \ra \frX_{0,\Q}$ is the inclusion in $\frX_{0,\Q}$ of a strict neighborhood $W$ of $V_k$.
We have an inclusion
$V_{k}\subset V_{k+1}$ and $V_{k+1}$ is a strict neighborhood of $V_{k}$, so that $\cV_s (V_{k+1})\subset \cV_s (V_k)$.
As a consequence, for any $k$, there is a canonical
morphism $\cE^{\dagger}_{k+1}\ra \cE^{\dagger}_{k}$. Moreover, since $\sp_*$ commutes with projective limits, we have
$$  \sp_{*}j_*\cO_{U_L}=\varprojlim_k \cE^{\dagger}_{k} \;.$$

The proposition will follow from the
\begin{lemma}\label{lemEk} \begin{enumerate}\item The sheaf $\cE^{\dagger}_k$ is a coherent $\sD^{\dagger}_{\frX_0,k}$-module.
              \item The canonical map $\cE^{\dagger}_{k+1}\ra \cE^{\dagger}_k$ induces a canonical isomorphism
                 of coherent $\sD^{\dagger}_{\frX_0,k}$-modules,
                 $$\sD^{\dagger}_{\frX_0,k}\ot_{\sD^{\dagger}_{\frX_0,k+1}}\cE^{\dagger}_{k+1}\simeq \cE^{\dagger}_k \;.$$
               \end{enumerate}
\end{lemma}
{\it Proof.} Let $W$ be admissible open in $\frX_L$. Then $W$ is the generic fiber of some Zariski open $\frW'$ of $\frX'$
where $\pr: \frX'\ra \frX_0$ is an admissible blow-up of $\frX_0$. Denote by $j'$ the inclusion : $\frW' \hra \frX'$. Then $j'_* \cO_{\frW',\Q}$ is a $\sD^{\dagger}_{\frX',k'}$-module, for $k'\geq k_{\frX'}$, so that the sheaf
$\sp_* j_{W*}\cO_W=\pr_*j'_* \cO_{\frW',\Q}$  has an action of $\sD^{\dagger}_{\frX_0,k'}$ as
$\sD^{\dagger}_{\frX_0,k'}=\pr_*\sD^{\dagger}_{\frX',k'}$ by ~\ref{prop-exactdirectimage}. In particular,
the sheaf $\sp_* j_{W*}\cO_W$ is a
$\sD_{\frX_0,\infty}$-module for any admissible open $W$. As a consequence, the sheaf $\cE^{\dagger}_k$ has a structure of
$\sD_{\frX_0,\infty}$-module as well.
Let us check locally that this structure extends to a structure of $\sD^{\dagger}_{\frX_0,k}$-module. Let $\frV=\Spf A\subset
\frX_0 $ be affine open in $\frX_0$ such that $\frZ \bigcap \frV=V(t_1)$ where $t_1$ is a local coordinate on
$\frV$. Then we have the following description, where $A_L=A\ot L$ is an affinoid algebra,
$$\cE^{\dagger}_k(\frV)=\left\{\sum_{\nu \geq 0}a_\nu \varpi^{k\nu}t_1^{-\nu-1}, a_{\nu}\in A_L \, | \,
             \exists C>0, \eta <1  \, | \,|a_\nu| < C \eta^{\nu}\right\} \;.$$
Denote by $\der_1$ the derivation corresponding to the coordinate $t_1$, and $\der_2,\ldots,\der_M$ the other
derivations. Let us denote $D^{\dagger}_k=\sD^\dagger_{\frX_0,k}(\frV)$, then we have the following description using ~\ref{ddagkX}
$$ D^{\dagger}_k=
\left\{\sum_{\unu}a_{\unu}\vpi^{k|\unu|}\uder^{[ \unu ]}\,|\, a_{\unu}\in A_L,
\textrm{ and } \exists C>0, \eta<1 \, | \, |a_{\unu}|< C \eta^{|\unu|} \right\} \;.$$
To prove the lemma it is thus enough to check the
\begin{lemma} \label{lem2ex2}There is a presentation
$$ \xymatrix @R-25pt {   {D^{\dagger}_k}^M \ar@{->}[r]^{\psi} & D^{\dagger}_k\ar@{->}[r]^{\varphi}
 & \cE^{\dagger}_k(\frV)\ar@{->}[r] & 0 \\
                   (P_1, \ldots,P_M)\ar@{->}[r] & P_1 \der_1 t_1+ \sum_{i=2}^M P_i \der_i & &  \\
                        &  P\ar@{->}[r] & P\cdot\frac{1}{t_1} \;.&  }$$
\end{lemma}
{\it Proof.} Again, we follow the proof by Berthelot ~\cite[4.4.2]{Berthelot_Trento} of the analogous statement for
arithmetic
differential operators. It is clear that $\im(\psi)\subset \ker(\varphi)$. Observe that
$$ \der_1^{[\nu_1]}\cdot t_1^{-1}=(-1)^{\nu_1}t_1^{-\nu_1-1} \;.$$
Let $h=\sum a_\nu \varpi^{k\nu}t_1^{-\nu-1} \in \cE^{\dagger}_k(\frV)$, then $P=\sum_{\nu}(-1)^{\nu}\varpi^{k\nu}a_\nu
\der_1^{[\nu]}$ belongs to $D^{\dagger}_k$ and $P(1/t_1)=h$, so that $\varphi$ is surjective. Let now $P\in
\ker(\varphi)$, then, applying repeatedly lemma~\ref{div1}, we see that modulo $\im(\psi)$, $P $ can be written
$P=\sum_{\nu} \varpi^{k\nu}a_\nu\der_1^{[\nu]} $, with coefficients $a_{\nu}\in A_L$ such that $\sum_{\nu} \varpi^{k\nu}(-1)^{\nu}a_\nu
(t_1)^{-\nu-1}=0 \in \cE^{\dagger}_k(\frV)$.
Moreover, there exist $C>0, \eta<1$ such that $|a_{\nu}|<C\eta^{\nu}$ where $|\cdot |$ is a Banach norm
on $A_L$. Denote $$ b_j=\sum_{\nu=0}^j a_{\nu} \varpi^{k\nu} t_1^{j-\nu} \;.$$
Let us now state the following
\begin{aux}\label{lem3ex2} There exist $C'>0, \eta'<1$ such that $|b_{j}|<C'|\varpi|^{kj}{\eta'}^{j}$.
\end{aux}
\begin{proof} Berthelot proved this lemma for $k=0$ in ~\cite[4.2.1]{Berthelot_Trento}.
 Let us check that the proof can be adapted to any $k$. Since $\frV$
is smooth, the affinoid algebra $A_L$ is reduced, so that the spectral semi-norm is a norm and defines the Banach topology on $A_L$. All Banach norms being equivalent, we can use this norm to prove the statement, which we keep on denoting by
$|\cdot|$. Let $\frV_\Q=\Spm A_L$ be the
generic fiber of $\frV$, seen as rigid analytic space. Let $\eta'>\eta$, such that $\eta'<1$ and some power of $\eta'$
lies in the valuation group of $L$. We consider the following admissible cover of $\frV_\Q$ by open $V_1$ and $V_2$
defined by
$$ V_1=\left\{ x\in \frV_\Q\, | \, |t_1(x)|\leq |\varpi|^k \eta'\right\} \textrm{ and }
            V_2=\left\{ x\in \frV_\Q\, |\,|\varpi|^k \eta' \leq |t_1(x)|\leq 1\right\} \;.$$
It is enough to bound the spectral norm of the $b_j$ on each of this affinoid open.
As $A_L$ is reduced, $\Ga(V_1,\cO_{\frX_L})$, resp. $\Ga(V_2,\cO_{\frX_L})$,
 is reduced as well by Corollary 10
of~\cite[7.3.2]{BGR}, so that the spectral norm induced a norm on these two affinoid open sets. If $x\in V_1$, then
$$ |b_j(x)| = |\sum_{\nu=0}^j \varpi^{k\nu} a_{\nu} t_1^{j-\nu}| \leq C |\varpi|^{kj}{\eta'}^j \;.$$
Consider $$B=\left\{\sum_{\nu \geq 0}a_\nu \varpi^{k\nu}t_1^{-\nu-1}, a_{\nu}\in A_L \, | \,
             |a_{\nu}|\ra 0 \right\} \;,$$
which is the Banach algebra of analytic functions on the affinoid $\{x \in \frV_\Q \, | \, |\varpi|^{k} \leq |t_1(x)| \leq
1 \}$. Obviously, $\cE^{\dagger}_k(\frV)\subset B$.
If $x\in V_2$, then $$|\varpi^{k\nu}a_{\nu}t_1^{-\nu}(x)|\leq C \left(\frac{\eta}{\eta'}\right)^{\nu} \;,$$
so that the series $\sum_{\nu} \varpi^{k\nu}(-1)^{\nu}a_\nu (t_1)^{-\nu-1}$ converges to some element
$b\in \Ga(V_2,\cO_{\frX_L})$. Moreover the image of this element $b$ in $B$ is zero. The support of $b$ is a closed
affinoid subset of $V_2$, and at each point $x$ of this support, $|t_1(x)|<|\varpi^{k}|$. By the maximum principle,
increasing $\eta'$ if necessary, provided that $\eta'<1$,
 we can assume that the support is contained in $\{x \, |\, |t_1(x)|\leq \,|\varpi|^{k}
(2\eta'-1) \}$,
so that $h=0$ restricted to $V_2$. Then we have the following upper bound for $x\in V_2$
$$ |b_j(x)|=|\sum_{\nu \geq j+1}\varpi^{k\nu}a_{\nu}(t_1(x))^{j-\nu}|\leq C |\varpi|^{kj} {\eta'}^j  \;.$$
\end{proof}
Let us come back to the proof of ~\ref{lem2ex2}. We need to check that $P=\sum_{\nu} \varpi^{k\nu}a_\nu\der_1^{[\nu]}$,
such that $\sum_{\nu} \varpi^{k\nu}(-1)^{\nu}a_\nu (t_1)^{-\nu-1}=0 \in \cE^{\dagger}_k(\frV)$, belongs to $\im(\psi)$.
Let us define
$$ b_j=(-1)^{j+1}\sum_{\nu=0}^{j-1} (-1)^{\nu}a_{\nu} \varpi^{k\nu} t_1^{j-\nu-1} \;,$$
and $$Q=\sum_{j\geq 0}b_j \der_1^{[j]} \;,$$
that belongs to $D^{\dagger}_k$ thanks to \ref{lem3ex2}. Berthelot checked at the end of the proof of
~\cite[4.2.1]{Berthelot_Trento}, that $P=Qt_1$. But since this is true in $D^{\dagger}_0$, this is also true in
$D^{\dagger}_k$ and $P=Qt_1$. By hypothesis, $b_0=0$, and by the lemma ~\ref{div1}, this implies that there exists $Q_1\in D^{\dagger}_k$
such that $Q=Q_1\der_1$. We finally conclude that $P=Q_1 \der_1 t_1$ and that $P\in \im(\psi)$.

This presentation proves (i) and (ii) of the lemma~\ref{lemEk} and completes the proof of ~\ref{coad2}.

\end{proof}


\end{document}